\documentclass[reqno,a4paper]{amsart}

\usepackage{mathrsfs,amsmath,amssymb,graphicx}
\usepackage{stmaryrd}
\usepackage[unicode]{hyperref}
\usepackage[english]{babel} 
\usepackage{enumitem}
\usepackage{bbm} 

\usepackage{subcaption}
\usepackage[percent]{overpic} 
\usepackage{transparent}

\usepackage{tikz}

\usetikzlibrary{calc}
\usetikzlibrary{matrix,arrows,decorations}

\theoremstyle{plain}
\newtheorem{mtheorem}{Main Theorem}
\newtheorem{theorem}{Theorem}[section]
\newtheorem{lemma}[theorem]{Lemma}
\newtheorem{corollary}[theorem]{Corollary} 
\theoremstyle{definition}
\newtheorem{definition}{Definition}[section]
\theoremstyle{remark}
\newtheorem{remark}{Remark}[section]

\newcommand{\sphere}{S^3}  
\newcommand{\Compl}[1]{E(#1)}
\newcommand{\rnbhd}[1]{\mathit N(#1)} 
\newcommand{\ver}{\mathtt{v}}
\newcommand{\edg}{\mathtt{e}}
\newcommand{\fac}{\mathtt{f}}
\newcommand{\mcell}{\mathtt{m}} 
\newcommand{\unloopedv}{\mathtt{u}_v}
\newcommand{\notlooped}{\mathtt{u}_v} 
\newcommand{\doubled}{\mathtt{d}}
\newcommand{\unloopede}{\mathtt{u}_e}
\newcommand{\uncovered}{\mathtt{u}_e}
\newcommand{\hopf}{\mathrm{H}}

\newcommand{\sym}[2][\sphere]{\mathcal{MCG}(#1, #2)}
\newcommand{\mcg}[1]{\mathcal{MCG}(#1)}
\newcommand{\pmcg}[1]{\mathcal{MCG}_+(#1)}
\newcommand{\psym}[2][\sphere]{\mathcal{MCG}_+(#1, #2)}

\newcommand{\Aut}[1]{{\mathcal Homeo}(#1)}
\newcommand{\pAut}[1]{{\mathcal Homeo_+}(#1)}
 
\newcommand{\id}{\mathrm{id}}

\newcommand{\smi}{{\scalebox{0.5}[.6]{\(- \)}}}
\newcommand{\spl}{{\scalebox{0.5}[.6]{\(+\)}}} 
\newcommand{\bm}{\mathbf m} 
\newcommand{\bs}{\mathbf s}   
\newcommand{\nada}[1]   {}
\newcommand{\cout}[1]   {}

\definecolor{mygray}{rgb}{0.92,0.92,0.92}

\newcommand{\draftMMM}[1]{\ifdraft{\color{blue}#1}\fi}
\newcommand{\draftGGG}[1]{\ifdraft{\color{red}#1}\fi}
\newcommand{\draftYYY}[1]{\ifdraft{\color{orange}#1}\fi}

\numberwithin{equation}{section}

\newif\ifdraft
\draftfalse

\title{Hopf handlebody-links: their symmetry and classification}
\author{Yi-Sheng Wang}
\address{Department of Applied Mathematics, National Sun Yat-sen University}
\email{yisheng@math.nsysu.edu.tw}

\date{\today}
\subjclass{primary 57K12, secondary 57K30, 57M15}

\begin{document}

\thanks{Y.-S. W. gratefully acknowledges the support from NSTC, Taiwan (grant no. 112-2115-M-110 -001 -MY3; 115-2115-M-110 -002 -MY3)}

\begin{abstract}	
We generalize the notion of the Hopf link to the context of spatial graphs and handlebody-links, and study their symmetry and classification problems.     
\end{abstract}

\maketitle

\draftMMM{DraftMMM comments by Maurizio}
\draftYYY{DraftYYY comments by Yi-Sheng}
\draftGGG{DraftGGG comments by Giovanni B}
 
\section{Introduction}\label{sec:intro}

The Hopf link, the simplest non-trivial link, is one of the most studied links in knot theory, and constantly appears in other fields, such as electromagnetism, molecular biology and polymer physics. Many linked objects have local structures that resemble a Hopf link, and they are often the simplest and most ubiquitous entanglements. 

\begin{figure}[h]
	\begin{subfigure}{.24\linewidth}
		\includegraphics[scale=.1]{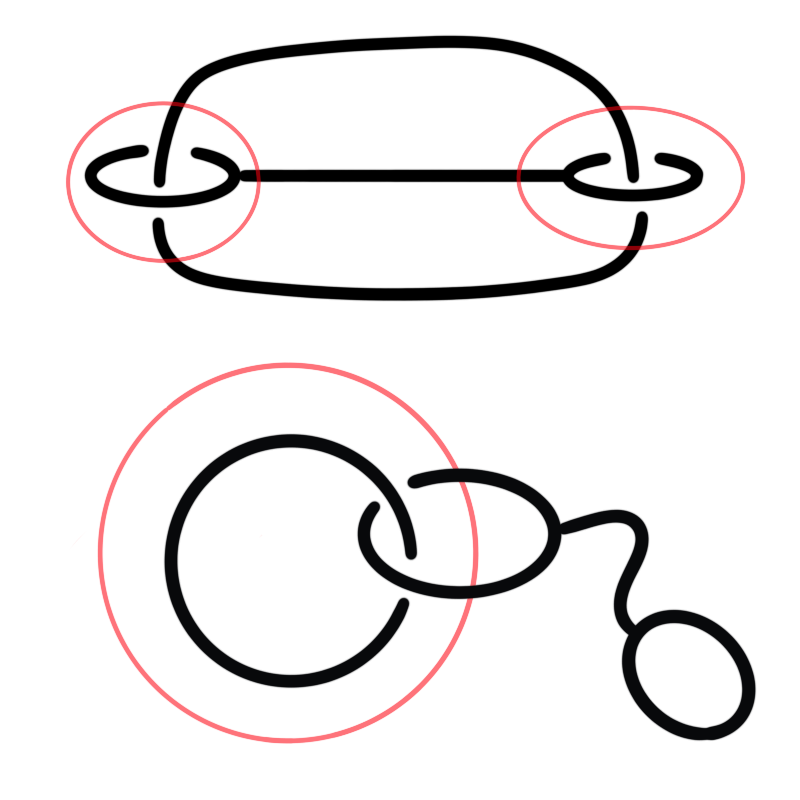}
		\caption{}
		\label{fig:equivalent}
	\end{subfigure} 
	\begin{subfigure}{.24\linewidth}
		\includegraphics[scale=.1]{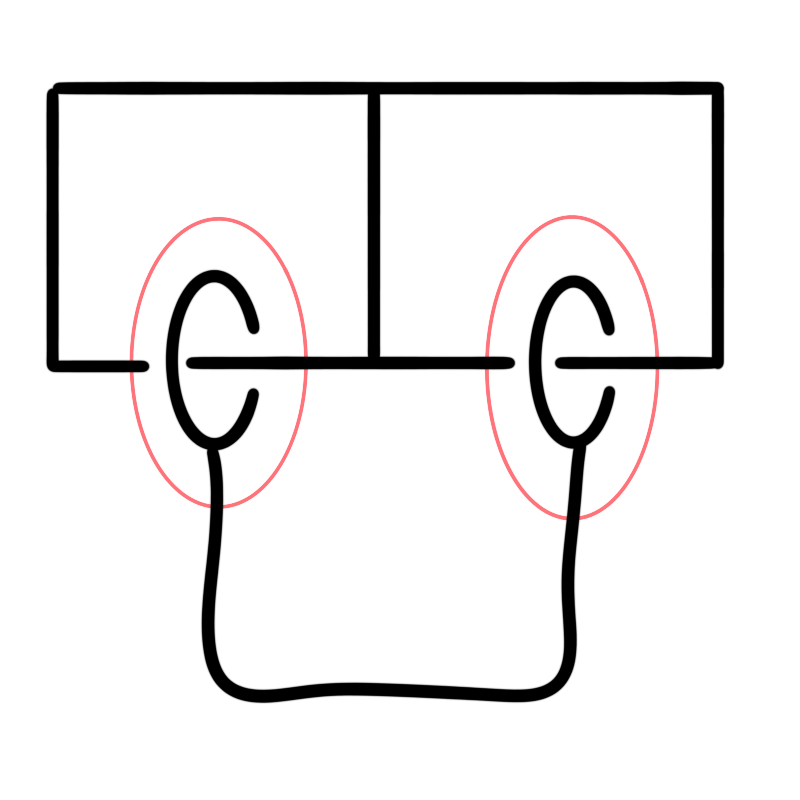}	
		\caption{}
		\label{fig:nonsimple} 
	\end{subfigure} 
	\begin{subfigure}{.24\linewidth}
		\includegraphics[scale=.1]{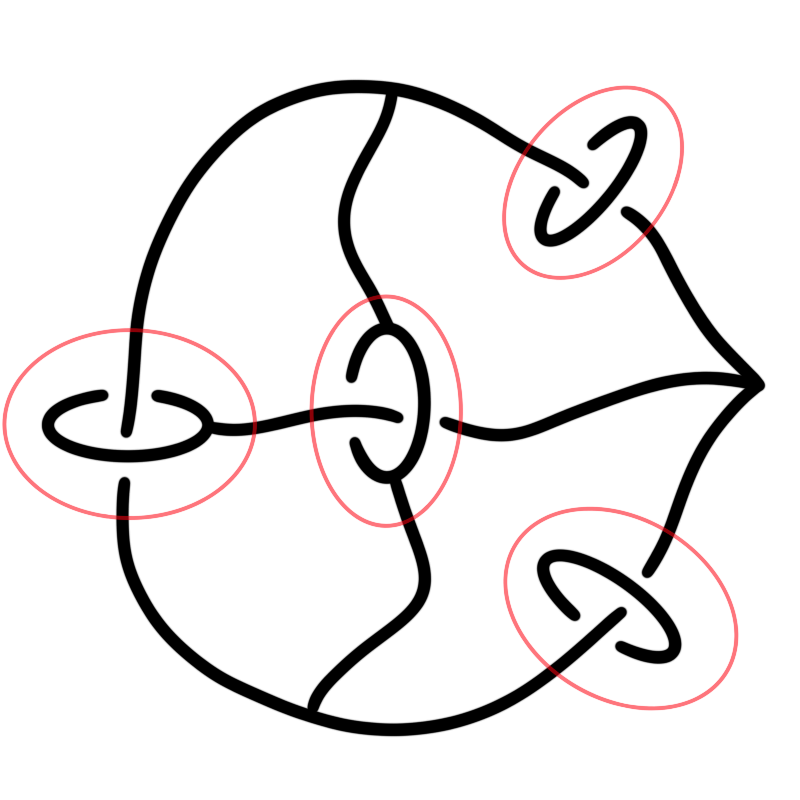}	 
		\caption{}
		\label{fig:hopf_summand}
	\end{subfigure} 
	\begin{subfigure}{.24\linewidth}
		\includegraphics[scale=.1]{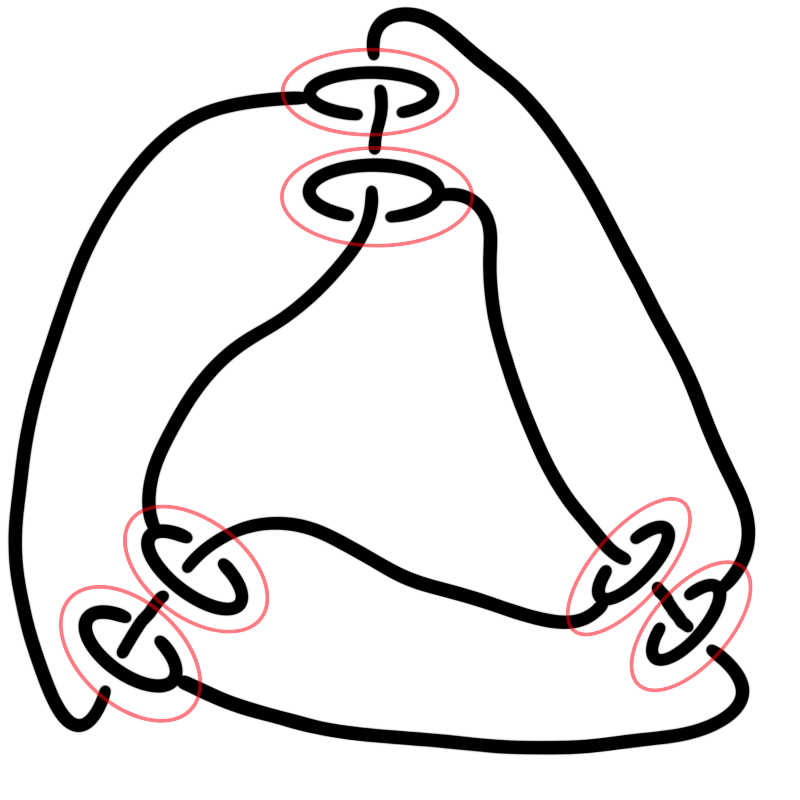}
		\caption{}
		\label{fig:dihedral_three}	 
	\end{subfigure} 
	\caption{Gallery of entanglements.}
	\label{fig:examples}
\end{figure}

%
We introduce here a generalization of the Hopf link in the context of spatial graphs and handlebody-links. 
The notion of generalized Hopf links is inclusive enough to cover many examples of interest, particularly those simplest spatial graphs and handlebody-links in Moriuchi \cite{Mor:07}, Ishii-Kishimoto-Moriuchi-Suzuki \cite{IshKisMorSuz:12}, and Bellettini-Paolini-Paolini-Wang \cite{BelPaoPaoWan:23},
yet at the same time, is sufficiently constrained for a systematic study. It is our hope that the investigation into this class of handlebody-links may prompt further research into handlebody-links of higher genus, an emerging yet underexplored field. 
     
\subsection*{Spatial graphs and handlebody-links}
A \emph{spatial graph} $\Gamma$ is a finite, \emph{not necessarily connected}, graph embedded in the oriented $3$-sphere $S^3$; 
we assume that $\Gamma$ has no $1$-valent vertex. 
A \emph{handlebody-link} $V$ is a finite disjoint union of handlebodies embedded in $S^3$.   
Two spatial graphs $\Gamma,\Gamma'$ (resp.\ two handlebody-links $V,V'$) are \emph{equivalent} ($\simeq$) 
if there is an
ambient isotopy $f_t$ with $f_1(\Gamma)=\Gamma'$ (resp.\ $f_1(V)=V'$). 
%
While a handlebody-link is often represented by its spine, which is itself a spatial graph, a handlebody-link may have infinitely many spines that are inequivalent as spatial graphs. For instance, the two spatial graphs in Fig.\ \ref{fig:equivalent} are spines of equivalent handlebody-links. 
Contrariwise, up to equivalence, a regular neighborhood of a spatial graph $\Gamma$ determines a unique handlebody-link, called the \emph{induced handlebody-link} by $\Gamma$. A spatial graph is \emph{trivial} if it is contained in a $2$-sphere in $S^3$, and a handlebody-link is \emph{trivial} if it admits a trivial spine. A spatial graph $\Gamma$ (resp.\ handlebody-link $V$) is \emph{non-split} if every $2$-sphere $S\subset S^3$ disjoint from $\Gamma$ (resp.\ $V$) bounds a $3$-ball $B$ disjoint from $\Gamma$ (resp.\ $V$) in $S^3$. 
We concern primarily non-split spatial graphs and handlebody-links.

The \emph{Euler characteristic} $\chi$ of a spatial graph $\Gamma$ or a handlebody-link $V$ is defined to be the Euler characteristic of $\Gamma$ or $V$. A \emph{handlebody-knot} is a handlebody-link with only one component, and its genus $g$ is the genus of the component; in particular, we have $\chi=1-g$.
Since the spine of a solid torus is unique, up to ambient isotopy, the theory of classical links, namely, spatial graphs with $\chi=0$, is equivalent to the study of handlebody-links with $\chi=0$. For $\chi<0$ however, handlebody-links and spatial graphs exhibit fundamentally different behaviors. 
 

\subsection*{Irreducibility} 
A \emph{$1$-decomposing sphere} 
of a spatial graph $\Gamma$ (resp.\ handlebody-link $V$) is a $2$-sphere $S\subset S^3$ meeting $\Gamma$ at one point transversally (resp.\ meeting $V$ at an essential disk). A non-split spatial graph or handlebody-link is \emph{irreducible} if it admits no $1$-decomposing sphere; see Suzuki \cite{Suz:70} and Taniyama \cite{Tan:02}. A reducible handlebody-link $V$ has a $\partial$-reducible exterior $\Compl V:=\overline{S^3-V}$. Having a $\partial$-reducible exterior, however, does \emph{not} imply the handlebody-link is reducible; see Suzuki \cite[Example $5.5$]{Suz:75} for a connected example or Bellettini-Paolini-Paolini-Wang \cite[Remark $3.3$]{BelPaoPaoWan:23} for a non-connected one.

Similarly, if a handlebody-link admits a reducible spine, then it is reducible, yet the converse is not true. Indeed, every reducible handlebody-link admits irreducible spines, for example, Fig.\ \ref{fig:equivalent}. This makes detecting handlebody-link irreducibility often challenging; for instance, Fig.\ \ref{fig:hopf_summand} represents a reducible handlebody-link, though it may not appear so.
The notion of reducibility is generalized to (weakly) $n$-decomposability in Ishii-Kishimoto-Ozawa \cite{IshKisOza:15} for handlebody-knots.

\subsection*{Symmetry}
The symmetry of a spatial graph $\Gamma$ (resp.\ handlebody-link $V$) can be measured by the \emph{symmetry group} defined as the mapping class group:
\[\sym{X}:=\pi_0(\Aut{S^3,X}),\]
where $\Aut{S^3,X}$ is the topological group of self-homeomorphisms of the pair $(S^3,X)$, $X=\Gamma$ or $V$. Considering the subgroup $\pAut{S^3,X}$ of \emph{orientation-preserving} self-homeomorphisms of $(S^3,X)$, we obtain the \emph{positive symmetry group} 
\[ \psym{X}:= \pi_0(\pAut{S^3,X}).\]
Note that $X$ is equivalent to its mirror image, namely, \emph{amphichiral}, if and only if the positive symmetry group is a proper subgroup of the symmetry group. 
 
The (positive) symmetry group of a classical link has been intensively studied over the past century, and its structure has been determined for many classes of links. It is now known that the symmetry group of a knot is finite if and only if its is cyclic or dihedral, if and only if the knot is a torus or hyperbolic knot or a cable of a tours knot; see Kawauchi \cite{Kaw:96}. The study of the symmetry group for general spatial graphs, by contrast, is relatively recent. It has drawn considerable attention in the past three decades, due partly to its application to molecular chemistry. Many results on spatial graph symmetry have since been proved; see Simon \cite{Sim:86} and Flapen-Naimi-Pommersheim-Tamvakis \cite{FlaNaiPomTam:05}, for instance. Also, a characterization for a connected $\Gamma$ to have a finite symmetry group is obtained in Cho-Koda \cite{ChoKod:13}. 

On the other hand, the highly deformable nature of handlebodies makes the handlebody-link symmetry rather unpredictable and difficult to classify. For instance, the symmetry groups of the spatial graphs in Figs.\ \ref{fig:nonsimple} and \ref{fig:hopf_summand} are finite, yet the symmetry groups of the handlebody-links they represent are both infinite. Indeed, it is an open question whether the symmetry group of a handlebody-link is finitely generated. Notably, it remains unsolved the Powell conjecture, which implies, if $V$ is a trivial handlebody-knot, then its symmetry group, known as the Goeritz group \cite{Goe:33}, is finitely generated; see Freedman-Scharlemann \cite{FreSch:18}, Cho-Koda-Lee \cite{ChoKodLee:24}. Even the chirality of handlebody-knots up to six crossings is not completely determined yet; see Ishii-Iwakiri-Jang-Oshiro \cite{IshIwaJanOsh:13}.

The past two decades has seen, however, significant progress in the genus two case. Koda \cite{Kod:15} proves the finite generation of the symmetry group of a reducible genus two handlebody-knot; Funayoshi-Koda \cite{FunKod:20} shows that a genus two handlebody-knot has a finite symmetry group if and only if it is atoroidal and non-trivial. Based on the Koda-Ozawa annulus classification \cite{KodOzaGor:15}, symmetry groups of several classes of genus two handlebody-knots are computed in Koda \cite{Kod:15}, Wang \cite{Wan:21, Wan:23, Wan:24} and Koda-Ozawa-Wang \cite{KodOzaWan:25}.
The symmetry groups of some $3$-decomposable handlebody-knots with genus $g>2$, namely $\chi<-1$, are determined in Bellettini-Paolini-Wang \cite{BelPaoWan:25}, where handlebody-knots mostly have large crossing numbers. Little is known about the symmetry of non-connected handlebody-links with $\chi\leq -1$. 

The \emph{Nielsen realization problem} for a handlebody-link $V$ asks whether every finite subgroup $H<\mcg{\sphere,V}$ 
can act on $(\sphere,V)$, and we say $V$ is \emph{Nielsen realizable} if it is the case. By the geometrization, a Nielsen realizable $V$ can be isotoped into a position where $H$ acts as isometries of $S^3$. Classical knots with a finite symmetry group are Nielsen realizable, but there are Nielsen unrealizable knots; see Sakuma \cite{Sak:89}. All genus two handlebody-knots whose symmetry groups have been determined are Nielsen realizable,
but they account for a tiny fraction of handlebody-knots, and little is known about the Nielsen realizability for handlebody-links with $\chi<-1$.

\subsection*{Main contributions}
\begin{enumerate}[label=(\roman*)]
\item The notion of \emph{Hopf handlebody-links} in Section \ref{subsec:looping} that generalizes the Hopf link and includes known examples of interest with small crossing numbers.
\item  A graph-theoretic criterion for a Hopf handlebody-link to be irreducible (Main Theorem \ref{teo:irreducibility}).

\item\label{itm:three}  
The notion of a \emph{simple} Hopf handlebody-link, via the JSJ-decomposition \cite{JacSha:79}, \cite{Joh:79}, \cite{Joh:95}, which guarantees the finiteness of the symmetry group. 
A graph-theoretic criterion for simpleness (Main Theorems \ref{teo:simpleness}, \ref{teo:classification}\ref{itm:finiteness}). 

\item Symmetry groups of simple Hopf handlebody-links classified, and their Neilsen realizability proved affirmatively (Main Theorem \ref{teo:classification}). 
. 

\item An enumeration of simple Hopf handlebody-\emph{knots} with $\chi\geq -4$, up to equivalence, without duplication in Table \ref{tab:hopf_handlebody_knots}.  
\end{enumerate}

\cout{
%

The second contribution (Main Theorem \ref{teo:irreducibility}) proves a graph-theoretic criterion for a Hopf handlebody-link to be irreducible, whereas the third (Main Theorem \ref{teo:simpleness}) gives a condition, based on the Jaco-Shalen-Johansson decomposition, that guarantees the finiteness of the symmetry group for Hopf handlebody-links. Their group structures are classified in the fourth result (Main Theorem \ref{teo:classification}). To the author's knowledge, symmetry groups of many handlebody-links with $\chi<-1$ are obtained for the first time there.

The JSJ-decomposition allows us to define the notion of \emph{simple Hopf handlebody-links}, and Main Theorem \ref{teo:enumeration} enumerates all simple Hopf handlebody-knots with $\chi\geq -4$, up to equivalence; see Table \ref{tab:hopf_handlebody_knots}. 
} 
 
\subsection*{Conventions}
We work in the piecewise linear category.
Given a subpolyhedron $X$ of a $3$-manifold $M$, $\overline{X}$, 
$\mathring{X}$, $\rnbhd X$, and 
$\partial_f X$ denote the closure, the interior, a regular neighborhood, and 
the frontier of $X$ in $M$, respectively.
The \emph{exterior} $\Compl X$ of $X$ in $M$ is defined to be $\overline{M-\rnbhd{X}}$ 
if $X\subset M$ is of positive codimension, and to be $\overline{M-X}$ otherwise. 
By $\vert X\vert$, we understand the number of components in $X$. Submanifolds are assumed to be proper and in general position.

A surface $S$ other than a disk and sphere in a $3$-manifold $M$ is \emph{essential} if it is incompressible, $\partial$-incompressible, and non-$\partial$-parallel. A disk or sphere in $M$ is \emph{essential} if it does not cut off from $M$ a $3$-ball.

\begin{table}[t]
	{\includegraphics[scale=.08]{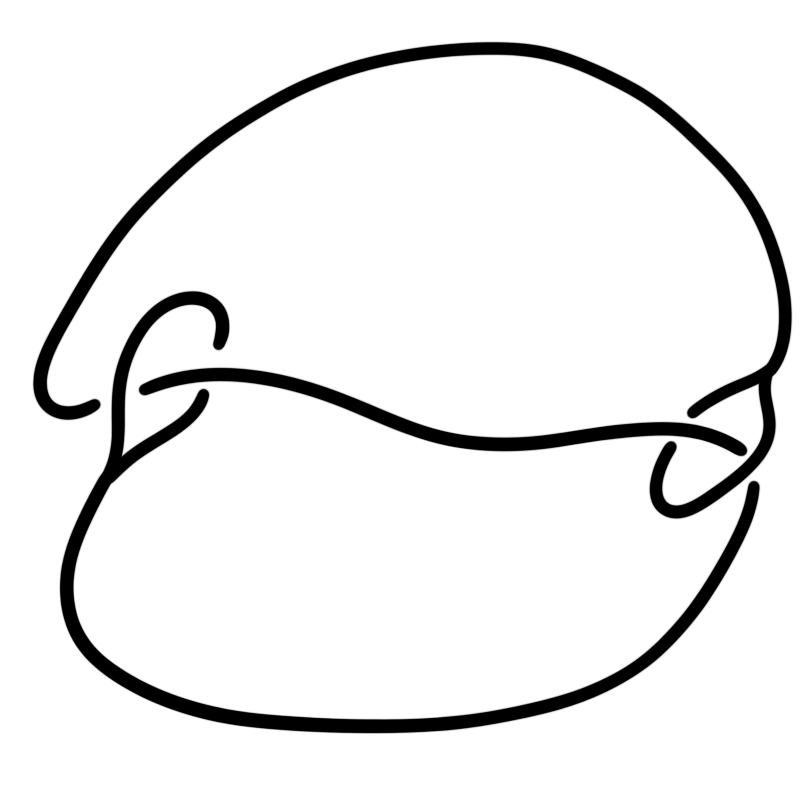}$1\mathrm{H}_1$}
	{\includegraphics[scale=.08]{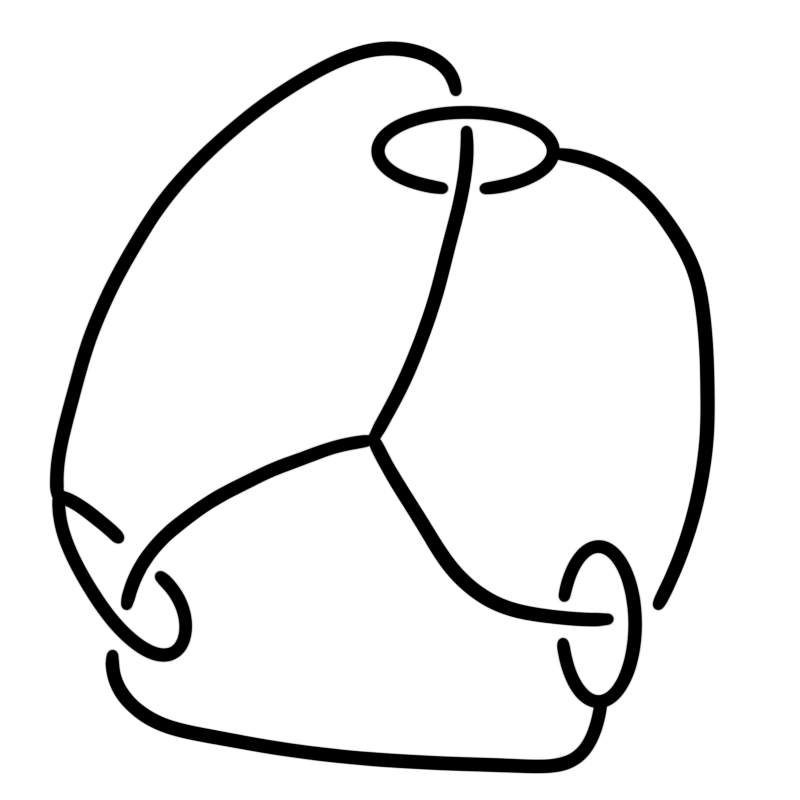}$2\mathrm{H}_1$}
	{\includegraphics[scale=.08]{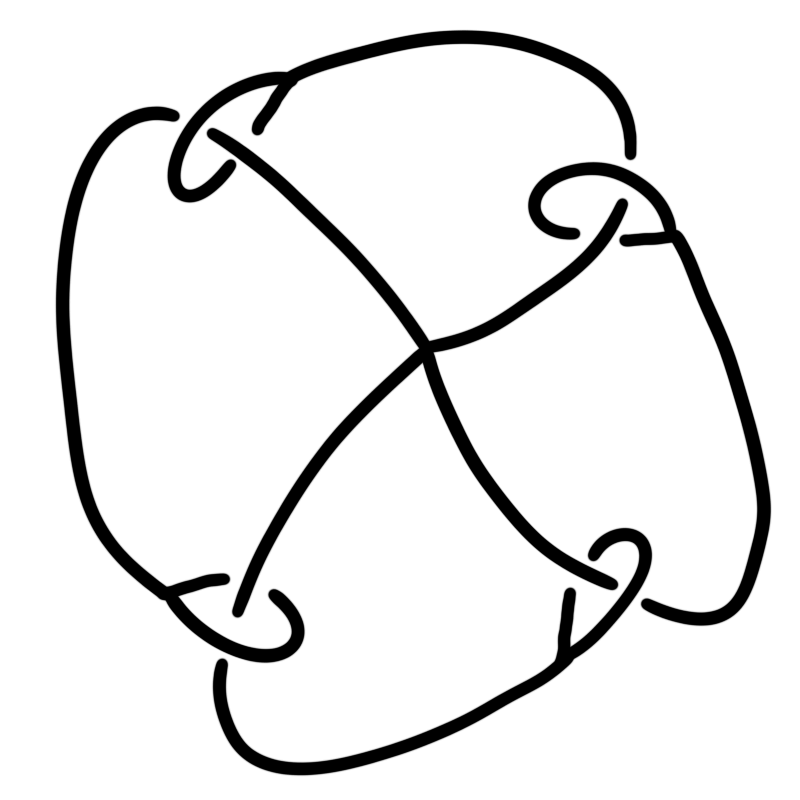}$3\mathrm{H}_1$}
	{\includegraphics[scale=.08]{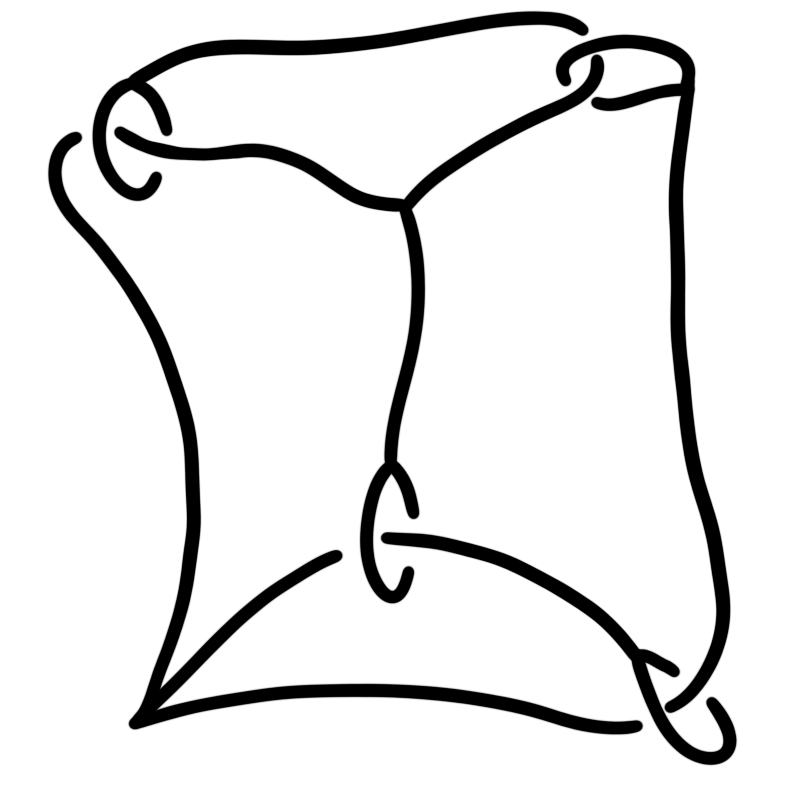}$3\mathrm{H}_2$} 
	{\includegraphics[scale=.08]{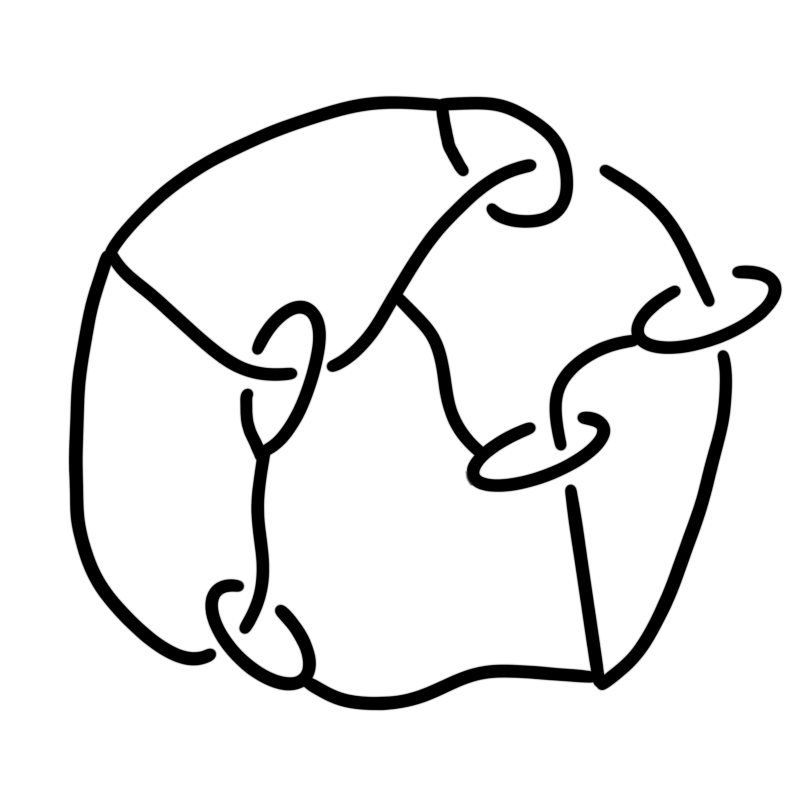}$4\mathrm{H}_1$}
	{\includegraphics[scale=.08]{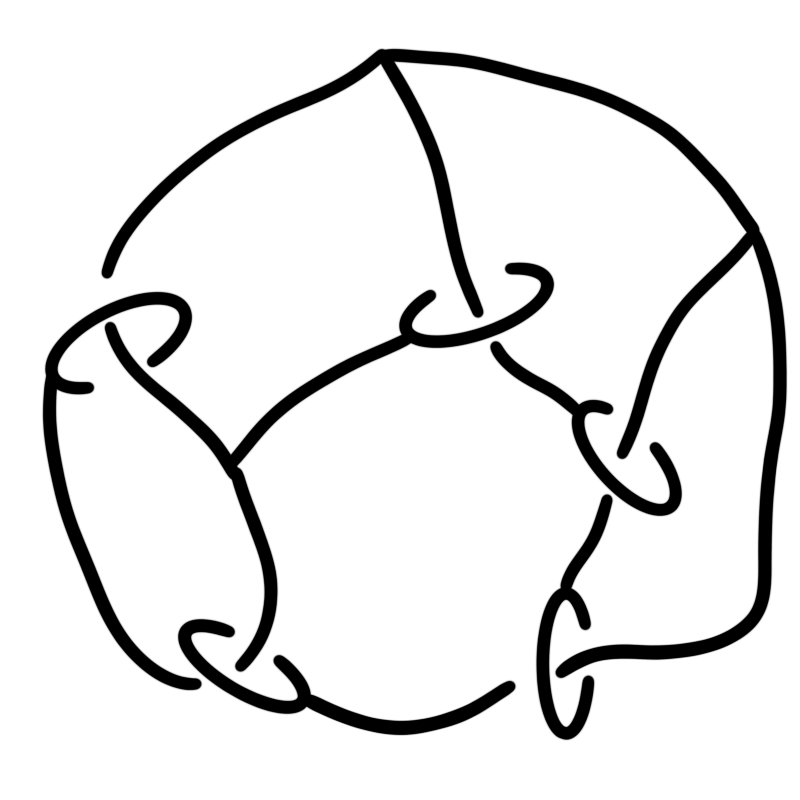}$4\mathrm{H}_2$}
	{\includegraphics[scale=.08]{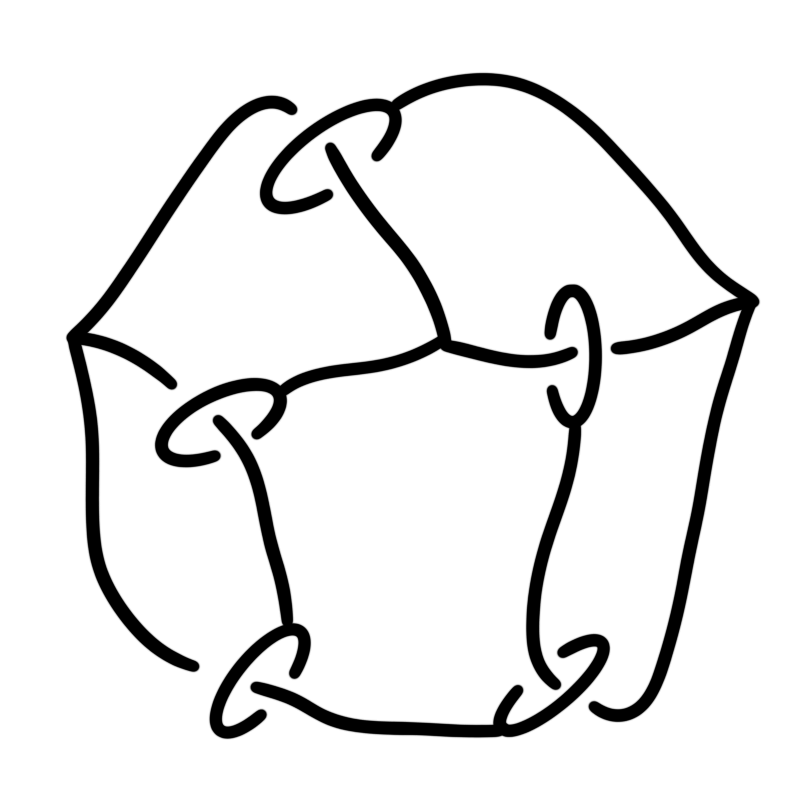}$4\mathrm{H}_3$}
	{\includegraphics[scale=.08]{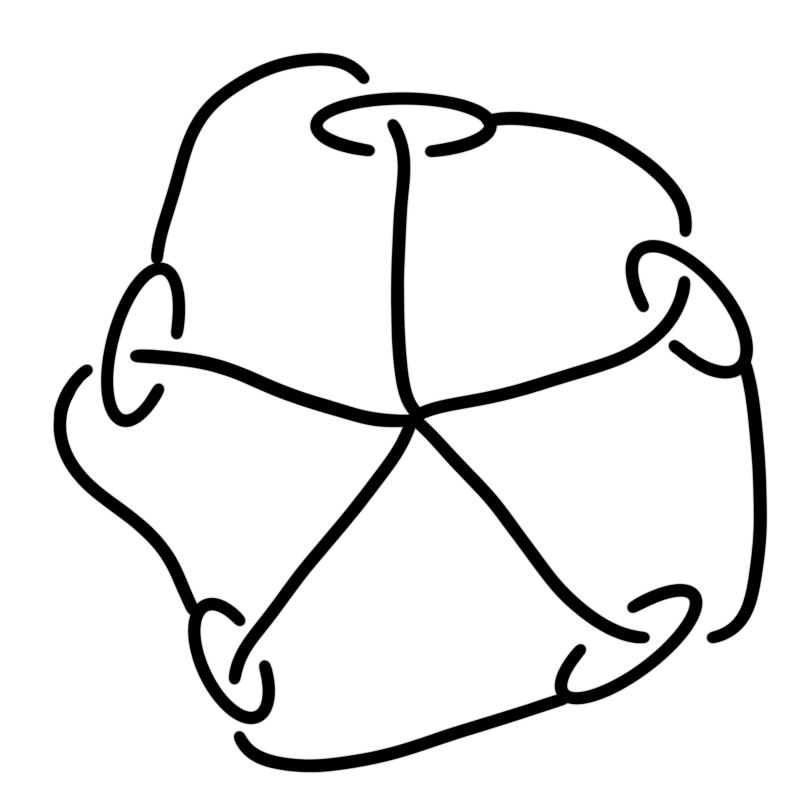}$4\mathrm{H}_4$}
	{\includegraphics[scale=.08]{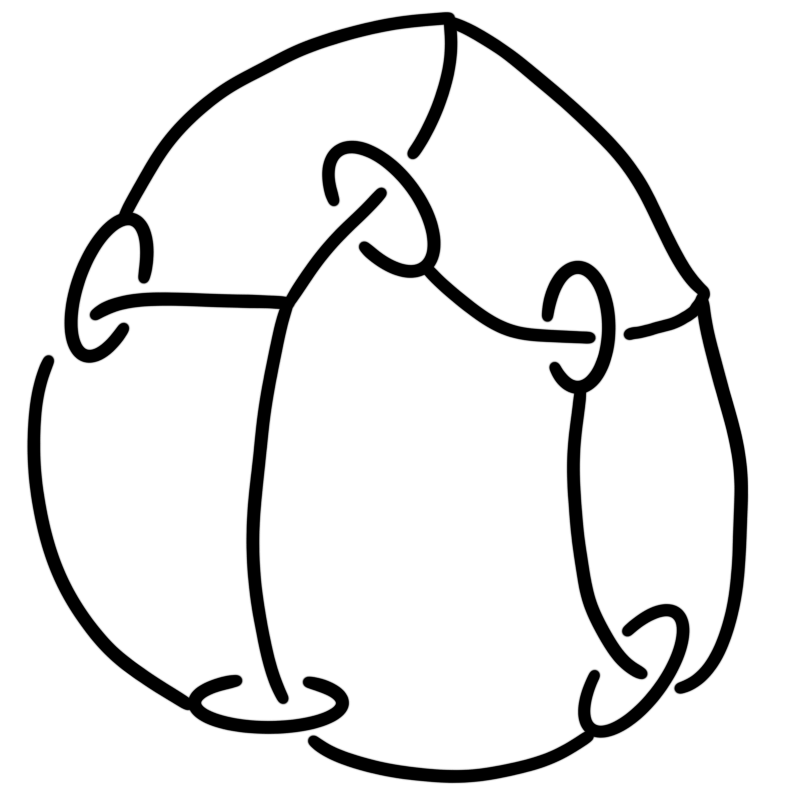}$4\mathrm{H}_5$}
	{\includegraphics[scale=.08]{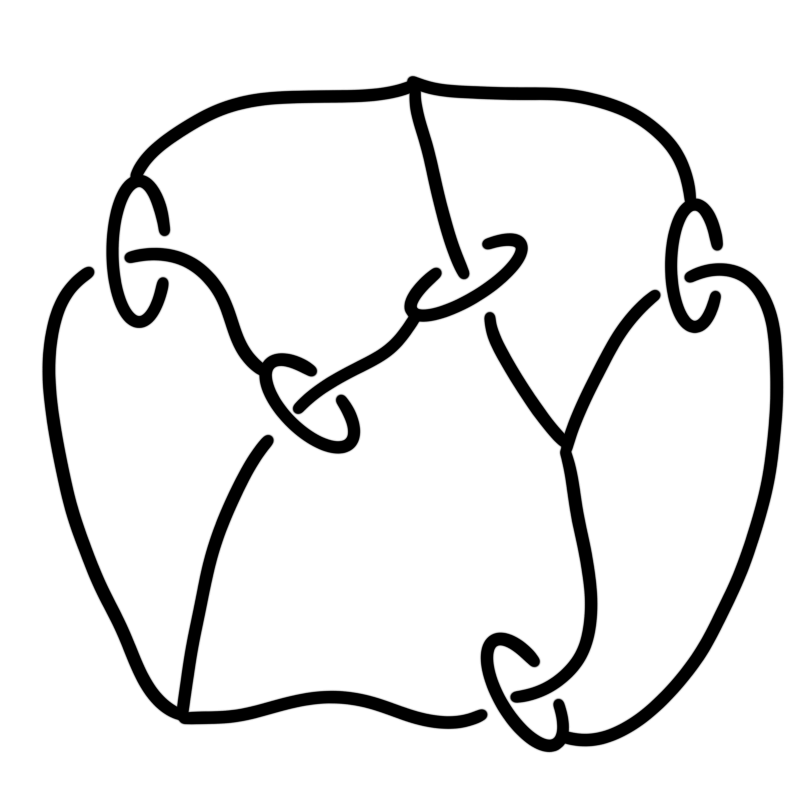}$4\mathrm{H}_6$}
	{\includegraphics[scale=.08]{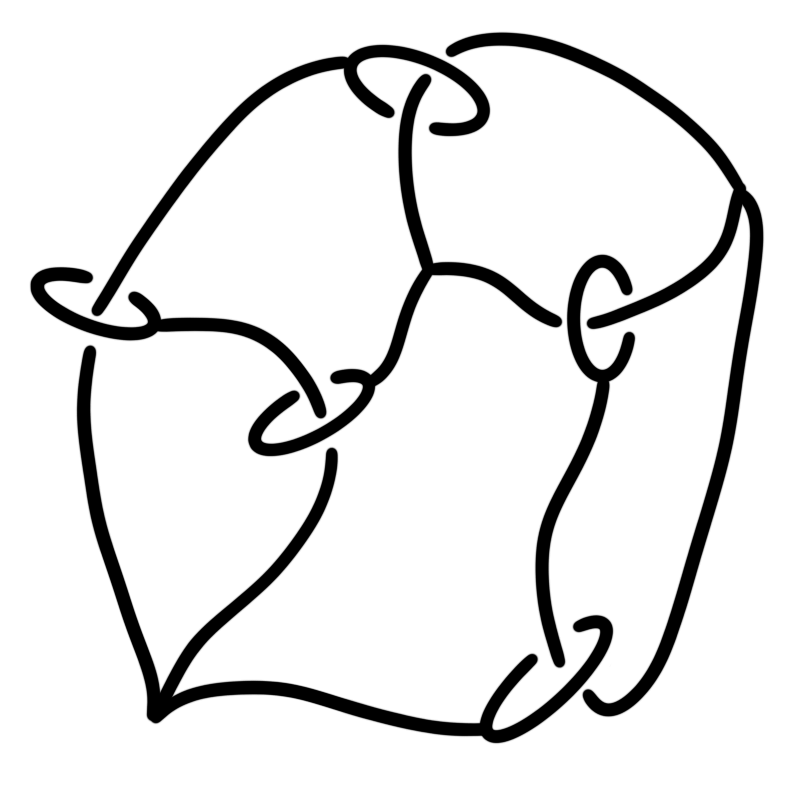}$4\mathrm{H}_7$}
	{\includegraphics[scale=.08]{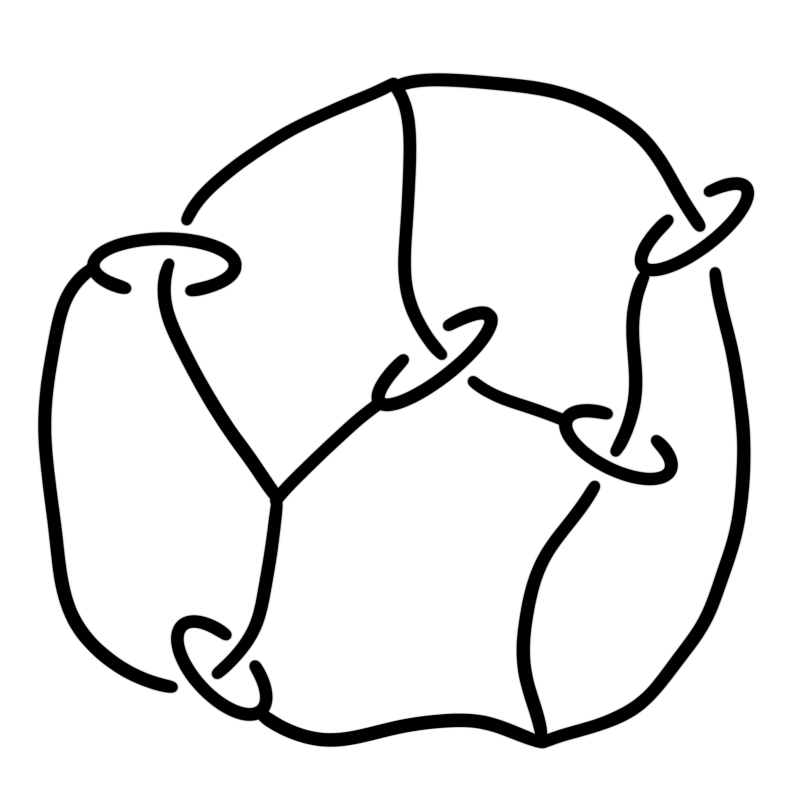}$4\mathrm{H}_8$}
	{\includegraphics[scale=.08]{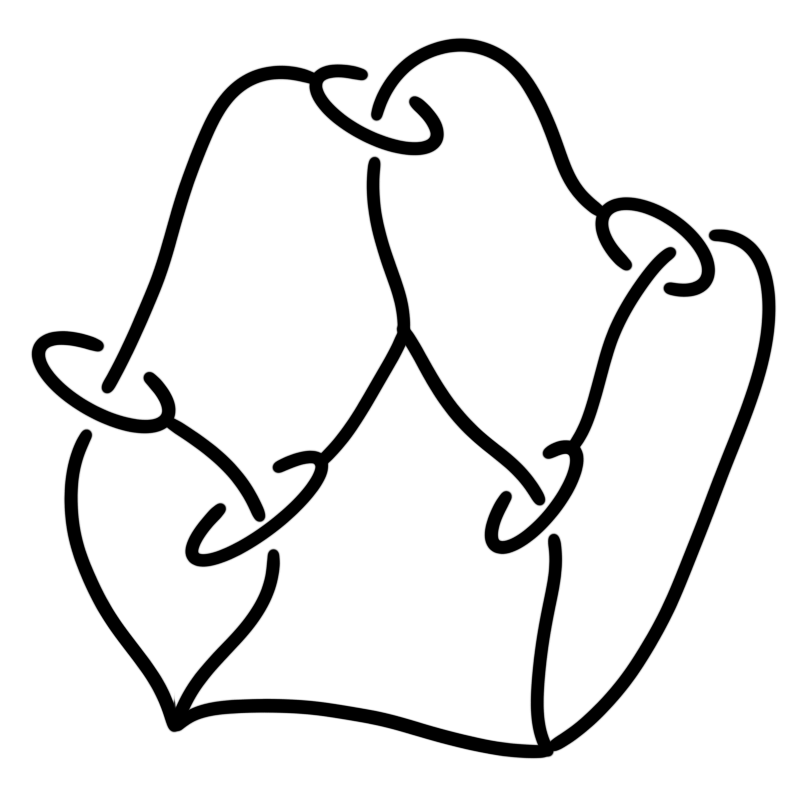}$4\mathrm{H}_9$}
	{\includegraphics[scale=.08]{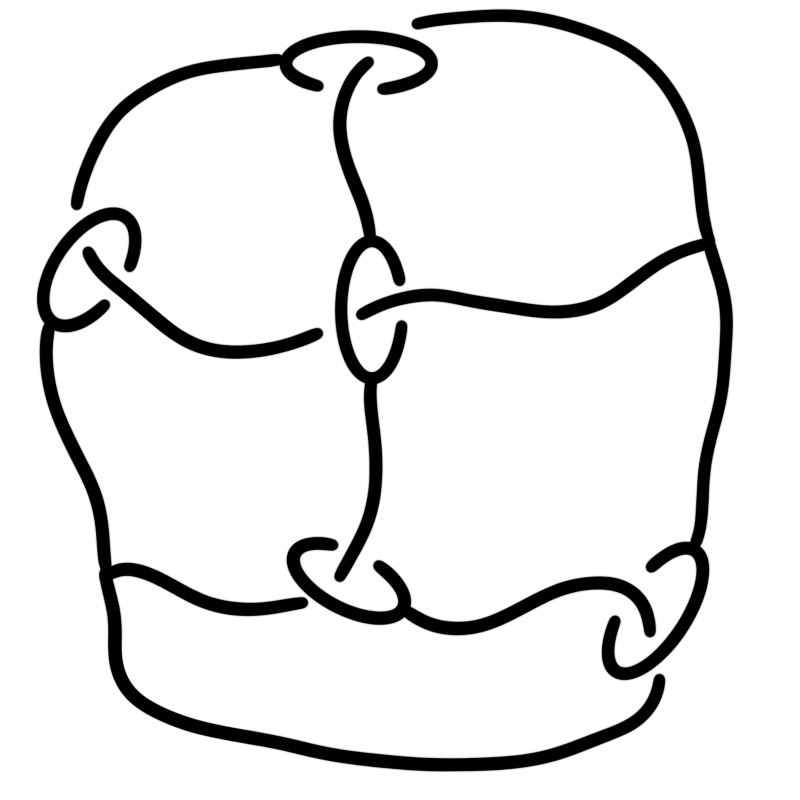}$4\mathrm{H}_{10}$}
	{\includegraphics[scale=.08]{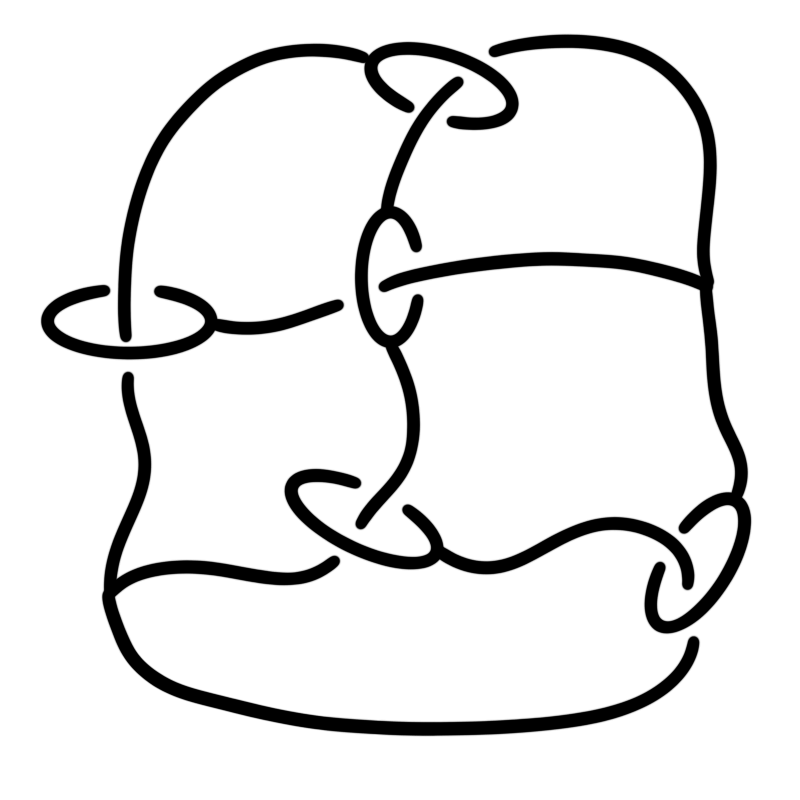}$4\mathrm{H}_{11}$}
	{\includegraphics[scale=.08]{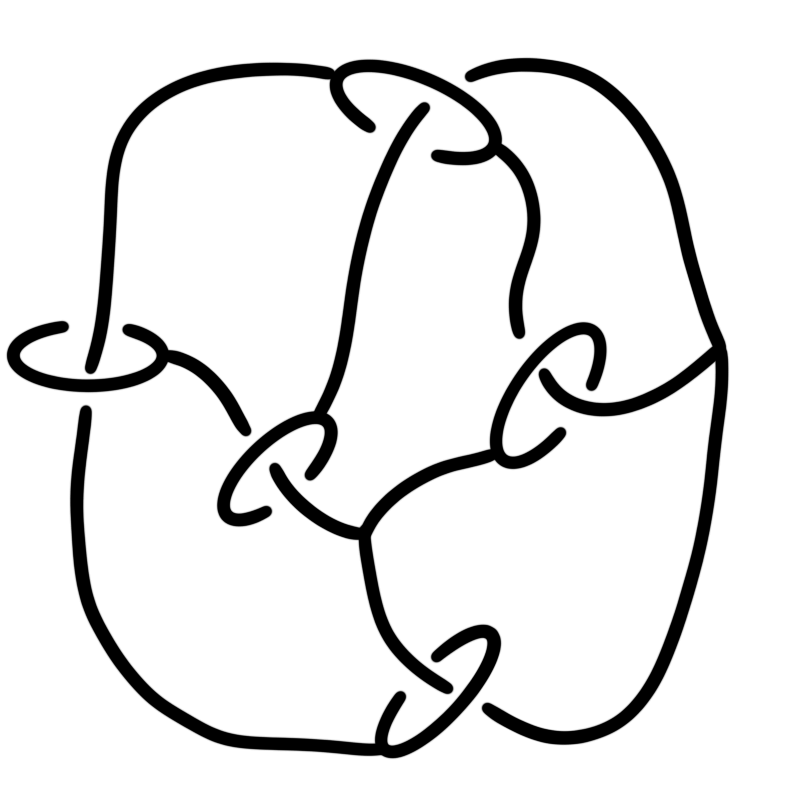}$4\mathrm{H}_{12}$}
	{\includegraphics[scale=.08]{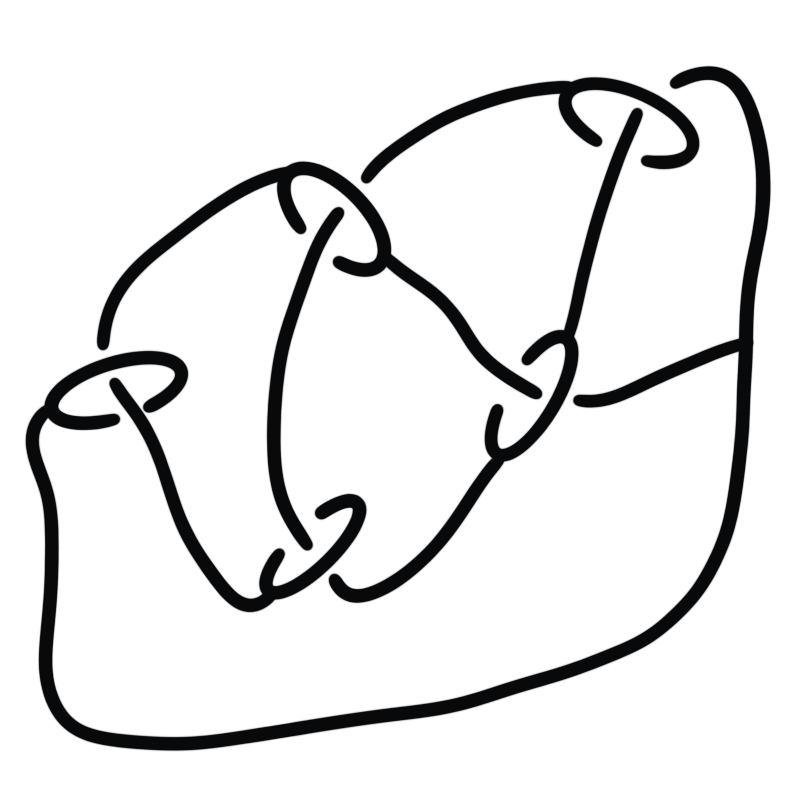}$4\mathrm{H}_{13}$}
	\caption{Hopf hanldboedy-knot table.}
	\label{tab:hopf_handlebody_knots}
\end{table}

\section{Looping}\label{sec:looping}
\begin{figure}[b]
	\begin{subfigure}{.22\linewidth}		
		\centering
		\begin{overpic}[scale=.1,percent]{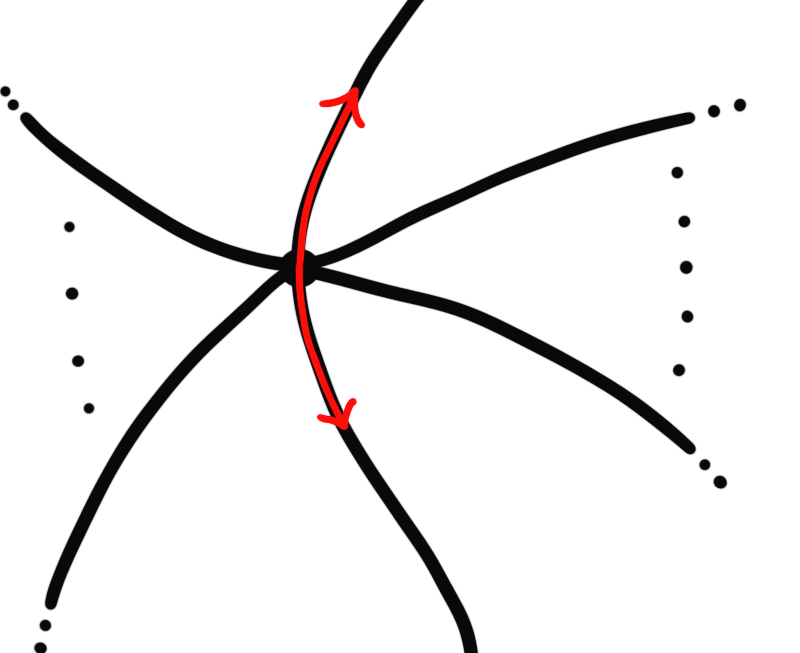}
			\put(31,57){\footnotesize $v$}
			\put(31,70){\footnotesize $e_1$}
			\put(30,33){\footnotesize $e_2$}
		\end{overpic}
		\caption{$\epsilon=\{v,e_1,e_2\}$.}
		\label{fig:convention_triplet}
	\end{subfigure} 
	\begin{subfigure}{.22\linewidth}		
		\centering
		\begin{overpic}[scale=.09,percent]{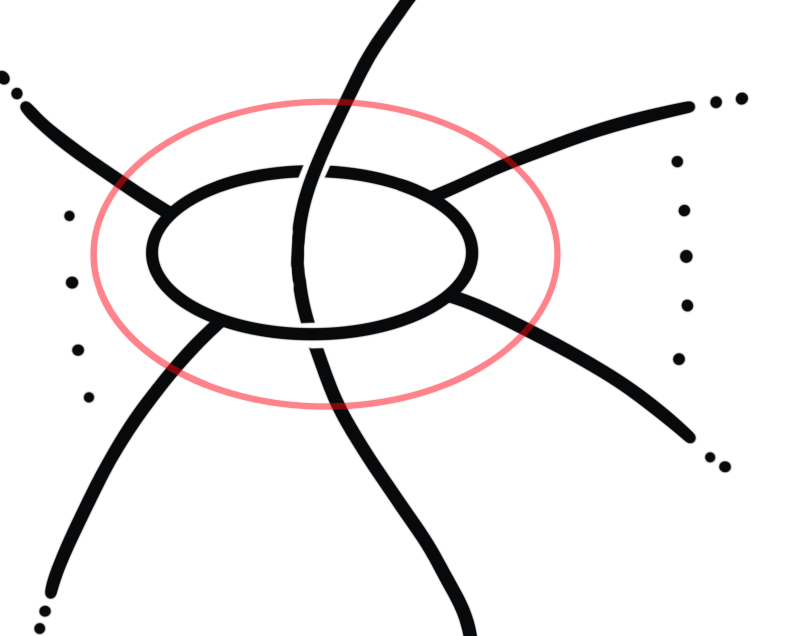}
		\end{overpic}
		\caption{Looping at $\epsilon$.}
		\label{fig:convention_loop}
	\end{subfigure} 
	\begin{subfigure}{.24\linewidth}		
		\centering
		\begin{overpic}[scale=.09,percent]{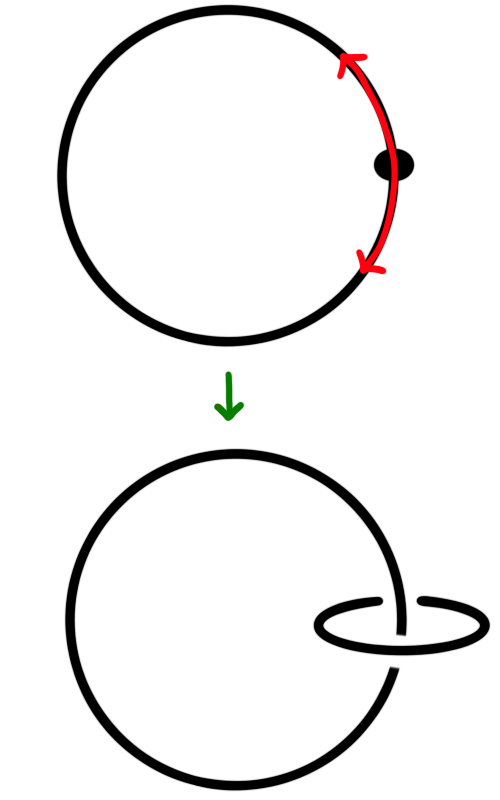}
		\end{overpic}
		\caption{Hopf link.}
		\label{fig:hopf}
	\end{subfigure} 
	\begin{subfigure}{.24\linewidth}		
		\centering
		\begin{overpic}[scale=.09,percent]{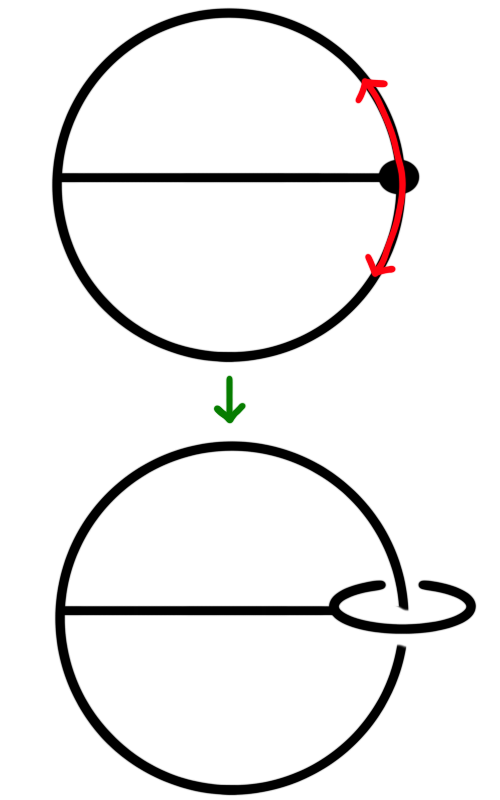}
		\end{overpic}
		\caption{Spatial handcuff.}
		\label{fig:twoone}
	\end{subfigure}
	\caption{} 
\end{figure}
\begin{figure}[b] 	
	\begin{subfigure}{.42\linewidth}		
		\centering
		\begin{overpic}[scale=.11,percent]{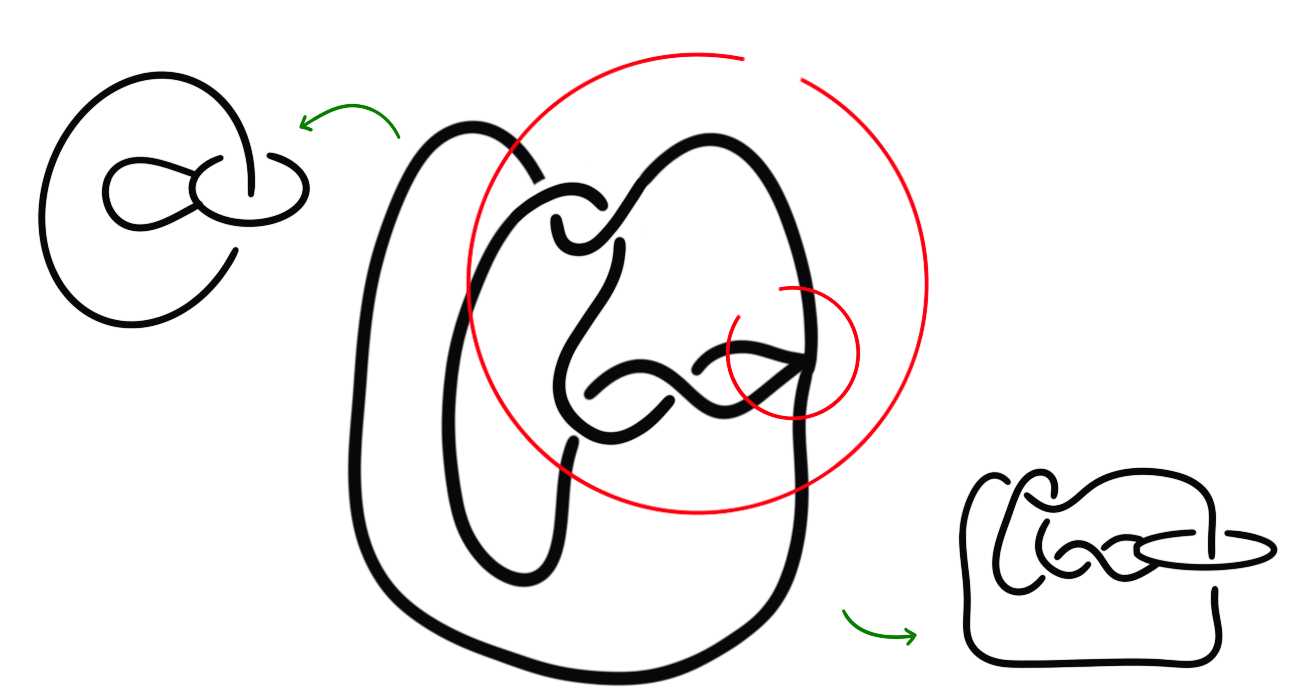}
			\put(55,45){\footnotesize $B_2$}
			\put(54,30){\tiny $B_1$}
			\put(9,23){\footnotesize $\Lambda^{\kappa_2}$}
			\put(85,20){\footnotesize $\Lambda^{\kappa_1}$}
		\end{overpic}
		\caption{Two vertical structures.}
		\label{fig:dependence}
	\end{subfigure} 
	\begin{subfigure}{.2\linewidth}		
	\centering
	\begin{overpic}[scale=.08,percent]{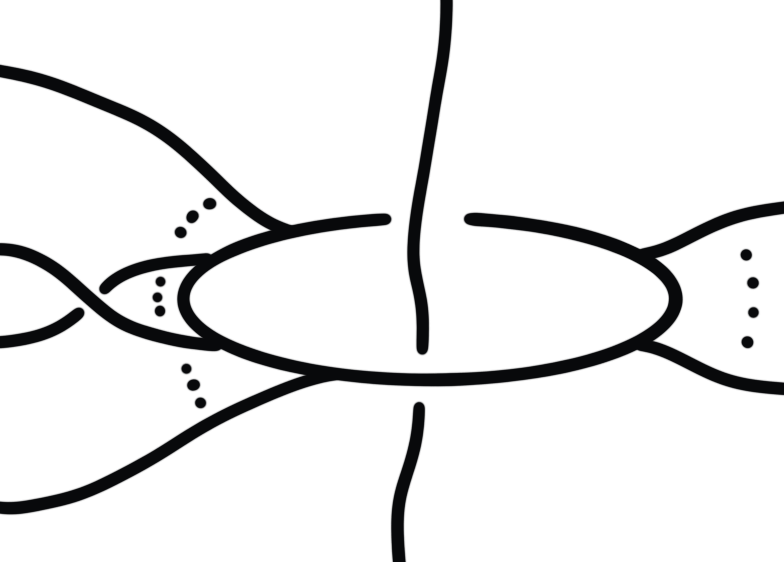}
	\end{overpic}
	\caption{Half twist $t_{ij}$.}
	\label{fig:h_ij}
	\end{subfigure} 
	\begin{subfigure}{.2\linewidth}		
	\centering
	\begin{overpic}[scale=.08,percent]{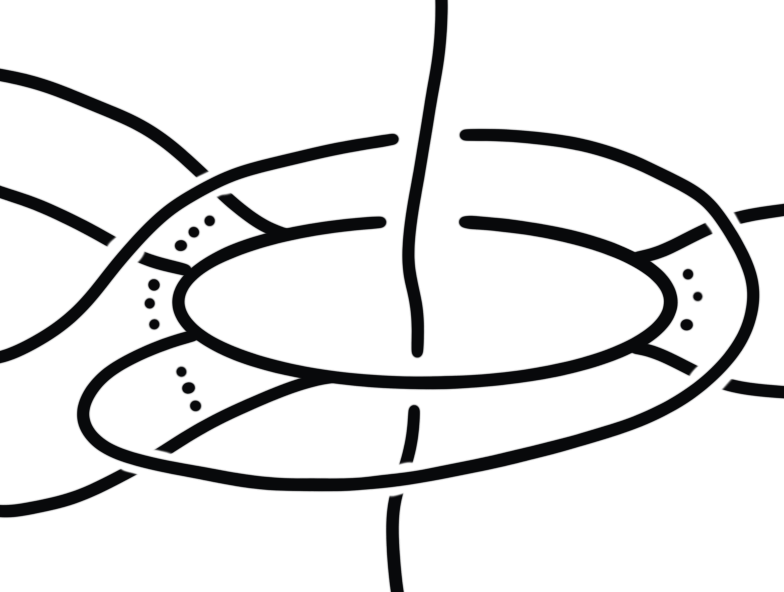}
	\end{overpic}
	\caption{Dehn twist $t_k$.}
	\label{fig:t_k}
	\end{subfigure}
	\begin{subfigure}{.15\linewidth}
	\centering	
	\begin{overpic}[scale=.05,percent]{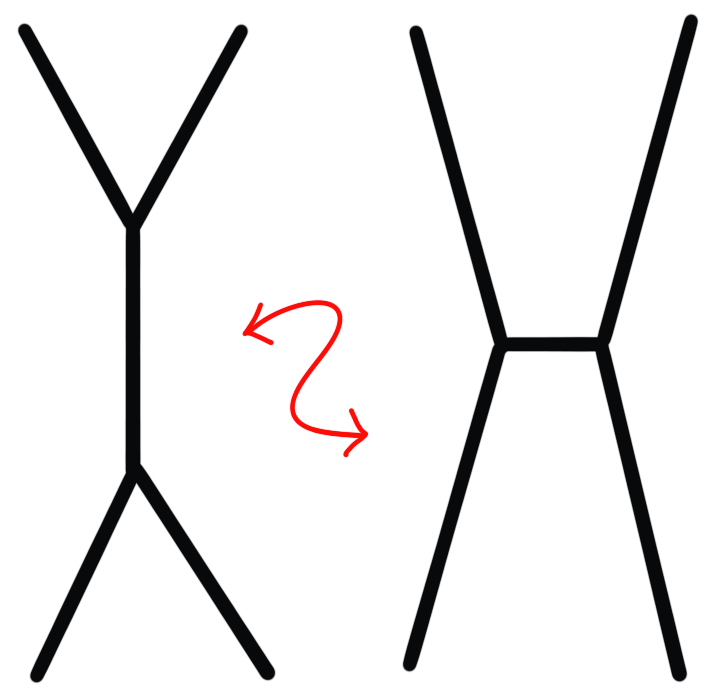}
	\end{overpic}
	\caption{IH-move.}
	\label{fig:IH_move}
	\end{subfigure} 	
	\caption{}
\end{figure}

\subsection{Looping}\label{subsec:looping}
Here we generalize the looping operation in \cite[Section $6$]{Wan:24}.
The \emph{standard vertical cone} $(B^3,\psi_n)$ is a pair of the unit $3$-ball $B^3\subset \mathbb R^3$ and the union $\psi_n:=\psi_{xy}\cup \psi_{z}$ where 
\begin{align*}
	\psi_{xy}&:=\{(t\cos \frac{2\pi i}{n},t\sin\frac{2\pi i}{n},0) \mid 0\leq t\leq 1, i=0,\dots, n-1 \};\\ 
	\psi_z&:=\{(0,0,t)\mid -1\leq t\leq 1\}.
\end{align*}
Let $o$ be the circle given by $\{(x,y,0)\mid x^2+y^2=(0.5)^2\}$, and denote by $o_{xy}$ the union $(\psi_{xy}-B^3_{.5})\cup o$, 
where $B^3_{.5}$ is the $3$-ball of radius $0.5$ at the origin. 
Then the \emph{looping} of the standard vertical cone is the pair $(B^3,\phi_n)$ given by 
\begin{equation}
\phi_n:=o_{xy}\cup 
\psi_z.
\end{equation}


Given a spatial graph $\Lambda$, a \emph{triplet $\epsilon=\{v,e_1,e_2\}$} consists of a vertex $v$ in $\Lambda$ and two edges $e_1,e_2$ incident to $\iota$ such that $e_1=e_2$ if one of them is a loop.
A \emph{vertical structure} $\kappa=(B,\iota)$ with respect to the triplet $\epsilon=\{v,e_1,e_2\}$ is a homeomorphism $\iota$ from a regular neighborhood $B$ of $v$ to $B^3$ such that
\begin{align*}
	\iota(B\cap \Lambda)&=\psi_n;\\ 
	\iota(B\cap (v\cup e_1\cup e_2))&=\psi_z.
\end{align*}

\begin{definition}
	A looping $\Lambda^\kappa$ of $\Lambda$ at the triplet $\epsilon$ with respect to the vertical structure $\kappa$ is the spatial graph 
	\[\Lambda^\kappa:=(\Lambda-B)\cup \iota^{-1}(\phi_n).\]	 
\end{definition}

A graphical presentation of a triplet $\epsilon$ (see Figs.\ \ref{fig:convention_triplet}, \ref{fig:convention_loop}) is a double-sided arrow where the two arrows indicate the two edges in $\epsilon$, and the center of the arrow is the vertex in $\epsilon$. For instance, if $\Lambda$ is a trivial loop, then the looping of $\Lambda$ is the Hopf link; see Fig.\ \ref{fig:hopf}. If $\Lambda$ is a trivial theta-graph, then its looping is the simplest non-trivial spatial handcuff; see Fig.\ \ref{fig:twoone}. In general, $\Lambda^\kappa$ depends on the choice of the vertical structure $\kappa$ when $n\geq 2$. Fig.\ \ref{fig:dependence} illustrates two inequivalent spatial graphs obtained by the vertical structures  $\kappa_1,\kappa_2$ of the same triplet. By contrast, the handlebody-link induced by $\Lambda^\kappa$ is independent of $\kappa$.  
 

\begin{lemma}\label{lm:independence}
Given two vertical structures $\kappa_1=(B_1,\iota_1),\kappa_2=(B_2,\iota_2)$ around $\epsilon$, the handlebody-links induced by $\Lambda^{\kappa_1},\Lambda^{\kappa_2}$ are equivalent. 
\end{lemma} 
\begin{proof}
Let $\Lambda_i:=B_i\cap \Lambda^{\kappa_i}$, $i=1,2$, 
and consider the homeomorphism $h:=\iota_2^{-1} \iota_1 :(B_1,\Lambda_1)\rightarrow (B_2,\Lambda_2)$. Denote by $\Lambda_i'$ the closure $\overline{\Lambda-\Lambda_i}$ and by $B_i'$ the exterior $\overline{S^3-B_i}$.
Then there is a homeomorphism $g:(B_1',\Lambda_1')\rightarrow (B_2',\Lambda_2')$.
Consider the restriction $g_\partial,h_\partial$ of $g,h$ on $S_i=\partial B_i=\partial B_i'$.
Then $f:=(h_\partial)^{-1} g_\partial$ is a self-homeomorphism of $S_1$. Let $p_0$, $p_\infty$ be the two points $e_1\cap S_1$, $e_2\cap S_1$, respectively, and $p_1,\cdots, p_n$ the points in $S_1\cap \Lambda-e_1\cup e_2$. 
Then $f$ fixes $p_0,p_\infty$ and preserves $p_1\cup\cdots\cup p_n$. 
In particular, by \cite[Section $4.2.4$]{FarMar:11}, 
$f$ is a composition 
of half Dehn twists $h_{ij}$ around $p_i\cup p_j$, $1\leq i <j\leq n$, and Dehn twists $t_k$ around $p_0\cup p_k$, $1\leq k\leq n$. In other words, 
$g_\partial=h_\partial  \sigma_1 \cdots  \sigma_s$, where $\sigma_l\in \{h_{ij},t_k\mid 1\leq i <j\leq n,1\leq k\leq n \}$.   
Since the local move in Fig.\ \ref{fig:IH_move} does not alter the regular neighborhood, one can extend the twists $h_{ij}$, $t_k$ to
a self-homeomorphism $\tilde{h}_{ij}, \tilde{t}_k$ of $(B_1,N(\Lambda_1))$; see Figs.\ \ref{fig:h_ij}, \ref{fig:t_k}. 
Set $\tilde{h}:=h \tilde{\sigma}_1 \cdots \tilde{\sigma}_k$. Then $\tilde{h},g$ restrict to the same homeomorphism on $S_1$, and the union $g\cup_{S_1} \tilde{h}$ gives a homeomorphism from $(S^3,\rnbhd{\Lambda^{\kappa_1}})$ to $(S^3,\rnbhd{\Lambda^{\kappa_2}})$.
\end{proof}

\begin{definition}
	A \emph{fat looping} $V^\epsilon$ of $\Lambda$ with respect to $\epsilon$ is the handlebody-link induced by the looping 
	$\Lambda^\kappa$ with respect to some vertical structure $\kappa$ around $\epsilon$. 
\end{definition} 


More generally, a \emph{pre-looping system} $\mathcal{E}$ is an non-empty collection of triplets $\epsilon_i=\{v_i,e_{i1},e_{i2}\}$, $i=1,\dots, m$, 
such that $v_1,\dots, v_m$ are distinct. A \emph{vertical structure} $\mathcal{K}$ of $\mathcal{E}$ is a collection of vertical structures $\kappa_i=(B_i,\iota_i)$ around $\epsilon_i$ with $B_i$, $i=1,\dots, m$, mutually disjoint and  
$\iota_i(\Lambda\cap B_i)=\psi_{n_i}$. 
The \emph{looping} $\Lambda^\mathcal{K}$ of $\Lambda$ with respect to $\mathcal{K}$ is then given by 
\[\Lambda^\mathcal{K}:=\left(\Lambda-(B_1\cup\cdots\cup B_m)\right)\cup \iota^{-1}_1(\phi_{n_1})\cup\cdots\cup \iota^{-1}_m(\phi_{n_m}).\]
The \emph{fat looping} $V^\mathcal{E}$ of $\Lambda$ at $\mathcal{E}=\{\epsilon_1,\dots, \epsilon_m\}$ is the handlebody-link induced by the looping $\Lambda^\mathcal{K}$ of $\Lambda$ with respect to some vertical structure $\mathcal{K}$ of $\mathcal{E}$. For instance, the pre-looping system in Fig.\ \ref{fig:cuff_like_E} yields Fig.\ \ref{fig:equivalent}. Note that fat loopings of different pre-looping systems may be equivalent; see Fig.\ \ref{fig:equivalent}.
%

A vertex (resp.\ edge) in $\epsilon_i$ is called a \emph{looped vertex} (resp.\ \emph{loop edge}). We use $\Lambda^\mathcal{E}$ to denote a spine of $V^\mathcal{E}$ obtained by the looping of $\Lambda$ with respect to some vertical structure of $\mathcal{E}$. 
%



\begin{figure}[b]
	\begin{subfigure}{.28\linewidth}		
		\centering
		\begin{overpic}[scale=.11,percent]{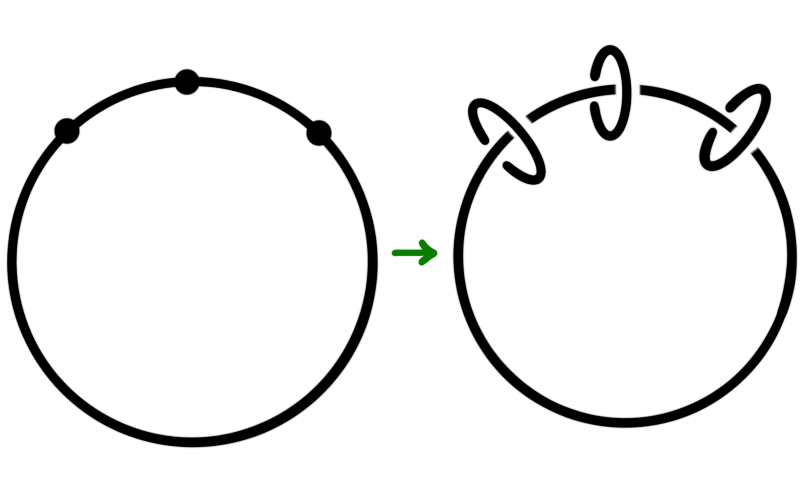}
		\end{overpic}
		\caption{Sum of Hopf links.}
		\label{fig:g_hopf}
	\end{subfigure} 
	\begin{subfigure}{.18\linewidth}		
		\centering
		\begin{overpic}[scale=.07,percent]{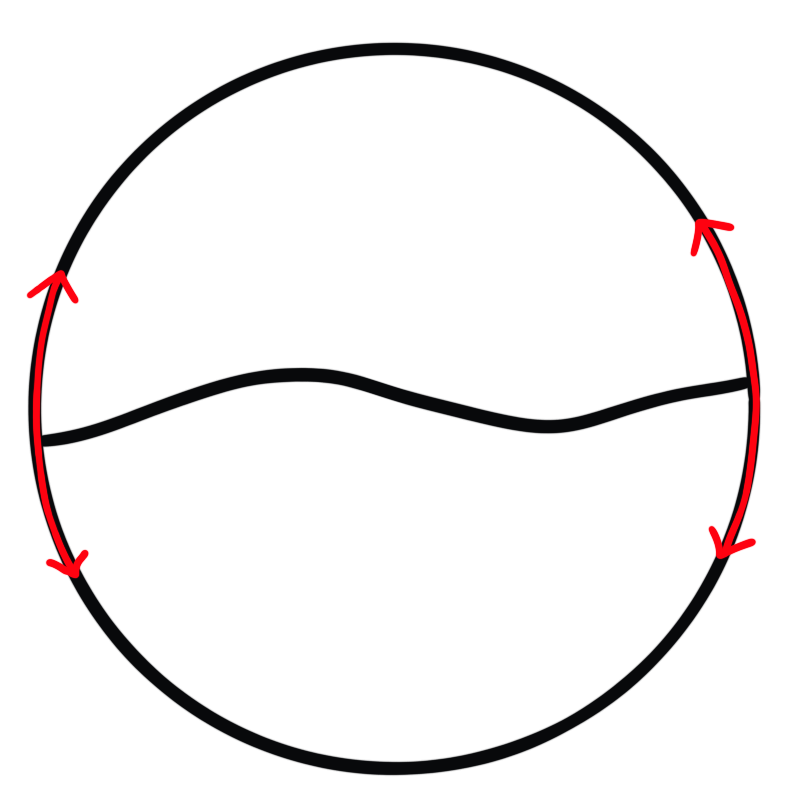}
		\end{overpic}
		\caption{}
		\label{fig:cuff_like_E}
	\end{subfigure} 
		\begin{subfigure}{.22\linewidth}		
		\centering
		\begin{overpic}[scale=.09,percent]{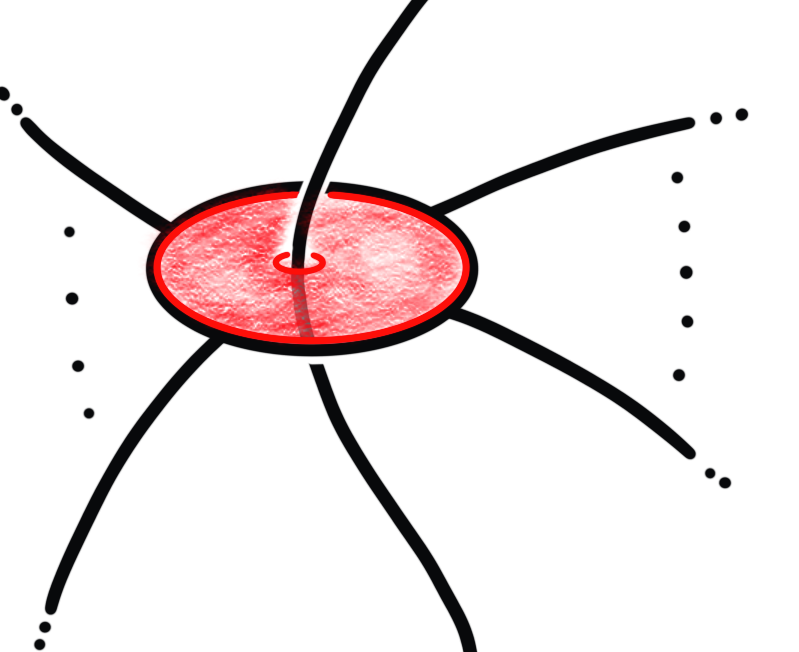}		\put(45,50){\tiny $A_i$}	
 		\end{overpic}
		\caption{Type $2$ annulus.}
		\label{fig:typetwo}
	\end{subfigure} 
	\begin{subfigure}{.28\linewidth}		
		\centering
		\begin{overpic}[scale=.09,percent]{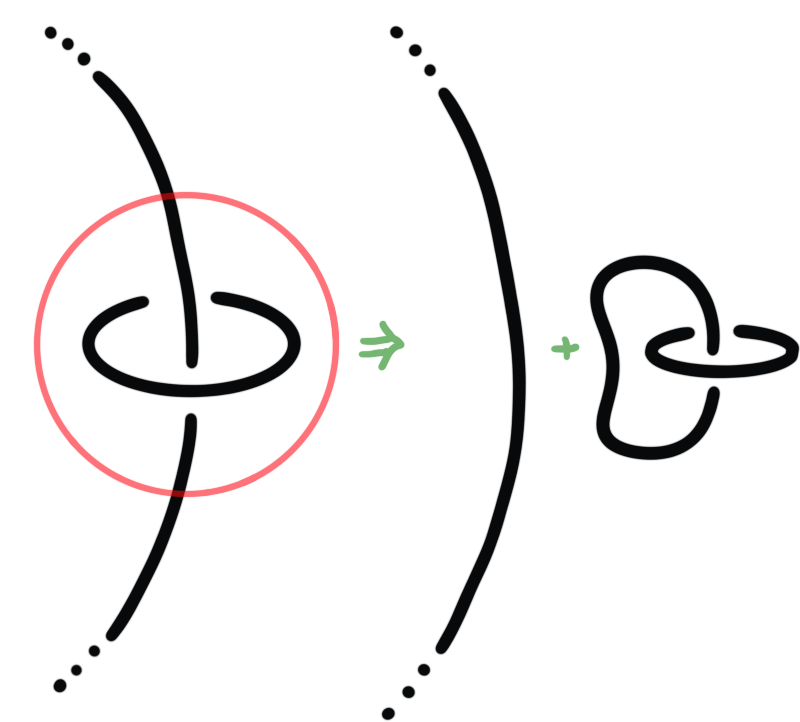}
		\end{overpic}
		\caption{Hopf summand.}
		\label{fig:classical_hopf_summand}
	\end{subfigure} 
	\caption{ }
\end{figure}

\subsection*{Induced annulus}  
Consider the disk $D_{.5}:=\{(x,y,0)\mid x^2+y^2\leq (0.5)^2\}$ bounded by the circle $o$ in $B^3$, and the fat looping $V^\mathcal{E}$ of $\Lambda$ with respect to $\mathcal{E}=\{\epsilon_1,\dots,\epsilon_m\}$. 
Let $\kappa_i=(B,\iota_i)$ be a vertical structure of $\epsilon_i$. Then the intersection $A_i:=\iota^{-1}_i(D_{.5})\cap \Compl {V^\mathcal{E}}$ is called the \emph{induced  
annulus} by $\epsilon_i$; see Fig.\ \ref{fig:typetwo}. 
Adopting the terminologies in \cite{KodOzaGor:15}, we say an annulus $A$ in the exterior of a handlebody-link $V$ is of \emph{type $1$} if both components of $\partial A$ bound essential disks in $V$, and is of \emph{type $2$} if exactly one component of $\partial A$ bounds an essential disk in $V$.
Otherwise, $A$ is of \emph{type $3$} (resp.\ \emph{type $4$}) if there exists a (resp.\ exists no) compressing disk of $\partial V$ in $\sphere$ disjoint from $A$. The induced annulus by a triplet is always of type $2$. The boundary component of a type $2$ annulus $A$ is \emph{longitudinal} if it is not \emph{meridional}.

\subsection*{Hopf handlebody-links} 
The fact that the Hopf link $\hopf_0$ is a looping of a loop trivially embedded in $\sphere$ motivates the following definition.
\begin{definition} 
A spatial graph (resp.\  handlebody-link) 
is \emph{Hopf} if it is equivalent to a looping (resp.\ fat looping) of a trivial spatial graph with respect to some vertical structures (resp.\  pre-looping system). 	
\end{definition}

A Hopf handlebody-link with $\chi=0$, 
namely $\Lambda$ is a cycle, is a connected sum of Hopf links; see Fig.\ \ref{fig:g_hopf}. Likewise, many simplest spatial graphs and handlebody-links with $\chi<0$ are Hopf.
In terms of the crossing number, the simplest handcuff graph in \cite{Mor:07}, the simplest genus two handlebody-knot in \cite{IshKisMorSuz:12} (see Fig.\ \ref{fig:theta}), and the simplest two-, three- and four-component handlebody-links in \cite{BelPaoPaoWan:23} are all Hopf. 

\subsection{Hopf handlebody-links of special interest}
Hereinafter $\Lambda$ is assumed to be a trivial connected spatial graph contained in a $2$-sphere $S_\ast\subset S^3$ with a pre-looping system  $\mathcal{E}$. Denote by 
$\chi$ the Euler characteristic of $\Lambda$.
Let $V\simeq V^\mathcal{E}$, and $E$ be the union of all looped edges. We are interested in \emph{irreducible} Hopf handlebody-links that are sufficiently \emph{indecomposable and rigid}.

\subsubsection*{Irreducibility}
While every Hopf handlebody-link with $\chi=0$ is irreducible, it is a subtle question to detect the irreducibility of a Hopf handlebody-link with $\chi<0$; see Fig.\ \ref{fig:hopf_summand}. 
Consider the subgraph 
\[\Lambda_0:=\Lambda-E.\]
Then the pre-looping system $\mathcal{E}$ is said to be \emph{adequate} if $\Lambda_0$ contains no cycle. 
For an adequate pre-looping system $\mathcal{E}$, we define the equivalence relation $\sim_\epsilon$ on $\Lambda$ as follows: 
$x\sim_\epsilon y$ if $x,y$ are in the same component of $\Lambda_0$. 
Set $\Lambda_\mathcal{E}:=\Lambda/\sim_\epsilon$. 
Since every component of $\Lambda_0$ is contractible, there is a quotient map $\pi$ from $(S_\ast,\Lambda)$ to 
$(S_\ast,\Lambda_\mathcal{E})$. Particularly, $\Lambda_\mathcal{E}$ is a plane graph whose edges are precisely those in $E$. 
A graph is \emph{$n$-connected} if it has \emph{vertex-connectivity} greater than $n-1$.

\begin{mtheorem}[Irreducibility]\label{teo:irreducibility}
Suppose $\chi<0$. 
Then $V\simeq V^\mathcal{E}$ is irreducible if and only if $\mathcal{E}$ is adequate, and $\Lambda_\mathcal{E}$ is $2$-connected, and contains no loop.
\end{mtheorem} 

Theorem \ref{teo:irreducibility} implies Fig.\ \ref{fig:dihedral_three} is the only irreducible handlebody-link in Fig.\ \ref{fig:examples}.

\subsubsection*{Simpleness}
While it is known, when $\chi=0$, the symmetry group of $V$ is finite if and only if $V$ is a Hopf link $\hopf_0$ or the sum $\hopf_0\# \hopf_0$ of two copies of Hopf links, it is considerably more involved to characterize Hopf handlebody-links with a finite symmetry group when $\chi<0$.

Given the Mostow rigidity theorem, hyperbolic $3$-manifolds are often regarded as the most rigid among all $3$-manifolds. Indeed, the mapping class group of a compact hyperbolic $3$-manifold is always finite.  
On the other hand, Johannson \cite{Joh:95} introduces the notion of the \emph{simple relative handlebody}. A \emph{relative handlebody} $(M,k)$ is a pair of a handlebody $M$ with a disjoint union $k$ of simple closed curves on $\partial M$, and $(M,k)$ is \emph{simple} if $M-k$ admits no essential disks and no essential annuli that are not $\partial$-parallel into a regular neighborhood $\rnbhd{k}\subset \partial M$ of $k$. 
Given a simple relative handlebody $(M,k)$, the mapping class group $\sym[M]{k}$ is always finite by \cite[Proposition $8.8$]{Joh:95}, and  by Thurston's hyperbolization theorem \cite{Thu:82}, either $M-k$ admits a hyperbolic structure with totally geodesic boundary or there is an I-bundle structure 
$\pi:M\rightarrow P$ over a pair of pants $P$ with $\pi^{-1}(\partial P)$ is $\rnbhd{k}\subset \partial M$.

Though the exterior of $V$ is neither hyperbolic nor simple in Johannson's sense, we can define a notion of \emph{simpleness modulo $\mathcal{E}$}. 
Consider the union $\mathcal{D}_\mathcal{E}\subset \rnbhd{\Lambda}$ 
of disjoint meridian disks dual to looped edges. 
Then 
\emph{the relative handlebody associated to $\mathcal{E}$} is the pair $(\Compl{\Lambda}, k_\mathcal{E})$ of the exterior $\Compl{\Lambda}$ of $\Lambda\subset S^3$ with the simple closed curves $k_\mathcal{E}:=\partial \mathcal{D}_\mathcal{E}$.
If $\mathcal{E}$ is adequate, then $\Compl{\Lambda}=\Compl{\Lambda_\mathcal{E}}$  
and $k_\mathcal{E}$ is the meridian system of $\Lambda_\mathcal{E}$.

\begin{definition}\label{def:simpleness}
A pre-looping system is \emph{simple} if the associated relative handlebody is simple. 
A Hopf handlebody-link is \emph{simple} if it is equivalent to a fat looping with respect to some simple pre-looping system.  
\end{definition}
Simpleness implies irreducibility by Lemma \ref{lm:simple_is_irreducible}. Also, Lemma \ref{lm:simple_no_classical_hopf} shows a simple Hopf handlebody-link is sufficiently indecomposable, in the sense that no Hopf link $\hopf_0$ can be factored out; see Fig.\ \ref{fig:classical_hopf_summand}. Note that Definition \ref{def:simpleness} requires only the existence of a simple pre-looping system, leaving the possibility that a non-simple pre-looping system may yield a simple Hopf handlebody-link. Our second result asserts this never happen, and gives a graph-theoretic criterion for simpleness.  

\begin{mtheorem}[Simpleness]\label{teo:simpleness}
Suppose $\chi<0$. Then the following statements are equivalent: 
\begin{enumerate}[label=(\roman*)]
\item\label{itm:one_simple} $V$ is simple,
\item\label{itm:other_simple} $\mathcal{E}$ is simple,
\item\label{itm:graph_simple} $\mathcal{E}$ is adequate, and $\Lambda_\mathcal{E}$ is $3$-connected, and $\Lambda_\mathcal{E}$ either is simplicial or is a theta graph.
\end{enumerate} 
\end{mtheorem}

As an application, Fig.\ \ref{fig:dihedral_three} is the only simple handlebody-link in Fig.\ \ref{fig:examples}. 
Recall that the positive symmetry groups of $H_0$ and $H_0\# H_0$ are $\mathbb{Z}_2^2$, and 
their symmetry groups are $D_4$ and $\mathbb{Z}_2^3$, respectively.
For the case $\chi<0$, our third result shows simpleness implies finiteness, and classifies their symmetry groups.

\begin{mtheorem}[Finite symmetry]\label{teo:classification}\hfill
Suppose $\chi\leq -1$ and $V$ is simple. Then
\begin{enumerate}[label=(\roman*)] 
	\item\label{itm:product}  
	$\sym{V}\simeq \psym{V}\times \mathbb{Z}_2$;  
	\item\label{itm:finiteness}
	 $\psym{V}\simeq \mcg{\Lambda_\mathcal{E},\mathcal{E}}<O(3)$, and $V$ is Nielsen realizable;
	\item\label{itm:realization} conversely, every finite subgroup of $O(3)$ is isomorphic to the positive symmetry group of some simple Hopf
	handlebody-link with $\chi\leq-1$.
\end{enumerate} 
\end{mtheorem}

Theorem \ref{teo:classification}\ref{itm:finiteness} reduces the $3$-dimensional symmetry computation to a $1$-dimensional problem; for instance, the symmetry group of Fig.\ \ref{fig:dihedral_three} is $D_3\times \mathbb Z_2$.

\begin{mtheorem}[Handlebody-knot table]\label{teo:enumeration} 
Table \ref{tab:hopf_handlebody_knots} enumerates all simple Hopf handlebody-knots with $\chi\geq -4$, up to equivalence, without duplication.  
\end{mtheorem}

Theorems \ref{teo:irreducibility}, \ref{teo:simpleness}, \ref{teo:classification} are, respectively, proved in Sections \ref{subsec:irre}, \ref{subsec:simp}, \ref{subsec:symm}. 
Theorem \ref{teo:enumeration} occupies the entire 
Section \ref{sec:enumeration}.

\section{Irreducibility, simpleness and symmetry}\label{sec:symmetry}
%

Recall that $\Lambda$ is a trivial connected spatial graph, contained in the $2$-sphere $S_\ast\subset \sphere$, with a pre-looping system $\mathcal{E}$. Assume in addition $\chi<0$ as the case $\chi=0$ is well-understood. 
Let $\mathcal{A}_\mathcal{E}$ (resp.\ $\mathcal{D}_\mathcal{E}$) be the union of the type $2$ annuli induced by triplets in $\mathcal{E}$ (resp.\ disks dual to looped edges), and $V$ a handlebody-link equivalent to $V^\mathcal{E}$. 
For the sake of simplicity, we also call the image of a triplet in $\mathcal{E}$ under $\pi$ a \emph{triplet}, and use the same notation $\epsilon$ and $\mathcal{E}$ for their images under $\pi$.

\begin{figure}[b]
\begin{subfigure}{.27\linewidth}		
\centering		\begin{overpic}[scale=.11,percent]{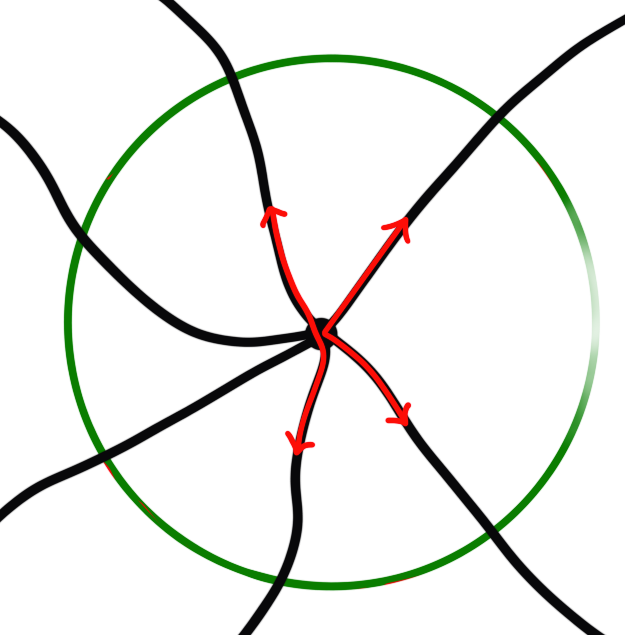}
			\put(36,93){\tiny $p_1$}
			\put(44,2){\tiny $q_1$}
			\put(72,89){\tiny $p_2$}
			\put(80,14){\tiny $q_2$}
			\put(55,45){\footnotesize $v$}
			\put(86,50){\footnotesize $N_v$}
\end{overpic}		
\caption{Cone neighborhood.}
\label{fig:cone_nbhd_1}
\end{subfigure} 
\begin{subfigure}{.22\linewidth}		
	\centering		\begin{overpic}[scale=.1,percent]{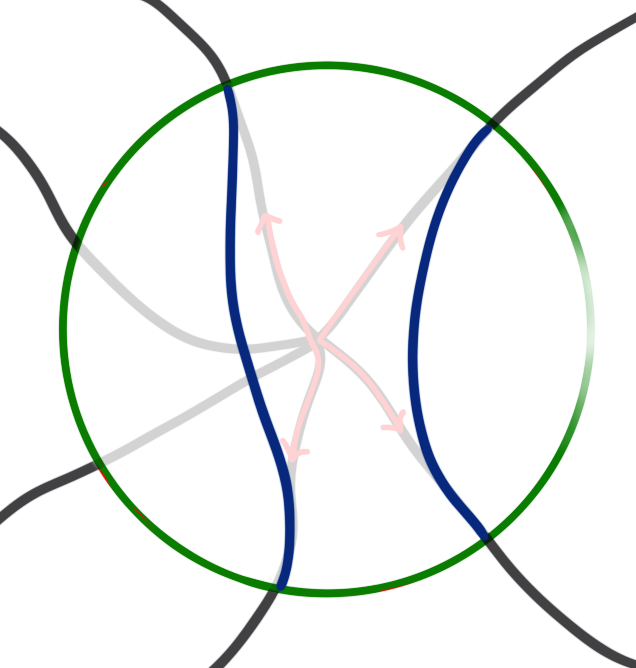}
	\put(34,92){\tiny $p_1$}
	\put(41,5){\tiny $q_1$}
	\put(68.5,87.5){\tiny $p_2$}
	\put(77,16){\tiny $q_2$}
	\put(38,50){\tiny $\gamma_1$}
	\put(63,50){\tiny $\gamma_2$}
	\put(83,50){\footnotesize $N_v$}
	\end{overpic}		
	\caption{Smoothing.}
	\label{fig:smoothing_1}
\end{subfigure} 
	\begin{subfigure}{.24\linewidth}		
		\centering
		\begin{overpic}[scale=.12,percent]{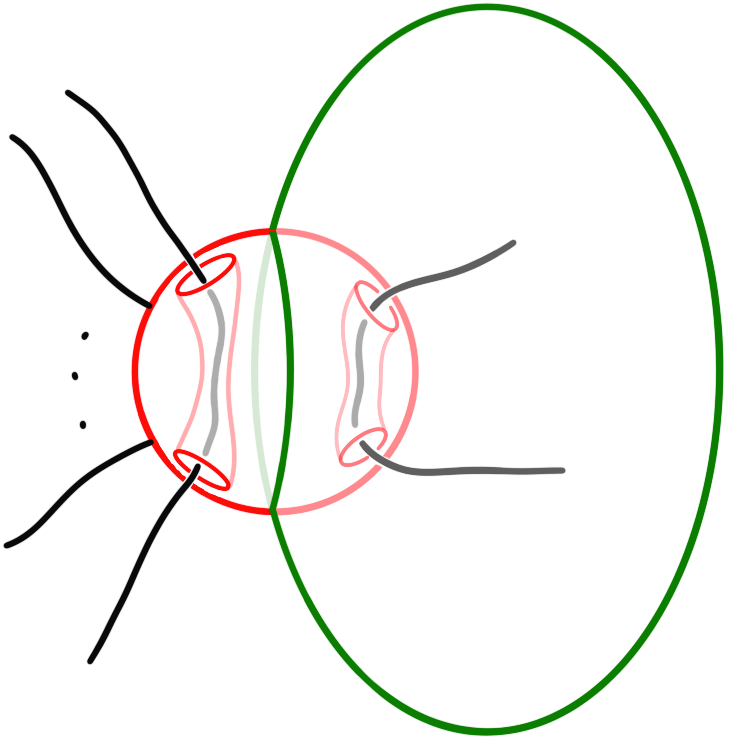}								  	\put(9,70){\footnotesize $\Lambda'$}
			\put(32,58){{\footnotesize \transparent{.7}$W_v$}}
			\put(60,7){{\footnotesize \transparent{.6}$D$}}			 
		\end{overpic}
		\caption{$\Lambda'$ and $W_v$.}
		\label{fig:not_separate}
	\end{subfigure} 
	\begin{subfigure}{.24\linewidth}		
		\centering
		\begin{overpic}[scale=.12,percent]{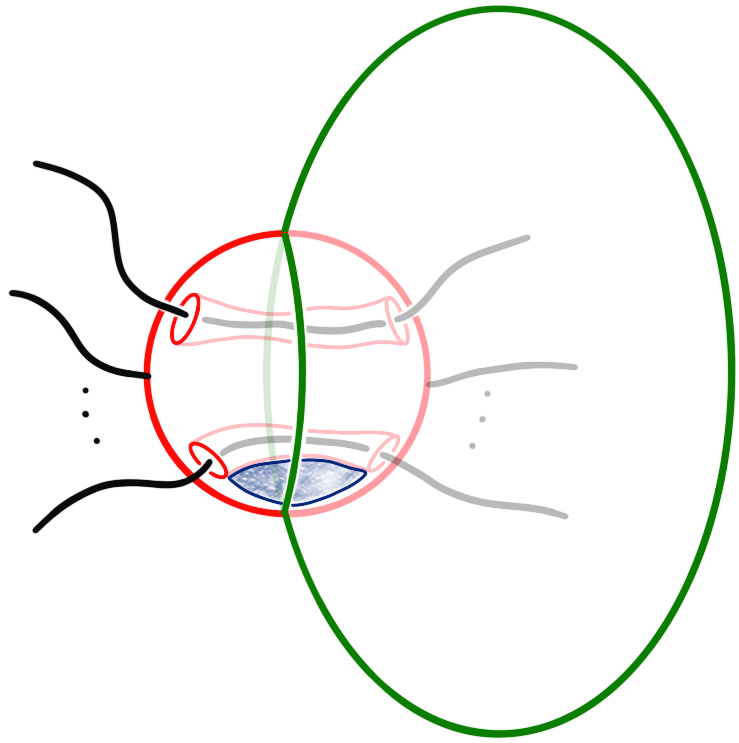}
			\put(7,79){\footnotesize $\Lambda'$}
			\put(22,45){\footnotesize $W_v$}
			\put(60,7){{\footnotesize  \transparent{.6}$D$}}	
			\put(35,31){{\tiny \transparent{.7} $D'$}}		 
		\end{overpic}
		\caption{$\Lambda'$ and $W_v$.}
		\label{fig:separate}
	\end{subfigure} 
	\caption{}
\end{figure}

\subsection{Construct $V^\mathcal{E}$ from $\Lambda_\mathcal{E}$}\label{subsec:Lambda_E_to_V_E}  
A \emph{cone neighborhood} of a vertex $v$ in $\Lambda_\mathcal{E}$ is a $3$-ball 
$B_v\subset \sphere$ 
such that $B_v\cap S_\ast$ is a disk $N_v$ containing no other vertices than $v$ and
$\Lambda_\mathcal{E}\cap N_v$ is a cone at $v$; see Fig.\ \ref{fig:cone_nbhd_1}. Choose $B_v$ so $B_v\cap B_w=\emptyset$ if $v\neq w$, and let $\epsilon_{v_1},\cdots, \epsilon_{v_k}$ be the triplets containing $v$, and $p_{i},q_{i}$ the intersection of $\partial N_v$ the edges in $\epsilon_{v_i}$.
Up to isotopy, there is a unique set of disjoint arcs $\gamma_i$ in $N_v$ joining $p_{i},q_{i}$, $i=1,\dots, k$. 
The union $\gamma_v:=\cup_i\gamma_i$ is called the \emph{smoothing} of $\Lambda_\mathcal{E}$ at $v$; see Fig.\ \ref{fig:smoothing_1}; denote by $W_v$ the exterior of $\gamma_v$ in $B_v$; see Figs.\ \ref{fig:separate}, \ref{fig:not_separate}. 

Set $\gamma:=\cup_v \gamma_v$, $\mathcal{B}:=\cup_v B_v$, and $\mathcal{W}:=\cup_v W_v$, and 
consider $\Lambda':=\overline{\Lambda_\mathcal{E}-\mathcal{B}}\cup \gamma$. Then, if $\mathcal{E}$ is adequate, $V^\mathcal{E}$ can be identified with a regular neighborhood of    
$\Lambda'\cup \mathcal{W}$. Let $\tilde\Lambda$ be another trivial spatial graph lying on $S_\ast$ with the pre-looping system $\tilde{\mathcal{E}}$. 
The above construction implies the following.
\begin{lemma}\label{lm:Lambda_E_to_V_E}
	If $\mathcal{E},\tilde{\mathcal{E}}$ are adequate and $(\tilde{\Lambda}_{\tilde{\mathcal{E}}},\tilde{\mathcal{E}})=(\Lambda_\mathcal{E},\mathcal{E})$, then $V^\mathcal{E}\simeq V^{\tilde{\mathcal{E}}}$.
\end{lemma}

\subsection{Irreducibility}\label{subsec:irre}
The following is a strengthening of Theorem \ref{teo:irreducibility}.   
 
\begin{theorem}\label{teo:strong_irreducibility}
The following statements are equivalent:
\begin{enumerate}[label=(\roman*)]
 \item\label{itm:irre} $V$ is irreducible;
 \item\label{itm:b_irre} $\Compl V$ is irreducible and $\partial$-irreducible;
 \item\label{itm:conn} $\mathcal{E}$ is adequate, and $\Lambda_\mathcal{E}$ is 
 $2$-connected and contains no loop. 
\end{enumerate}	 
\end{theorem}
\begin{proof}
\ref{itm:b_irre}$\Rightarrow$\ref{itm:irre} is clear. For \ref{itm:irre}$\Rightarrow$\ref{itm:conn}, we divide it into three cases.
 
\textbf{Case 1: $\mathcal{E}$ is not adequate.}
Since $\Lambda_0$ contains a cycle $c$, 
there is a disk $\sigma\subset S^3$ such that
$\partial \sigma=c=\sigma\cap S_\ast$. In particular, $\sigma\cap \Compl V$ is a non-separating disk $D$ in $\Compl V$ meeting an essential disk $D'\subset V$ dual to $c$. Since $\mathcal{E}\neq\emptyset$, $c\neq \Lambda$, so there is an $n$-valent vertex with $n>2$ in $c$. Thus the frontier of a regular neighborhood of $D\cup \partial D'\subset \Compl V$ induces a $1$-decomposing sphere, contradicting $V$ is irreducible. 

Suppose $\mathcal{E}$ is adequate.
If $\Lambda_\mathcal{E}$ has vertex connectivity $1$ or contains a loop, then since $\Lambda_\mathcal{E}$ is trivial with $\chi<0$, there is a $2$-sphere $S\subset S^3$ meeting $\Lambda_\mathcal{E}$ at a vertex $v$ and cutting $\Lambda_\mathcal{E}$ into two components. Consider the identification in Section \ref{subsec:Lambda_E_to_V_E}. It may be assumed that $B_v\cap S$ is a disk $D_v$. 
If $D_v$ does not separate edges in any  triplet (see Fig.\ \ref{fig:not_separate}), then $S$ is a $1$-decomposing sphere, a contradiction. 

Suppose 
$D_v$ separates edges of some triplet $\epsilon=\{v,e_1,e_2\}$, and let $D:=\overline{S-D_v}$. 
If $v$ is bi-valent, then $W_v$ a solid torus component of $V^\mathcal{E}$. The disk $D$ and the annulus $A$ induced by $\epsilon$ cut off an annulus $A'$ from $\partial W_v$. The union $D\cup A\cup A'$ induces a $1$-decomposing sphere of $V$, a contradiction. If $v$ is not bi-valent, then a parallelism of $e_1\cup e_2\cup v$ in $B_v$ induces a disk $D'\subset W_v-\Lambda'$ meeting $D$ at a point; see Fig.\ \ref{fig:separate}. The frontier of $D\cup \partial D'\subset\Compl{V}$ induces a $1$-decomposing sphere, a contradiction.

For \ref{itm:conn}$\Rightarrow$\ref{itm:b_irre}, we observe that, if $\Compl V$ is reducible, and $S$ is an essential $2$-sphere in $\Compl V$, then $\Lambda$ being connected implies that $S$ can be isotoped so 
$S\cap \mathcal{A}_\mathcal{E}$ consists of some essential loops. The innermost disk $D\subset S$ cut off by the intersection 
induces an essential disk $D'$ in $\Compl V$, so $\Compl V$ is $\partial$-reducible; see Fig.\ \ref{fig:essential_sphere}.
Thus, it suffices to show that $\Compl V$ is $\partial$-irreducible.

Suppose otherwise, and $D$ is an essential disk that minimizes $\vert D\cap \mathcal{A}_\mathcal{E}\vert$ among all essential disks in $\Compl V$. If $D\cap \mathcal{A}_\mathcal{E}$ contains a circle, then an innermost disk in $D$ cut off by $D\cap\mathcal{A}_\mathcal{E}$ induces an essential disk disjoint from $\mathcal{A}_\mathcal{E}$, contradicting the minimality. 

Suppose $D\cap \mathcal{A}_\mathcal{E}$ contains some arcs, and $D'$ is an outermost disk $D'$ cut off by $D\cap\mathcal{A}_\mathcal{E}$. 
If $D'\cap \mathcal{A}_\mathcal{E}$ is inessential in $\mathcal{A}_\mathcal{E}$, and $D''\subset \mathcal{A}_\mathcal{E}$ be the disk 
cut off by $D'$, then $D'\cup D''$ induces an essential disk in $\Compl V$ disjoint from $\mathcal{A}_\mathcal{E}$, contradicting the minimality. Therefore $D'\cap \mathcal{A}_\mathcal{E}$ is essential in some component $A\subset \mathcal{A}_\mathcal{E}$, yet the frontier of a regular neighborhood of $A\cup D$ induces an essential disk in $\Compl V$ disjoint from $\mathcal{A}_\mathcal{E}$, a contradiction. As a result, $D\cap \mathcal{A}_\mathcal{E}=\emptyset$. 
 
Since $\mathcal{E}$ is adequate, $\mathcal{D}_\mathcal{E}$ cuts $\rnbhd{\Lambda}=\rnbhd{\Lambda_\mathcal{E}}$ 
into some $3$-balls, each meeting $\Lambda_{\mathcal{E}}$ at a cone. One of them, say $B$, contains $\partial D$. Since $\partial D$ bounds a disk $D'\subset B$ meeting $\Lambda_\mathcal{E}$ at the vertex in $B$, the $2$-sphere $D\cup D'$ cuts $\Lambda_\mathcal{E}$ into two components. This implies $\Lambda_\mathcal{E}$ either contains a loop or is $2$-connected, a contradiction.  
\end{proof}

\begin{figure}[t]
\begin{subfigure}{.32\linewidth}		
		\centering		\begin{overpic}[scale=.11,percent]{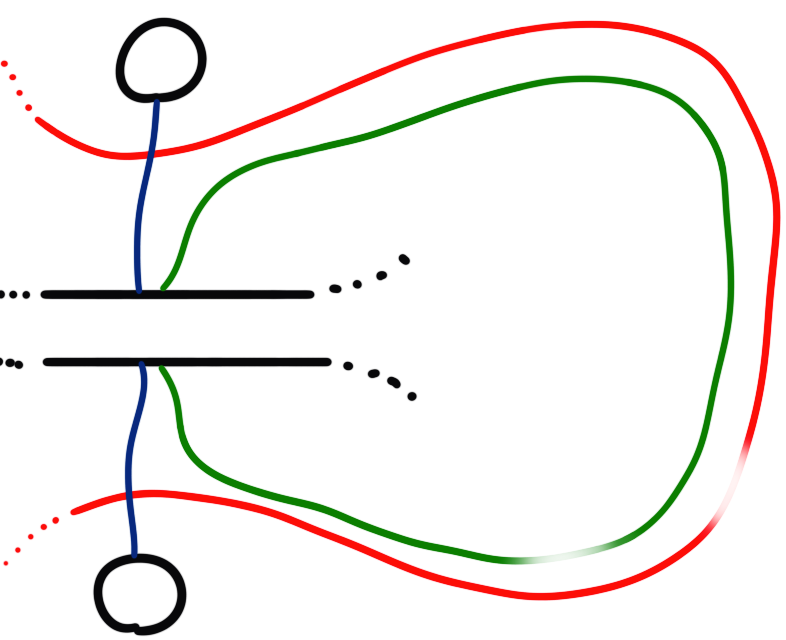}
			\put(89,16){\footnotesize $D$}
			\put(65,8){\footnotesize $D'$}
			\put(4,70){\tiny $S$}
			\put(4.5,50){\tiny $\mathcal{A}_\mathcal{E}$}
		\end{overpic}		
		\caption{Essential sphere.}
		\label{fig:essential_sphere}
\end{subfigure} 
	\begin{subfigure}{.32\linewidth}		
		\centering
		\begin{overpic}[scale=.12,percent]{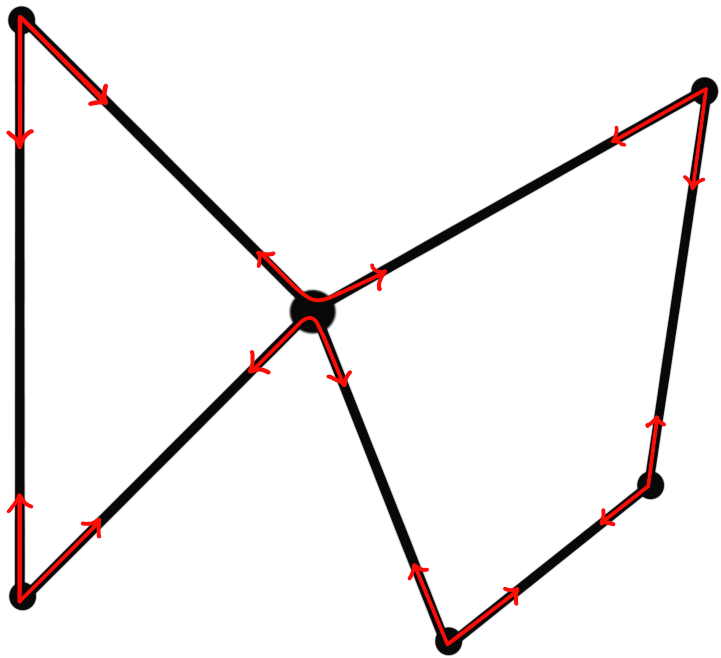}	
		\put(47,45){\tiny $v_1=v_4$}
		\put(59.5,37){\tiny $=v_8$}
		\put(25,67){\tiny $e_1$}
		\put(0,91.5){\tiny $v_2$}
		\put(-7.5,50){\tiny $e_2$}
		\put(95,82){\tiny $v_7$}
		\put(40,53){\footnotesize $\epsilon_1$}
		\put(3.5,76){\footnotesize $\epsilon_2$}
		\put(37.5,37.5){\footnotesize $\epsilon_4$}
		\end{overpic}
		\caption{Closed walk.}
		\label{fig:closed_walk}
	\end{subfigure} 
	\begin{subfigure}{.32\linewidth}		
		\centering
		\begin{overpic}[scale=.12,percent]{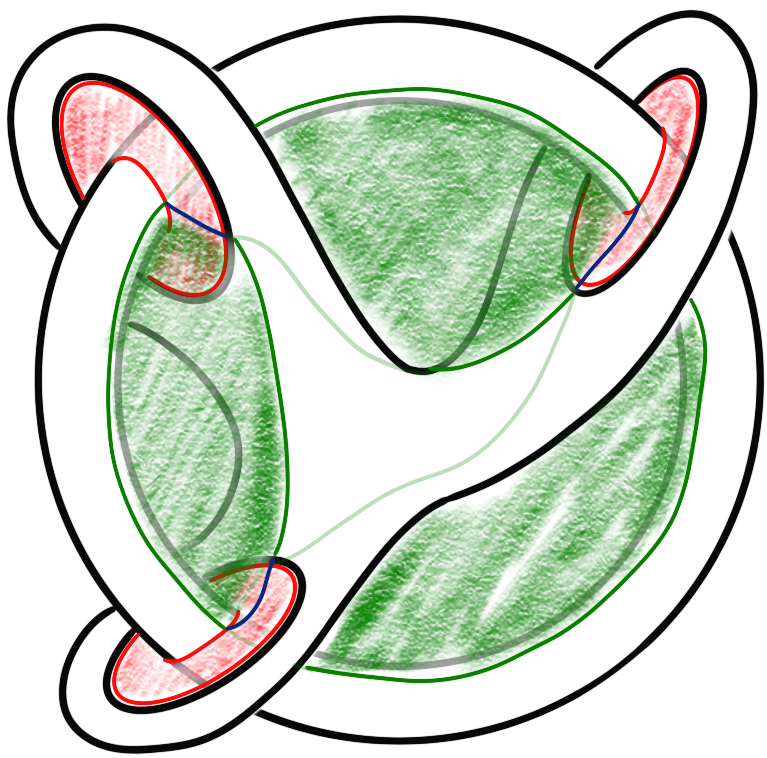}
			\put(37.5,45){\footnotesize $m$}
			\put(10.5,45){\footnotesize $l$}
			\put(10,79.7){\tiny $A_1$}	
			\put(14,9.5){\tiny $A_3$}		 
			\put(74.5,67.5){{\tiny \transparent{.7}$A_2$}}		 
			\put(50,70){\footnotesize $A$}		 
		\end{overpic}
		\caption{$l,m\subset \partial A$.}
		\label{fig:longitudinal_meridional}
	\end{subfigure} 
	\caption{}
\end{figure}
 
\subsection{Simpleness}\label{subsec:simp}
\begin{lemma}\label{lm:simple_is_irreducible}
If $\mathcal{E}$ is simple, then $V\simeq V^\mathcal{E}$ is irreducible.  
\end{lemma}
\begin{proof}
By the proof of \ref{itm:irre}$\Rightarrow$\ref{itm:conn} in Theorem \ref{teo:strong_irreducibility}, if $\mathcal{E}$ is inadequate, or if $\mathcal{E}$ is adequate and $\Lambda_\mathcal{E}$ has vertex-connectivity $1$ or contains a loop, then there is an essential disk in $\Compl{\Lambda}-k_\mathcal{E}$, contradicting simpleness. 
The assertion then follows
from the direction \ref{itm:conn}$\Rightarrow$\ref{itm:irre}
of Theorem \ref{teo:strong_irreducibility}.
\end{proof}

\begin{lemma}\label{lm:simple_no_bivalent}
Suppose $\mathcal{E}$ is simple.
Then $\Lambda_\mathcal{E}$ contains no bi-valent vertex. 
\end{lemma}
\begin{proof}
Suppose $v$ is a bi-valent vertex in 
$\Lambda_\mathcal{E}$. Then the meridians $k_1,k_2\subset N(\Lambda_\mathcal{E})$ dual to the edges adjacent to $v$, being parallel, cut off an annulus $A$ from 
$\partial\rnbhd{\Lambda_\mathcal{E}}$. Since $\chi<0$, the frontier of a regular neighborhood of $A\subset\Compl{\Lambda_\mathcal{E}}$ is an essential annulus in $\Compl{\Lambda_\mathcal{E}}-k_\mathcal{E}$ not $\partial$-parallel into a regular neighborhood of $k_\mathcal{E}$.
\end{proof}

We say $V$ admits a \emph{classical Hopf summand} if it admits a spine $\Gamma$ and a $3$-ball $B\subset S^3$ with $(B, B\cap \Gamma)$ 
homeomorphic to $(B^3, \phi_0)$; see Fig.\ \ref{fig:classical_hopf_summand}.
 
\begin{lemma}\label{lm:simple_no_classical_hopf}
If $\mathcal{E}$ is simple, then 
$V^\mathcal{E}$ admits no classical Hopf summand.
\end{lemma}
\begin{proof}
Suppose $V^\mathcal{E}$ admits a classical Hopf summand. Since $\Lambda_\mathcal{E}$ admits no bi-valent vertex by Lemma \ref{lm:simple_no_bivalent}, the loop component in $\phi_0$ corresponds to a closed walk 
\[v_1-e_1-v_2-\cdots-v_{n}-e_{n}-v_{n+1}=v_1\] 
in $\Lambda_\mathcal{E}$ where no two edges $e_i,e_j$ are identical, and $e_{i-1},v_i,e_i$ are in the same triplet $\epsilon_i$, $i= 2,\cdots, n$, and $e_n,v_1,e_1$ are in the triplet $\epsilon_1$; see Fig.\ \ref{fig:closed_walk}.
Let $A$ be the type $2$ annulus given by the classical Hopf summand, and $l$ (resp.\ $m$) be its longitudinal (resp.\ meridional) boundary component. 
Up to isotopy, $l$ meets the type $2$ annulus $A_i$ induced by $\epsilon_i$ each at one point and in the meridional component of $\partial A_i$, and $l$ is disjoint from other annuli in $\mathcal{A}_\mathcal{E}$; see Fig.\ \ref{fig:longitudinal_meridional}.
By the irreducibility of $V^\mathcal{E}$, $\mathcal{A}_\mathcal{E}$ and $A$ are essential, so it may be assumed that $m$ only meets $A_i$ at its longitudinal boundary components, for each $i$, and is disjoint from $\partial (A_\mathcal{E}-\cup_i A_i)$; see Fig.\ \ref{fig:longitudinal_meridional}. Since meridional components of $\partial \mathcal{A}_\mathcal{E}$ 
are parallel to $\partial \mathcal{D}_\mathcal{E}$, $m$ 
can be isotoped away from $\mathcal{D}_\mathcal{E}$. 
Since $\mathcal{D}_\mathcal{E}$ is dual to $\Lambda_\mathcal{E}$, it cuts $\rnbhd{\Lambda_\mathcal{E}}$ into some $3$-balls; each meets $\Lambda_\mathcal{E}$ at a cone. Since $m\cap \mathcal{D}_\mathcal{E}=\emptyset$, and $m$ meets each $A_i$, we have $m$ and $A_1,\cdots, A_n$ are in the same $3$-ball. This implies all vertices $v_1,\cdots, v_n$ are identical, so $\Lambda_\mathcal{E}$ contains a loop, contradicting Theorem \ref{teo:irreducibility} and Lemma \ref{lm:simple_is_irreducible}.
\end{proof}

\subsubsection{JSJ-decomposition}
Consider a pair $(X,F)$ of a compact $3$-manifold $X$
with a compact surface $F$ in $\partial X$. 
Then $(X,F)$ is \emph{Seifert fibered (resp.\ I-fibered)} if there is a Seifert bundle structure $\pi:(X,F)\rightarrow (S,b)$ (resp.\ I-bundle structure $\pi:(X,F)\rightarrow (S,\partial S)$), where $S$ is a compact surface with $b\subset\partial S$ some arcs, and 
$(X,F)$ is \emph{hyperbolic} if $X-F$ admits a hyperbolic structure with totally geodesics boundary. For instance, given a relative handlebody $(M,k)$, let $F$ be a regular neighborhood of $k$ in $\partial M$. Then $(M,k)$ is simple if and only if $(M,F)$ is hyperbolic or $(M,F)$ is I-fibered over a pair of pants.

Given a compact irreducible $\partial$-irreducible $3$-manifold $M$,
the JSJ-decomposition 
\cite{JacSha:79}, 
\cite{Joh:79} and Thurston's hyperbolization theorem \cite{Thu:82} guarantee the existence and uniqueness of a union $F$ of disjoint, mutually non-parallel essential annuli and tori in $M$ such that, 
\begin{enumerate}[label=(C\arabic*)]
\item\label{itm:complete} for each component $X$ in $\overline{M-\rnbhd F}$, $(X,\partial_f X)$ is either Seifert or I-fibered or hyperbolic, and   
\item\label{itm:minimality} the removal of any component in $F$ renders \ref{itm:complete} to fail.  
\end{enumerate} 
The union $S$ is called the \emph{characteristic surface} of $M$. 
If $M$ is atoroidal with $\partial M\neq \emptyset$, every Seifert fibered component is a solid torus. 
  
The \emph{JSJ-graph} of $M$ is a graph whose nodes are components in $\overline{M-\rnbhd{F}}$ and edges are components in $\rnbhd{F}$. To distinguish the hyperbolic from the fibered, we use the \emph{hollow} node for the hyperbolic component, and the solid node for a Seifert or I-fibered component. A \emph{JSJ-graph} of a handlebody-link is the JSJ-graph of its exterior. The JSJ-graph of any Hopf handlebody-link with $\chi=0$ is a solid node.

\begin{figure}[b]
	\begin{subfigure}{.3\linewidth}		
		\centering
		\begin{overpic}[scale=.12,percent]{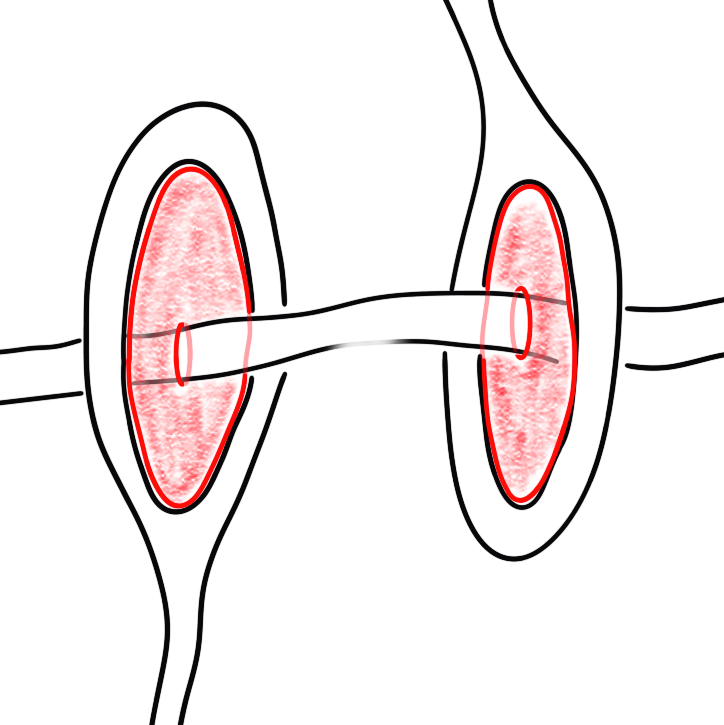}
			\put(22,60){\tiny $A_1^1$}
			\put(67,42.5){\tiny $A_1^2$}			 
			\put(46,51.5){\tiny $F_1^{.5}$}			 
		\end{overpic}
		\caption{}
		\label{fig:theta_F}
	\end{subfigure} 
	\begin{subfigure}{.3\linewidth}		
		\centering
		\begin{overpic}[scale=.12,percent]{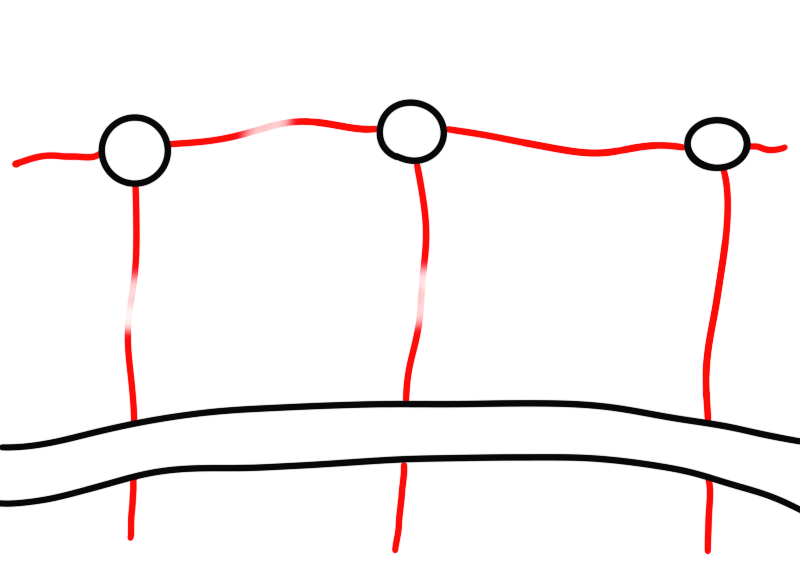}
		\put(15,32){\tiny $A_j^i$}
		\put(50,32){\tiny $A_j^{i+1}$}			 
		\put(26,55.8){\tiny $A_j^{.5}$}			 
		\put(30,41){\tiny $N_j^{i.5}$}			 
		\put(66,41){\tiny $N_j^{i+1.5}$}			 
		\put(35,15){\tiny $V$}			 
		\end{overpic}
		\caption{}
		\label{fig:S1_bundle}
	\end{subfigure} 
	\begin{subfigure}{.3\linewidth}		
		\centering
		\begin{overpic}[scale=.11,percent]{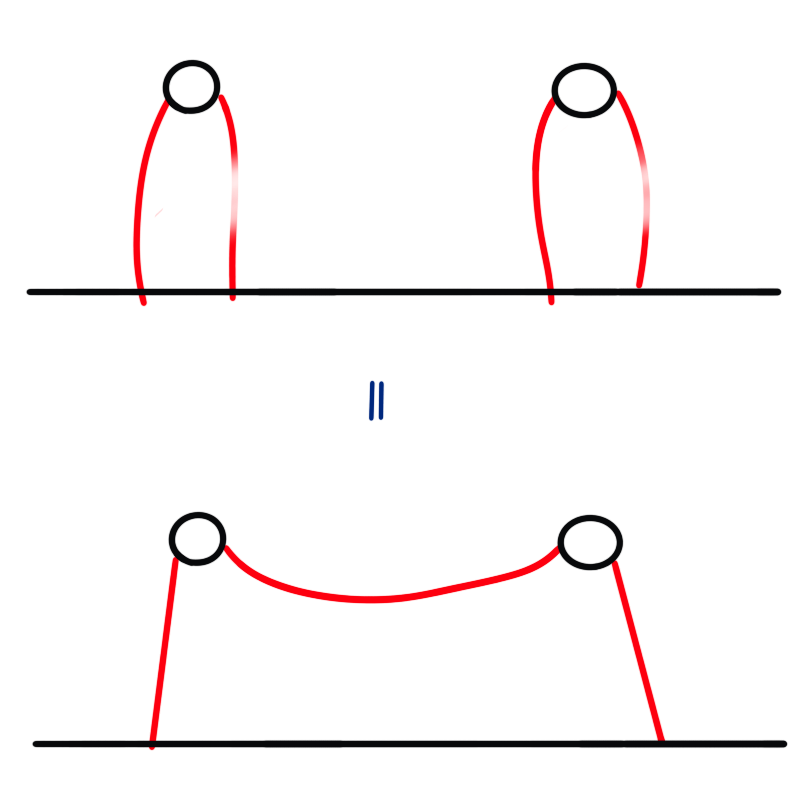}
		\put(18,74){\tiny $\rnbhd{A_j^i}$}
		\put(70,74){\tiny $\rnbhd{A_j^{i+1}}$}
		\put(46,15){\footnotesize $N_j^i$}
		\end{overpic}
		\caption{}
		\label{fig:nbhd_A_N}
	\end{subfigure}  
\end{figure}
\subsubsection*{Classification of JSJ-graphs}
To describe the JSJ-graph for Hopf handlebody-links with $\chi<0$, we say two triplets $\epsilon,\epsilon'$ 
are \emph{combinable} if there is a sequence of 
$\epsilon=\epsilon_1,\dots,\epsilon_k=\epsilon'$ of triplets such that $\epsilon_{i}\cap \epsilon_{i+1}\neq \emptyset$, for every $i=1,\dots,k-1$. This induces an equivalence relation on $\mathcal{E}$: Denote by $\bs_1,\cdots, \bs_k$ the equivalence classes containing only one triplet and by $\bm_1,\cdots, \bm_l$ the equivalence classes containing more than one triplets, where $\bm_j=\{\epsilon_j^1,\dots, \epsilon_j^{m_j}\}$ and $\epsilon_j^i=\{v^i_j,e^{i\spl}_j,e^{i\smi}_j\}$ with $e^{i\spl}_j=e^{i+1\smi}_j$, for every $1\leq i\leq m_j-1$. As examples, the pre-looping system in Fig.\ \ref{fig:theta} has $(k,l)=(0,1)$ with $m_1=2$, and the ones in Fig.\ \ref{fig:cyclic} and Fig.\ \ref{fig:dihedral} have $(k,l)=(3,0)$ and $(0,3)$, respectively.


Let $A_i$ be the annulus induced by the triplet in $\bs_i$, and $A_j^i$ the annulus induced by $\epsilon_j^i\in \bm_j$. Denote by $A_i^{\smi}, A_i^{\spl}$ (resp.\ $A_j^{i\smi},A_j^{i\spl}$) 
the frontier of a regular neighborhood of $\rnbhd{A_i}$ corresponding to $e_i^\smi, e_i^\spl$ (resp.\ $\rnbhd{A_j^i}$ corresponding to $e_j^{i\smi},e_j^{i\spl}$). Since $e^{i\spl}_j=e^{i+1\smi}_j$, the annuli $A_j^{i\spl}$, $A_j^{i+1\smi}$ cut off an annulus $F^{i.5}_j$ from $\partial V^\mathcal{E}$; see Fig.\ \ref{fig:theta_F}. 
 
A regular neighborhood $N_j^i$ of $A^{i\spl}_j\cup F^{i.5}_j\cup A^{i+1\smi}_j$ is an admissible $S^1$-bundle in $\Compl{V^\mathcal{E}}$ whose frontier contains three components: two of them can be identified with $A_j^i$ and 
$A_j^{i+1}$, while the other, denoted by $A_j^{i.5}$, is a type $3$ annulus. Since $N_j^i\cap N_j^{i+1}=A_j^{i+1}$, for every $i=1,\dots,  m_j -2$, the union 
\[N_j=N_j^1\cup\dots\cup N_j^{ m_j -1}\] 
is still an admissible $S^1$-bundle; see Fig.\ \ref{fig:S1_bundle}.
The frontier of $N_j$ is the union 
\[\mathcal{F}_j^\ast:=A_j^1\cup A_j^{1.5}\cup\dots\cup A_j^{ m_j -0.5}\cup A_j^{ m_j }.\]

In the case where $\mathcal{E}$ is adequate, the regular neighborhood $N(\Lambda_\mathcal{E})$ of $\Lambda_\mathcal{E}$ can be identified with $\rnbhd \Lambda=V\cup_i N(A_i)\cup_{i,j}N(A_j^i)$.
Since $N^i_j$ is an admissible $S^1$-bundle with frontiers $A_j^i\cup A_j^{i+1}\cup A_j^{i.5}$, we
have $V\cup_i N(A_j^i)= V\cup_i N_j^i$; see Fig.\ \ref{fig:nbhd_A_N}. Therefore  
\begin{equation}\label{eq:nbhd_identification}
\rnbhd{\Lambda_{\mathcal{E}}} = V\mathop{\cup}_i N(A_i)\mathop{\cup}_{i,j}N(A_j^i)=V\mathop{\cup}_i N(A_i)\mathop{\cup}_{i,j} N^i_j=V\mathop{\cup}_i N(A_i)\mathop{\cup}_{j} N_j.
\end{equation}
Consequently, we have the identifications: 
\begin{equation}\label{eq:simple_component}
\Compl{\Lambda_\mathcal{E}}=\overline{S^3-\rnbhd{\Lambda_\mathcal{E}}}=\overline{S^3-V\mathop{\cup}_i N(A_i)\mathop{\cup}_{i,j} N^i_j}=\overline{\Compl V-\mathop{\cup}_i N(A_i)\mathop{\cup}_{j} N_j},  
\end{equation}
and the frontier of $\Compl{\Lambda_\mathcal{E}}$ in $\Compl {V^\mathcal{E}}$, namely, the frontier of $\cup_i N(A_i)\cup_j N_j$, is   
\[\mathcal{F}_\mathcal{E}:=\mathop{\cup}_{i=1}^k( A_i^{\spl}\cup A_i^{\smi})\mathop{\cup}_{j=1}^l\mathcal F_j^*.\]
   
In the example of Fig.\ \ref{fig:theta}, we have $\mathcal{F}_\mathcal{E}=\mathcal{F}^\ast_1$, consisting of two type $2$ annuli and one type $3$ annulus. They cut $\Compl{V^\mathcal{E}}$ into a Seifert bundle 
$(N_1,\mathcal{F}_\mathcal{E})$ and an I-bundle $(\Compl{\Lambda_\mathcal{E}},\mathcal{F}_\mathcal{E})$ over a pair of pants. Thus, its JSJ-graph is a theta graph with two solid nodes; see \cite{Wan:24}. This shows $V^\mathcal{E}\simeq 4_1$ in \cite{IshKisMorSuz:12}, or $1\hopf_1$ in Table \ref{tab:hopf_handlebody_knots}, is simple. The converse is also true. 

\begin{lemma}\label{lm:theta}
$\chi=-1$ and $\mathcal{E}$ is simple if and only if $\Lambda_\mathcal{E}$ is a theta graph.
Particularly, $1\hopf_1$ is the only Hopf simple handlebody-link with $\chi=-1$, and its JSJ-graph is \raisebox{-.1cm}{\includegraphics[scale=.09]{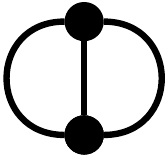}}.
\end{lemma}
\begin{proof}
	It may be assumed that $\Lambda$ has no unlooped bi-valent vertex, so
	if $\Lambda_\mathcal{E}$ is a theta graph,
	then $\mathcal{E}$ is the one in Fig.\ \ref{fig:theta}. The assertion thus follows from the observation preceding the lemma.
	If $\mathcal{E}$ is simple, then 
	by Lemma \ref{lm:simple_is_irreducible} and Theorem \ref{teo:irreducibility}, $\Lambda_\mathcal{E}$ is homeomorphic to the theta graph. By Lemma \ref{lm:simple_no_bivalent}, $\Lambda_\mathcal{E}$ is theta graph. 
\end{proof}

\begin{figure}[t]
	\begin{subfigure}{.3\linewidth}		
		\centering
		\begin{overpic}[scale=.09,percent]{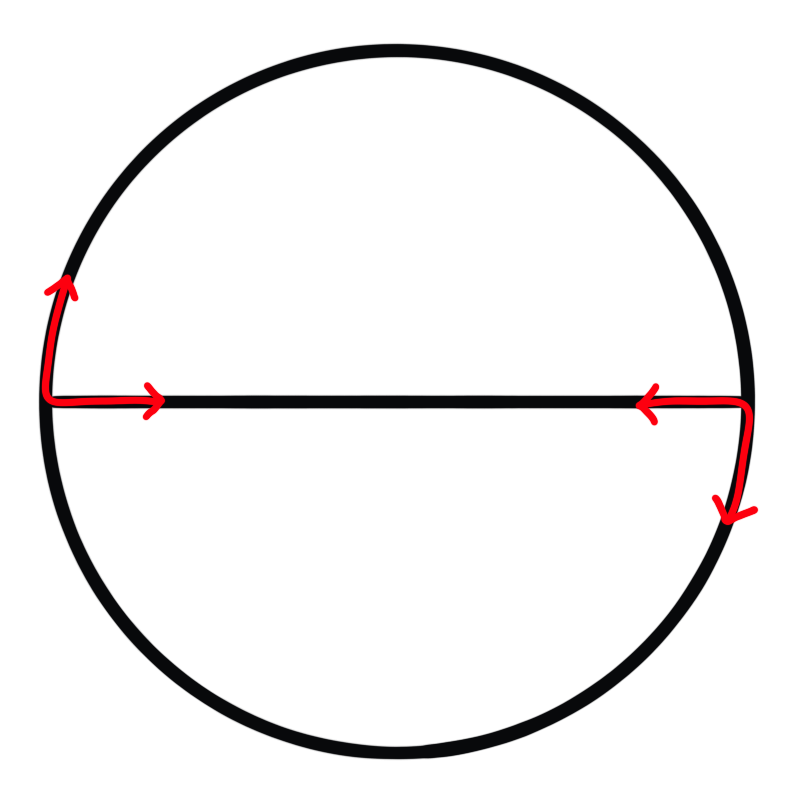}
		\end{overpic}
		\caption{$\theta$-graph.}
		\label{fig:theta}
	\end{subfigure} 
	\begin{subfigure}{.3\linewidth}
		\centering
		\begin{overpic}[scale=.09,percent]{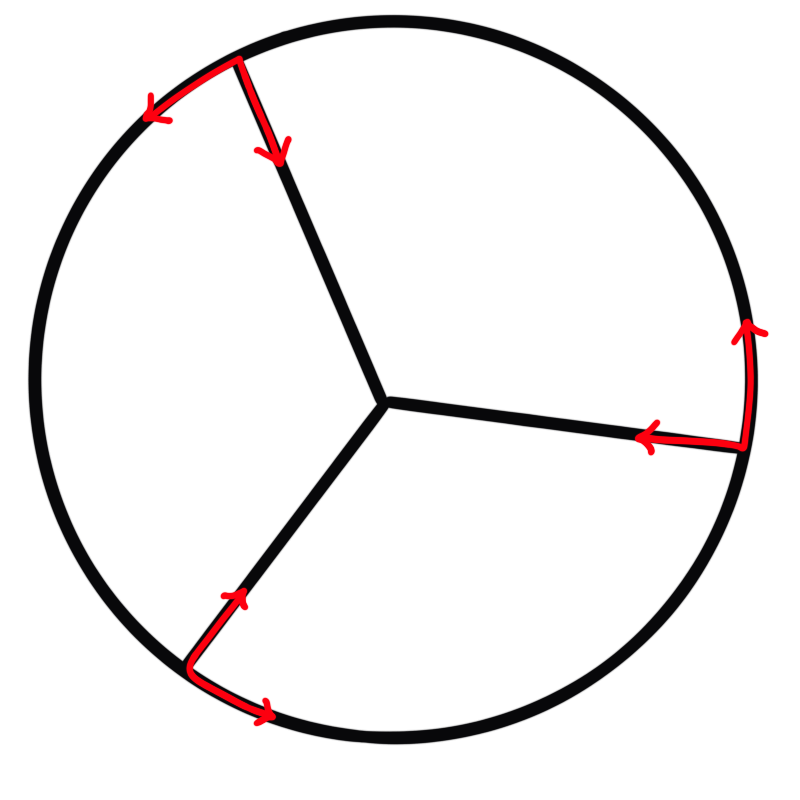}
		\end{overpic}
		\caption{$G=\mathbb Z_3$.}
		\label{fig:cyclic}
	\end{subfigure} 
	\begin{subfigure}{.3\linewidth}
		\centering
		\begin{overpic}[scale=.09,percent]{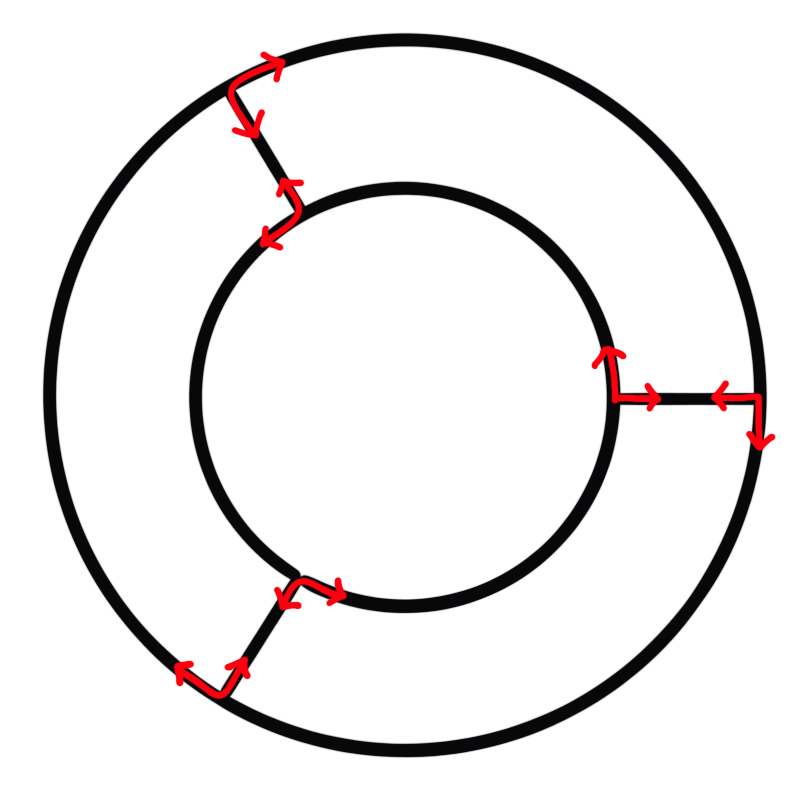}
		\end{overpic}
		\caption{$G=D_3$.}
		\label{fig:dihedral}
	\end{subfigure} 
	\caption{}
\end{figure}

In general, we consider the union of mutually non-isotopic annuli:
\[\mathcal{F}:=\mathop{\cup}_i A_i\mathop{\cup}_j \mathcal{F}_j^\ast,\]
which cuts $\Compl{V^\mathcal{E}}$ into $l+1$ components: $(\Compl{\Lambda_\mathcal{E}},\mathcal{F}_\mathcal{E})$, and $(N_j,\mathcal{F}_j^\ast)$, $j=1,\dots,l$.

\begin{lemma}\label{lm:char_surface}
Suppose $\chi<-1$ and $\mathcal{E}$ is simple. Then $\mathcal{F}$ is the characteristic surface of $\Compl{V^\mathcal{E}}$, and its JSJ-graph consists of one hollow node $w_0$ with
$k$ loops incident to it and $l$ solid nodes $w_j,j=1,\dots,l$ with $\vert m_j\vert+1$ multiple edges between $w_0,w_j$. 

More precisely, $w_0$ corresponds to $\Compl{\Lambda_\mathcal{E}}$ and $w_j$ to
$N_j$, for each $j$, and each $A_i$ corresponds to a loop and each 
component in $\mathcal{F}_j^\ast$ to an edge incident to $w_j$. For instance, the JSJ-graphs of $V^\mathcal{E}$ given by Fig.\ \ref{fig:cyclic} and Fig.\ \ref{fig:dihedral} are
\[ 
\raisebox{-.3cm}{\includegraphics[scale=.09]{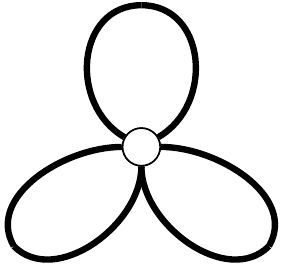}}\text{ and } \raisebox{-.3cm}{\includegraphics[scale=.1]{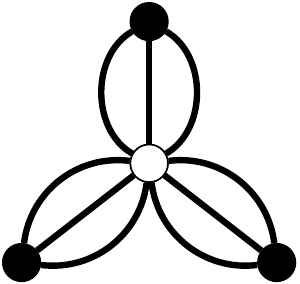}}, \text{ respectively. }
\] 
\end{lemma}
\begin{proof}
Observe that, with the identification \eqref{eq:nbhd_identification}, the cores of $A_i^\spl,A_i^\smi$, respectively, bound disks in $\rnbhd{\Lambda_\mathcal{E}}$ dual to $e_{i+},e_{i-}\in\epsilon_i$, and the cores of $A_j^1,A_j^{m_j}$ bound disks dual to $e_j^{1\smi},e_j^{ m_j  \spl}$, and the core of $A_j^{i.5}$ bounds a disk dual to 
$e_j^i=e_j^{i+1}$, where $j=1,\dots, m_j-1$. We may therefore assume the core of $\mathcal{F}_\mathcal{E}$ bounds $\mathcal{D}_\mathcal{E}$ and hence is $k_\mathcal{E}$. Now, since $\chi<-1$ and  $(\Compl{\Lambda_\mathcal{E}},k_\mathcal{E})$ is simple, we have $(\Compl{\Lambda_\mathcal{E}},\mathcal{F}_\mathcal{E})$ is hyperbolic.  
This, along with $(N_j,\mathcal{F}_j^\ast)$ being Siefert fibered, implies the condition \ref{itm:complete}. On the other hand, for every $i,j$, neither the union $N(A_i)\cup \Compl{\Lambda_\mathcal{E}}$ with its frontier nor the union $N_j\cup N(\mathcal{F}'_j)\cup \Compl{\Lambda_\mathcal{E}}$ with its frontier in $\Compl{V^\mathcal{E}}$ is fibered or hyperbolic, where $\mathcal{F}'_j$ is the union of some annuli in $\mathcal{A}_j^\ast$, so the condition \ref{itm:minimality} is satisfied. Thus, $\mathcal{F}$ is the characteristic surface.

The assertion concerning its JSJ-graph follows from the fact that the regular neighborhood $N(A_i)$ of $A_i$ meets $\Compl{\Lambda_\mathcal{E}}$ at $A_i^\spl\cup A_i^\smi$, and $N_j$ meets $\Compl{\Lambda_\mathcal{E}}$ at $\mathcal{F}_j^\ast$.
\end{proof} 


\begin{corollary}\label{cor:type_two}
If $\mathcal{E}$ is simple and $\chi\leq -1$, then the type $2$ annuli in $\Compl{V^\mathcal{E}}$ are precisely those induced by the triplets in $\mathcal{E}$, namely $\mathcal{A}_\mathcal{E}$; additionally, the longitudinal boundary components of $\mathcal{A}_\mathcal{E}$ are mutually non-homotopic.  
\end{corollary}
\begin{proof}
By Lemmas \ref{lm:theta} and \ref{lm:char_surface}, every type $2$ annulus $A$, not isotopic to $A_i$, $i=1,\dots, k$, can be isotoped into to $N_j$, for some $j$. 
On the other hand, the frontier of $N_j$ consists of the annuli: 
$A_j^1, A_j^{1.5}, \dots, A_j^{ m_j-0.5}, A_j^{ m_j }$, 
of which only $A_j^1,A_j^{ m_j }$ are of type $2$. In particular, the meridional components of $\partial A_j^1$ and $\partial A_j^{ m_j }$ are parallel in $\partial V^\mathcal{E}$, and cut off an annulus $F_j$ from $\partial V$. This implies the meridional component of $\partial A$ is in $F_j$, and hence isotopic to $A_j^i$, for some $i$. 
The second assertion follows from the first and the looping construction since vertical structures at different triplets can be chosen to be disjoint.
\end{proof}

Consider another trivial spatial graph $\Lambda'$ with a pre-looping system $\mathcal{E}'$, and define 
$\Lambda_{\mathcal{E}'}'$, $\mathcal{D}_{\mathcal{E}'}$ and $k_{\mathcal{E}'}$ as $\Lambda_{\mathcal{E}}$, $\mathcal{D}_{\mathcal{E}}$ and $k_{\mathcal{E}}$, respectively, with $(\Lambda,\mathcal{E})$ replaced by $(\Lambda',\mathcal{E}')$.   
   
\begin{theorem}\label{teo:equivalence_Lambda}
Suppose $\mathcal{E}$ is simple and $\chi\leq -1$. Then every homeomorphism $f:(S^3,V^\mathcal{E})\rightarrow (S^3,V^{\mathcal{E}'})$ induces a homeomorphism 
$(S^3,\Lambda_\mathcal{E},\mathcal{E})\rightarrow (S^3,\Lambda_{\mathcal{E}'}',\mathcal{E}')$.
Particularly, $(\Compl{\Lambda_\mathcal{E}},k_\mathcal{E})$, $(\Compl{\Lambda_{\mathcal{E}'}'},k_{\mathcal{E}'})$
are homeomorphic, and $V^{\mathcal{E}'}$ is simple. 
\end{theorem}
\begin{proof}
By Lemma \ref{lm:simple_is_irreducible} and $V^\mathcal{E}$ and hence $V^{\mathcal{E}'}$ are irreducible. In particular, $\mathcal{E},\mathcal{E}'$ are adequate by Theorem \ref{teo:irreducibility}, and hence
$\rnbhd{\Lambda}=\rnbhd{\Lambda_\mathcal{E}}$, $\rnbhd{\Lambda'}=\rnbhd{\Lambda_{\mathcal{E}'}'}$.

Suppose $\mathcal{E}=\{\epsilon_1,\dots, \epsilon_m\}$ and $A_i$ be the type $2$ annuli induced by $\epsilon_i$. Since $V^\mathcal{E}$ is simple, by Corollary \ref{cor:type_two}, the type $2$ annuli in $\Compl{V^\mathcal{E}}$ are precisely $A_1,\cdots,A_m$. Since $V^\mathcal{E},V^{\mathcal{E}'}$ are equivalent, the annuli $f(A_1),\cdots, f(A_m)$ are all the type $2$ annuli in $\Compl{V^{\mathcal{E}'}}$. Suppose one of them, say $A'$, is not induced from any triplet in $\mathcal{E}'$. Then, since $A'$ is disjoint from the other type $2$ annuli, $A'$ can be regarded as in $\Compl{\Lambda_{\mathcal{E}'}'}$ and disjoint from $k_{\mathcal{E}'}$. 

On the other hand, $\mathcal{D}_\mathcal{E'}$ cuts 
$N(\Lambda_{\mathcal{E}'}')$ into some $3$-balls, each meeting $\Lambda_{\mathcal{E}'}'$ at a cone. 
Since $\partial A'$ is disjoint from $k_{\mathcal{E}'}$, $\partial A'$ is contained in some of the $3$-balls. If both components of $\partial A'$ are in the same $3$-ball, then $A'$ induces a $2$-sphere $S$ meeting $\Lambda'_{\mathcal{E}'}$ at a vertex.
By the essentiality of $A'$, both components of $S^3-S$ meet $\Lambda_{\mathcal{E}'}'$, contradicting that
$\Lambda'_{\mathcal{E}'}$ is $2$-connected and has no loop by Theorem \ref{teo:irreducibility}. Therefore components of $\partial A'$ are in different components, and $A'$ induces a $2$-sphere $S$ meeting $\Lambda_{\mathcal{E}'}$ at two vertices. Since 
$A'$ is of type $2$, exactly one of the two vertices is in some $\epsilon''\in\mathcal{E}'$ and the two edges in $\epsilon''$ are separated by $S$. Let $A''$ be the type $2$ induced by $\epsilon''$. Then the longitudinal components of $\partial A'', \partial A'$ are parallel in $V^{\mathcal{E}'}$, contradicting Corollary \ref{cor:type_two}.     

Consequently, $f(A_1),\cdots, f(A_m)$ are all induced from triplets in $\mathcal{E}'$. Thus $f$ induces a bijection from $\mathcal{E}$ to $\mathcal{E}'$. We may therefore assume that $f(k_\mathcal{E})=k_{\mathcal{E}'}$ and   $f(\mathcal{D}_\mathcal{E})=\mathcal{D}_{\mathcal{E}'}$. Since $\mathcal{D}_\mathcal{E}, \mathcal{D}_{\mathcal{E}'}$ are dual to $\Lambda_\mathcal{E}, \Lambda_{\mathcal{E}'}'$, $f$ can be further isotoped so that $f(\Lambda_\mathcal{E})=\Lambda'_{\mathcal{E}'}$. The assertion thus follows.
\end{proof}

\subsection*{Proof of Theorem \ref{teo:simpleness}} \ref{itm:one_simple}$\Leftarrow$\ref{itm:other_simple} is the definition, and \ref{itm:one_simple}$\Rightarrow$\ref{itm:other_simple} follows from Theorem \ref{teo:equivalence_Lambda}.  
To see \ref{itm:other_simple}$\Rightarrow$\ref{itm:graph_simple}, it suffices to show that, if $\Lambda_\mathcal{E}$ is not a theta graph, then $\Lambda_\mathcal{E}$ is $3$-connected and simplicial. 
Suppose otherwise. Then, by Lemma \ref{lm:simple_is_irreducible} and Theorem \ref{teo:irreducibility}, $\Lambda_\mathcal{E}$ either has vertex-connectivity $2$ or contains multiple edges.

In particular, there is a $2$-sphere $S$ in $S^3$ meeting $\Lambda_\mathcal{E}$ at two vertices with every component in $S^3-S$ containing at least two edges or one $n$-valent vertex with $n>2$ by Lemma \ref{lm:simple_no_bivalent}. In either case, the annulus $S\cap\Compl{\Lambda_\mathcal{E}}$ is essential in $\Compl{\Lambda_\mathcal{E}}-k_\mathcal{E}$ 
and cannot be isotoped into a regular neighborhood of $k_\mathcal{E}\subset \partial \Compl{\Lambda_\mathcal{E}}$,
contradicting $\mathcal{E}$ is simple.

Consider now \ref{itm:graph_simple}$\Rightarrow$\ref{itm:other_simple}. 
The case $\Lambda_\mathcal{E}$ is a theta graph follows from Lemma \ref{lm:theta}. 
Suppose $\Lambda_\mathcal{E}$ is not a theta graph, and $\mathcal{E}$ is not simple. Then $\Compl{\Lambda_\mathcal{E}}-k_\mathcal{E}$ admits an essential disk or an essential annulus not isotopic into a regular neighborhood of $k_\mathcal{E}$.
Since $\mathcal{E}$ is adequate, $\mathcal{D}_\mathcal{E}$ cuts $\rnbhd{\Lambda_\mathcal{E}}$ into some $3$-balls $B_1,\dots,B_n$, each meeting $\Lambda_\mathcal{E}$ at a cone.
If $\Compl{\Lambda_\mathcal{E}}-k_\mathcal{E}$ 
admits an essential disk $D$, then $\partial D$ is in a $3$-ball $B_i$, for some $i$. Thus $D$ induces a $2$-sphere meeting $\Lambda_\mathcal{E}$ at the vertex in $B_i$ and cutting $\Lambda_\mathcal{E}$ into $2$ components. This implies $\Lambda_\mathcal{E}$ contains a loop or is $2$-connected, a contradiction. 

If $\Compl{\Lambda_\mathcal{E}}-k_\mathcal{E}$ 
admits an essential annulus $A$ not $\partial$-parallel to a regular neighborhood of any component in $k_\mathcal{E}$, then there are two cases: Components of $\partial A$ are in the same $3$-ball $B_i$ or in different $3$-balls $B_i\neq B_j$. In the former, $A$ induces a pinched torus $T$ with $T\cap \Lambda_\mathcal{E}$ being the pinched point of $T$ and a vertex in $\Lambda_\mathcal{E}$. Since $\Lambda_\mathcal{E}$ is contained in a $2$-sphere and $A$ is essential, $T$ cuts 
$\Lambda_\mathcal{E}$ into $2$ components. This implies 
$\Lambda_\mathcal{E}$ contains a loop or is $2$-connected, a contradiction. In the latter, $A$ induces a $2$-sphere $S$ meeting $\Lambda_\mathcal{E}$ at two vertices. By the essentiality of $A$, 
each component of $\Lambda_\mathcal{E}\cap (S^3-S)$ contains
at least one vertex or two edges. In particular, $\Lambda_\mathcal{E}$ is not a theta graph, but either has some multiple edges or has vertex-connectivity $2$ or less, a contradiction.   
\qed
\begin{remark} 
The above $2$-sphere $S$ induced by the essential disk or annulus is a reminiscence of the type I, II, III spheres in \cite{ChoKod:13}.
\end{remark}

\subsection{Symmetry classification}\label{subsec:symm}
Here we classify the symmetry of simple Hopf handlebody-links with $\chi<0$. Throughout the section, we assume $V$, and hence $\mathcal{E}$, are simple. By Theorem \ref{teo:simpleness} and Lemma \ref{lm:theta}, either $V^\mathcal{E}$ is $1\hopf_1$ or $\Lambda_\mathcal{E}$ is 
simplicial and $3$-connected. The symmetry of $1\hopf_1$ has been classified.  

\begin{lemma}[{\cite[Theorem $1.7$ and the pragraph thereafter]{Wan:24}}]\label{lm:symmetry_fourone}
The positive symmetry group of $1\hopf_1$ is $\mathbb Z_2$, and its symmetry group is $\mathbb{Z}_2\times \mathbb{Z}_2$. 
\end{lemma}  
In view of Lemma \ref{lm:symmetry_fourone}, we assume $\Lambda_\mathcal{E}$ is simplicial and $3$-connected. 
We first recall two known results.
Consider a genus $g>1$ handlebody-knot $U$ and a union $\mathcal{D}$ of some disjoint, mutually non-parallel essential disks in $U$ that cuts $U$ into some $3$-balls. Denote by $\Gamma_\mathcal{D}$ the spine of $U$ dual to $\mathcal{D}$. Then the following follows from the Alexander trick and \cite{Hat:99}; see also \cite[Lemma $2.2$]{ChoKod:13}, \cite[Section $2.1$]{Kod:15}. 
\begin{lemma}\label{lm:diagram_isomorphisms}
The restriction of maps induce a diagram of isomorphisms:
\begin{equation*}
\pmcg{\sphere,\Gamma_\mathcal{D}}\leftarrow  \pmcg{\sphere,U,\mathcal{D},\Gamma_\mathcal{D}}\rightarrow  \pmcg{\sphere,U,\mathcal{D}} 
\rightarrow \pmcg{\sphere,U,\partial \mathcal{D}}.
\end{equation*}
\cout{
\begin{multline*}
\pmcg{\sphere,\Gamma_\mathcal{D}}\xleftarrow{\Psi_1} \pmcg{\sphere,U,\mathcal{D},\Gamma_\mathcal{D}}\xrightarrow{\Psi_2} \pmcg{\sphere,U,\mathcal{D}}\\
\xrightarrow{\Psi_3}\pmcg{\sphere,U,\partial \mathcal{D}}.
\end{multline*}
}
\end{lemma}
\cout{
\begin{proof}
The surjectivity of $\Psi_1$ follows from the uniqueness of regular neighborhood of $\Gamma_\mathcal{D}$. Given $f\in\pAut{\sphere,U,\mathcal{D},\Gamma_\mathcal{D}}$, if $h_t$ is an isotopy of $f$ to $\id$ that preserves $\Gamma_\mathcal{D}$, by \cite{Hat:99}, there is an isotopy $h_t'$ between $f,\id$ that preserves also $U,\mathcal{D}$. It follows from the Alexander trick that $\Psi_2$ is an isomorphism. It is clear that $\Psi_3$ is surjective and its injectivity follows from \cite{Hat:99}.    
\end{proof}
} 
Let $S$ be an essential surface in $\Compl U$.
Then \cite{Hat:99} gives us the following. 
\begin{lemma}\label{lm:injection}
The forgetful homomorphism
\[
\pmcg{\sphere,U,\rnbhd{S}}
\simeq \pmcg{\sphere,U,S}\rightarrow \pmcg{\sphere,U}\]
is injective.
\end{lemma}
\cout{
\begin{proof}
Corollary \ref{cor:type_two} asserts that $\mathcal{A}_\mathcal{E}$ is the union of all type $2$ annuli and hence every $f\in\pAut{\sphere,V^\mathcal{E}}$ can be isotoped so it preserves $\mathcal{A}_\mathcal{E}$. The surjectivity thus follows. The injectivity follows \cite{Hat:99}. 
\end{proof}
}
Since $\mathcal{A}_\mathcal{E}$ is the union of all type $2$ annuli in $\Compl{V^\mathcal{E}}$ by Corollary \ref{cor:type_two}, every self-homeomorphism 
$f$ of $(\sphere,V^\mathcal{E})$ can be isotoped so it preserves $\mathcal{A}_\mathcal{E}$. This, together with Lemma \ref{lm:injection}, implies the isomorphism
\begin{equation}\label{eq:isomorphism_1}
\pmcg{\sphere,V^\mathcal{E},\rnbhd{\mathcal{A_\mathcal{E}}}}
\simeq \pmcg{\sphere,V^\mathcal{E},\mathcal{A}_\mathcal{E}}\rightarrow \pmcg{\sphere,V^\mathcal{E}}. 
\end{equation}
Identify $\rnbhd{\mathcal{A}_\mathcal{E}}$ with a product $\mathcal{A}_\mathcal{E}\times [0,1]$, let $c\mathcal{A}_\mathcal{E}$ be the core of 
$\mathcal{A}_\mathcal{E}$. We denote by  $\mathcal{A}^\perp\subset \rnbhd{\mathcal{A}_\mathcal{E}}$ the annulus given by the product $c\mathcal{A}_\mathcal{E}\times [0,1]$. In particular, $\mathcal{A}^\perp$ is a union of annuli properly embedded in 
$\rnbhd{\Lambda_\mathcal{E}}$ with $\rnbhd{\mathcal{A}_\mathcal{E}}$ its regular neighborhood. Therefore, we have the identification: 
\begin{equation}
\pmcg{\sphere,V^\mathcal{E},\rnbhd{\mathcal{A}_\mathcal{E}}}=\pmcg{\sphere,\rnbhd{\Lambda_\mathcal{E}}, \rnbhd{\mathcal{A}^\perp}}.
\end{equation}
Since $\partial\mathcal{A}^\perp$ bounds disks dual to $\Lambda_\mathcal{E}$, by Lemma \ref{lm:diagram_isomorphisms}, we have the homomorphism  
\[
\Psi:\pmcg{\sphere,\rnbhd{\Lambda_\mathcal{E}},\rnbhd{\mathcal{A}^\perp}}\rightarrow
\pmcg{\sphere,\rnbhd{\Lambda_\mathcal{E}},\partial \mathcal{A}^\perp}\rightarrow
\pmcg{\sphere,\Lambda_\mathcal{E}}\]   
Now, for each annulus $A^\perp\subset \mathcal{A}^\perp$, there is a unique 
triplet $\epsilon\in \mathcal{E}$ such that components of $\partial A^\perp$ bound disks dual to edges in $\epsilon$, so $\Psi$ factors through the subgroup $\pmcg{\sphere,\Lambda_\mathcal{E},\mathcal{E}}$ of $\pmcg{\sphere,\Lambda_\mathcal{E}}$.

\begin{lemma}\label{lm:isomorphism_2}
$\Psi:\pmcg{\sphere,\rnbhd{\Lambda_\mathcal{E}},\rnbhd{\mathcal{A}^\perp}}\rightarrow \pmcg{\sphere,\Lambda_\mathcal{E},\mathcal{E}}$ is an isomorphism.  
\end{lemma}
\begin{proof}
For the injectivity, we let $f\in \pAut{\sphere,\rnbhd{\Lambda_\mathcal{E}},\rnbhd{\mathcal{A}^\perp},\Lambda_\mathcal{E},\mathcal{E}}$ and $f$ is isotopic to $\id$ through an isotopy $h_t$ preserving $\Lambda_\mathcal{E}$ and $\mathcal{E}$. By Lemma \ref{lm:diagram_isomorphisms}, it may be assumed $h_t$ also preserves $\rnbhd{\Lambda_\mathcal{E}}$. Since $h_t$ sends triplets to triplets, for each $t$, 
by \cite{Hat:99}, we can homotopy the path $h_t$ in $\pAut{\sphere,\rnbhd{\Lambda_\mathcal{E}}}$, relative to $f\cup \id$, to an isotopy that preserves $\rnbhd{\mathcal{A}^\perp}$. This proves the injectivity.
 
To see the surjectivity,   
we consider $f\in \pAut{\sphere,\Lambda_\mathcal{E},\mathcal{E}}$. Since $\Lambda_\mathcal{E}$ is $3$-connected and simplicial, $f$ can be isotoped so it preserves the $2$-sphere $S_\ast$ containing $\Lambda_\mathcal{E}$. 
Now, recall from Section \ref{subsec:Lambda_E_to_V_E} that $\mathcal{B}$ is the cone neighborhood of vertices in $\Lambda_\mathcal{E}$ and $\gamma$ is the union of arcs in $S_\ast\cap \mathcal{B}$, each joining the intersection of the edges in a triplet with $\partial \mathcal{B}$.  
Therefore, we may further isotope $f$ so it preserves $\mathcal{B}$ and $\gamma$.
In particular, $f$ preserves $\Lambda':=(\Lambda_\mathcal{E}-\mathcal{B})\cup \gamma$.
Isotope $f$ so it preserves a regular neighborhood $\rnbhd{\gamma}$ of $\gamma$ in $\mathcal{B}$, and hence preserves its exterior $\mathcal{W}$ in $\mathcal{B}$. Now $V^\mathcal{E}$ is a regular neighborhood of $\Lambda'\cup \mathcal{W}$, so $f$ can be isotoped so it preserves $V^\mathcal{E}$. This shows the surjectivity. 
\end{proof} 

Consider the forgetful homomorphism
\[
\pi: \Aut{S_\ast,\Lambda_\mathcal{E},\mathcal{E}} \rightarrow \mcg{\Lambda_\mathcal{E},\mathcal{E}}
\]
\begin{lemma}\label{lm:realization}
There is a subgroup $G<\Aut{S_\ast,\Lambda_\mathcal{E},\mathcal{E}}$ such that $\pi$ restricts to an isomorphism on $G$.
\end{lemma}
\begin{proof}
Since $\Lambda_\mathcal{E}$ is simplicial, every mapping class $x$ in
$\mcg{\Lambda_\mathcal{E},\mathcal{E}}$ is determined by its restriction on the vertices of $\Lambda_\mathcal{E}$.  

Now, $\Lambda_\mathcal{E}$ induces a regular CW-complex structure $K$ on the $2$-sphere $S_\ast$; namely, the characteristic map of every $1$- or $2$-cell is an embedding. 
Given a $1$- or $2$-cell $\sigma$ in $K$, denote by $j_\sigma:C\rightarrow K$ the characteristic map of $\sigma$ with $j_\sigma(\mathring{C})=\sigma$, where $C$ is the unit interval $I$ or the unit disk $D$. 
Now, for each $x=[f]\in\mcg{\Lambda_\mathcal{E},\mathcal{E}}$, and define $g_x\in \Aut{S,\Lambda_\mathcal{E},\mathcal{E}}$ so that
$g_x(v)=f(v)$, for every vertex $v\in \Lambda_\mathcal{E}$, and if $f(e)=e'$, for a $1$-cell $e$ in $\Lambda_\mathcal{E}$, then $j_{e'}^{-1} g_x j_e:I\rightarrow I$ is linear, and if $f(d)=d'$, for a $2$-cell $d$, then $j_{d'}^{-1} g_x j_d:D\rightarrow D$ is the homeomorphism obtained by the Alexander trick. Since $g_x$ is determined by the restriction of $f$ on the vertices, $g_x$ depends only on $x$, and the construction guarantees that $g_x g_y=g_{xy}$, for every $x,y\in\mcg{\Lambda_\mathcal{E},\mathcal{E}}$. 
Thus $G:=\{g_x\mid x\in \mcg{\Lambda_\mathcal{E},\mathcal{E}}\}$.
\end{proof}

\begin{lemma}\label{lm:isomorphism_3}
The forgetful homomorphisms
\[\Theta:   \psym{\Lambda_\mathcal{E},\mathcal{E}} \rightarrow \mcg{\Lambda_\mathcal{E},\mathcal{E}}\]
is an isomorphism.
\end{lemma}
\begin{proof}
The surjectivity follows from Lemma \ref{lm:realization} since every element in $G$ can be extended to an element in $\pAut{\sphere,\Lambda_\mathcal{E},\mathcal{E}}$ by the Alexander trick.

For the injectivity, we consider $x=[f]\in \psym{\Lambda_\mathcal{E},\mathcal{E}}$ with $\Theta(x)=1$. Let $D_1,\dots, D_k$ be the disks in $\mathcal{D}_\mathcal{E}$. 
Since $\mathcal{D}_\mathcal{E}$ is dual to $\Lambda_\mathcal{E}$. It may be assumed that $f(N(\Lambda_\mathcal{E}))=N(\Lambda_\mathcal{E})$ and $f(D_i)=D_i$, for every $i$. In particular, $x$ is in $\psym{N(\Lambda_\mathcal{E}),D_1,\dots, D_k}$. Consider the composition of the homomorphisms 
\begin{multline}\label{eq:composition}
\psym{N(\Lambda_\mathcal{E}),D_1,\dots,D_k}
\xrightarrow{j_1} 
\pmcg{E(\Lambda_\mathcal{E}),\partial D_1,\dots, \partial D_k}\\
\xrightarrow{j_2}
\pmcg{\partial E(\Lambda_\mathcal{E}),\partial D_1,\dots,\partial D_k} 
\xrightarrow{j_3}
P\pmcg{\partial E(\Lambda_\mathcal{E})-\partial \mathcal{D}_\mathcal{E}}.
\end{multline}
Since $D_{\mathcal{E}}$ cuts $N(\Lambda_\mathcal{E})$ into some $3$-balls, the last group in \eqref{eq:composition} is the pure mapping class group of some punctured spheres, and hence is torsion free by the Birman exact sequence. Note also $j_1,j_2$ are injective since the mapping class group $\mcg{W,\mathrm{rel}\ \partial W}$ of a handlebody $W$ is trivial. In addition, $j_3$ being the cut homomorphism, its kernel is a free group generated by Dehn twists along $\partial \mathcal{D}_\mathcal{E}$.

Now by the assumption, $(E(\Lambda_\mathcal{E}),k_\mathcal{E})$ is simple, so $\pmcg{E(\Lambda_\mathcal{E}),\partial D_1,\dots, \partial D_k}$ is finite, and hence $x$ is of finite order.
Given $P\pmcg{\partial E(\Lambda_\mathcal{E})-\partial \mathcal{D}_\mathcal{E}}$ is torsion free, $j_3j_2j_1(x)=1$, so $j_2j_1(x)$ is in the kernel of $j_3$. The kernel being free implies 
$j_2j_1(x)=1$. Since $j_2,j_1$ are injective, we conclude $x=1$, so 
$\Theta$ is injective.  
\end{proof}

\begin{corollary}\label{cor:finiteness}
$\psym{V^\mathcal{E}}\simeq \mcg{\Lambda_\mathcal{E},\mathcal{E}}$ is a finite subgroup of $O(3)$ and is Nielsen realizable.
\end{corollary}
\begin{proof}
The isomorphism follows from \eqref{eq:isomorphism_1} and Lemmas \ref{lm:isomorphism_2} and \ref{lm:isomorphism_3}. 
Lemma \ref{lm:realization} implies $\psym{V^\mathcal{E}}$ acts faithfully on the $2$-sphere $S_\ast$. The first assertion thus follows from the geometrization of surfaces.  By the Alexander trick, the subgroup $G$ in Lemma \ref{lm:realization} can be lifted to a subgroup $G'<\pAut{\sphere,S_\ast,\Lambda_\mathcal{E},\mathcal{E}}$. 
The second assertion then follows from the first.    
\end{proof}

\subsection*{Proof of Theorem \ref{teo:classification}}
The case $\Lambda_\mathcal{E}$ is a theta graph, namely, $V^\mathcal{E}$ equivalent to $1\hopf_1$, follows from Lemma \ref{lm:symmetry_fourone}.
Suppose $\Lambda_\mathcal{E}$ is simplicial and $3$-connected, and consider the reflection $r\in \Aut{S^3,V^\mathcal{E}}$  
about the $2$-sphere $S_\ast$ containing $\Lambda_\mathcal{E}$. For every $f\in \pAut{S^3,V^\mathcal{E}}$, the mapping classes $[r^{-1}\circ f\circ r]$ and $[f]$ have the same image in $\mcg{\Lambda_\mathcal{E},\mathcal{E}}$, so by Corollary \ref{cor:finiteness}, $[r]$ commutes with every mapping class in 
$\pmcg{S^3,V^\mathcal{E}}$. Theorem \ref{teo:classification}\ref{itm:product} thus follows from the short exact sequence 
\[0\rightarrow \pmcg{S^3,V^\mathcal{E}}\rightarrow \mcg{S^3,V^\mathcal{E}}\xrightarrow{\pi} \mathbb{Z}_2\rightarrow 0.\] 
%
Theorem \ref{teo:classification}\ref{itm:finiteness}
is Corollary \ref{cor:finiteness}.
To prove Theorem \ref{teo:classification}\ref{itm:realization},
we first observe that, if $G$ is the subgroup in Lemma \ref{lm:realization},
then the quotient $\Lambda_\mathcal{E}/G$ is a graph in the orbifold $S_\ast/G$ with some triplets given by $\mathcal{E}/G$.
Conversely, given a finite subgroup $G<O(3)$ and a graph $\bar\Lambda$ in the orbifold $S_\ast/G$ with some triplets $\bar{\mathcal{E}}$. The lifting of $(\bar\Lambda,\bar{\mathcal{E}})$ on $S_\ast$ gives a trivial spatial graph with a pre-looping system. 
Now, the orbifold $S_\ast/G$ is classified into four types (see Fig.\ \ref{fig:orbifold}): 
\begin{enumerate}[label=(\roman*)]
\item\label{itm:sphere_two_cone} a $2$-sphere $S_n$ with two cone points of order $n$;
\item\label{itm:disk_two_cone} the quotient $D_n$ of $S_n$ by the reflection in a circle containing the cone points.
\item\label{itm:sphere_three_cone} a $2$-sphere $S_{p,q,r}$ with three cone points of order $p,q,r$ and $\frac{1}{p}+\frac{1}{q}+\frac{1}{r}>1$. 
\item\label{itm:disk_three_cone} the quotient $D_{p,q,r}$ of $S_{p,q,r}$ by the reflection in a circle containing the cone points. 
\end{enumerate}
Note that $G<SO(3)$ if and only if \ref{itm:sphere_two_cone} or \ref{itm:sphere_three_cone} occurs. 
In Figs.\ \ref{fig:sphere_two_cone}, \ref{fig:disk_two_cone}, \ref{fig:sphere_three_cone}, and \ref{fig:disk_three_cone}, each has a graph $\bar\Lambda$ in $S_\ast/G$ with some triplets,  and corresponds to the cases \ref{itm:sphere_two_cone}, \ref{itm:disk_two_cone}, \ref{itm:sphere_three_cone}, and  \ref{itm:disk_three_cone}, respectively.  
Its lifting $\Lambda$ on $S_\ast$ is a spatial graph with a pre-looping system $\mathcal{E}$, and $G$ is the positive symmetry group of its fat looping; see Fig.\ \ref{fig:dihedral_reflection} for the case \ref{itm:disk_three_cone} with $(p,q,r)=(2,2,3)$, namely $G=D_3\times \mathbb{Z}_2$, and see Figs.\ \ref{fig:A_4}, \ref{fig:A_4_hl} for the case \ref{itm:sphere_three_cone} with $(p,q,r)=(2,3,3)$, namely $G=A_4$; see also Figs.\ \ref{fig:cyclic} and \ref{fig:dihedral}.

\begin{figure}[t]
	\begin{subfigure}{.24\linewidth}
	\centering
		\begin{overpic}[scale=.1,percent]{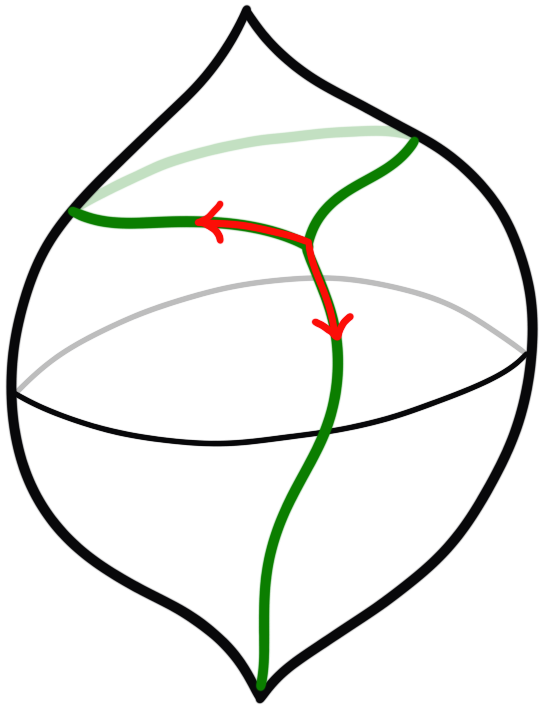}
		\put(36,97){$n$}
		\put(26.5,-2.8){$n$}
 		\end{overpic}
		\caption{ }
	\label{fig:sphere_two_cone}
	\end{subfigure}  
	\begin{subfigure}{.24\linewidth}
		\centering
		\begin{overpic}[scale=.1,percent]{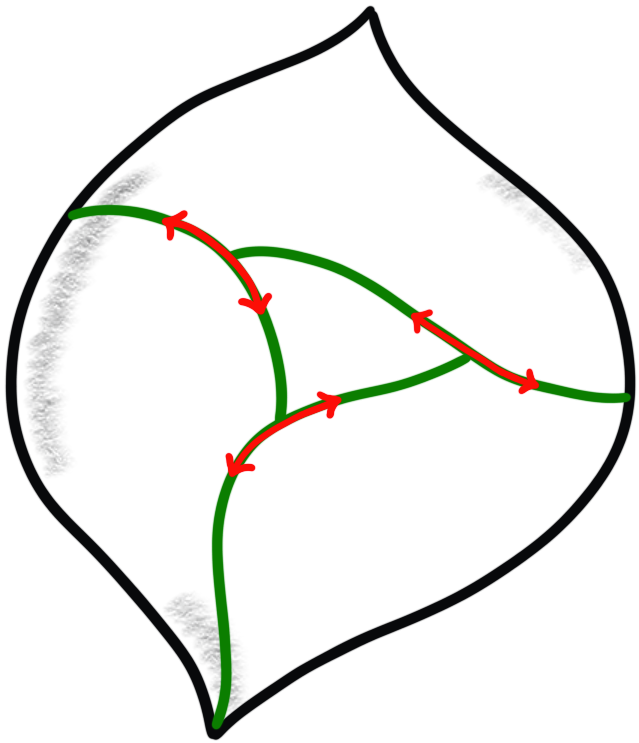}
		\put(51,96){$2n$}
		\put(13,-3){$2n$}
 		\end{overpic}
		\caption{}
		\label{fig:disk_two_cone}
	\end{subfigure}  
	\begin{subfigure}{.24\linewidth}
	\centering
	\begin{overpic}[scale=.1,percent]{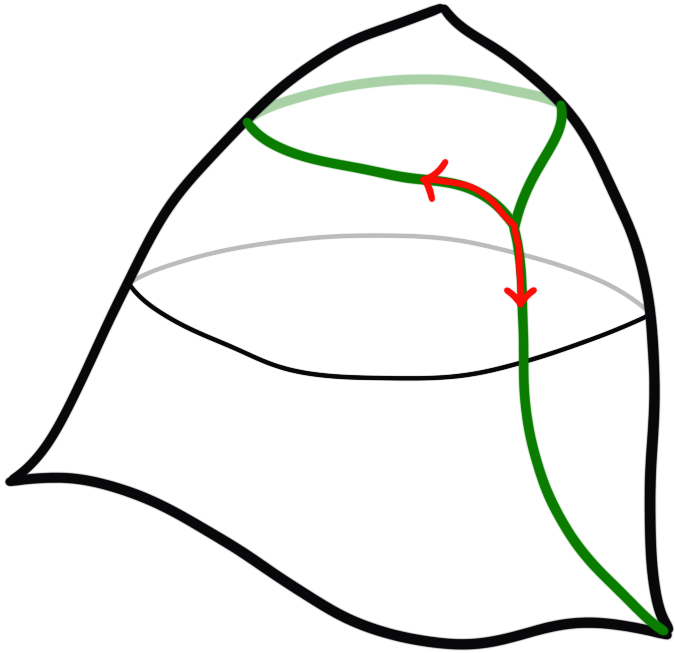}
	\put(-5.2,25){$q$}
	\put(100.2,2){$p$}
	\put(69,94){$r$}
 	\end{overpic}
	\caption{}
	\label{fig:sphere_three_cone}
    \end{subfigure}  
	\begin{subfigure}{.24\linewidth}
	\centering
	\begin{overpic}[scale=.1,percent]{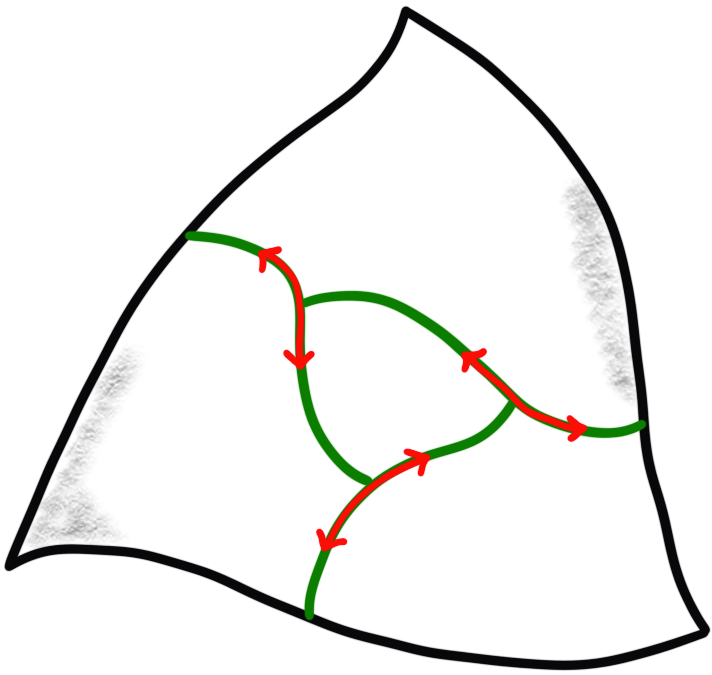}
	\put(-6.5,6.9){$2q$}
	\put(99.3,3.5){$2p$}
	\put(59,93){$2r$}	
	\end{overpic}
	\caption{}
	\label{fig:disk_three_cone}
    \end{subfigure}  
    \caption{}
    \label{fig:orbifold}
\end{figure}

\begin{figure}[t]
	\begin{subfigure}{.33\linewidth}
		\centering		\begin{overpic}[scale=.13,percent]{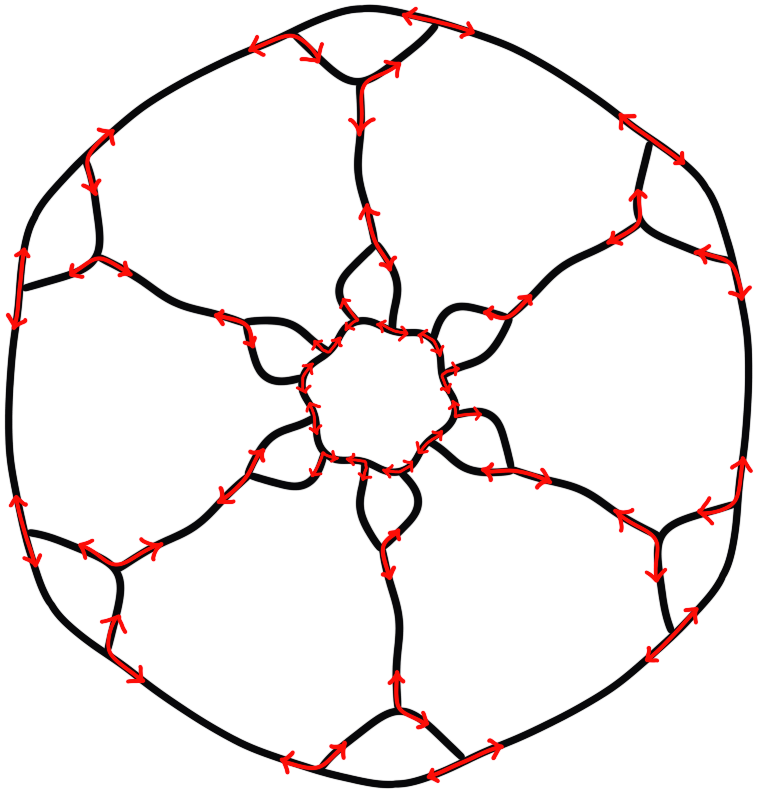}
  		\end{overpic}
		\caption{$(\Lambda,\mathcal{E})$; $G=D_3\times \mathbb{Z}_2$.}
		\label{fig:dihedral_reflection}
	\end{subfigure}  
	\begin{subfigure}{.32\linewidth}
		\centering
		\begin{overpic}[scale=.11,percent]{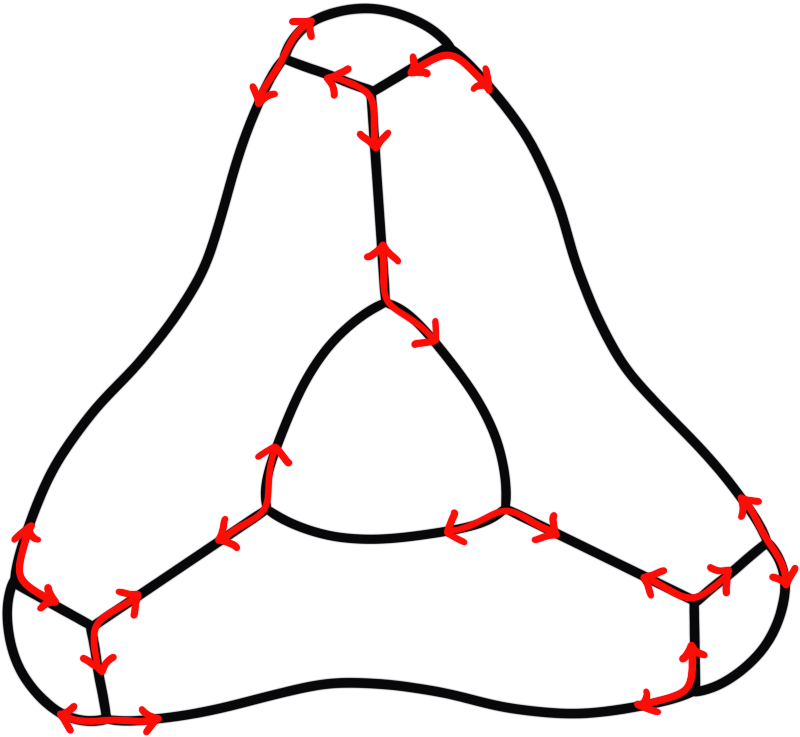} 
		\end{overpic}
		\caption{$(\Lambda,\mathcal{E})$; $G=A_4$.}
		\label{fig:A_4}
	\end{subfigure}  
	\begin{subfigure}{.32\linewidth}
		\centering
		\begin{overpic}[scale=.11,percent]{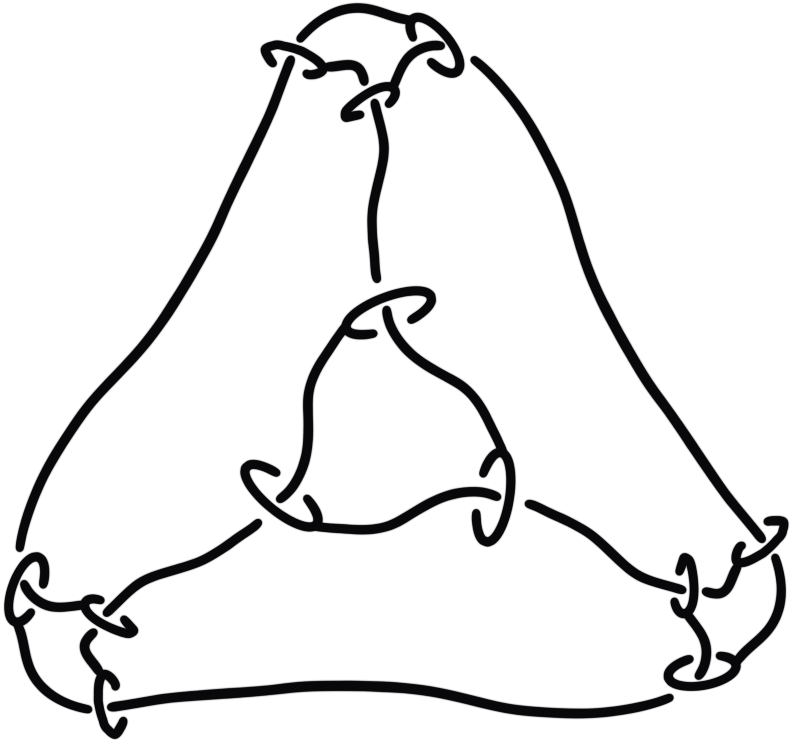}
		\end{overpic}
		\caption{$V^\mathcal{E}$; $G=A_4$.}
		\label{fig:A_4_hl}
	\end{subfigure}
	\caption{Handlebody-link with $G$-symmetry.}  
	\end{figure}

\section{Enumeration}\label{sec:enumeration}
Here we enumerate simple handlebody-knots with Euler characteristic $\chi\geq -4$. Now, there is no Hopf handlebody-knot with $\chi=0$, and by Lemma \ref{lm:theta}, $1\hopf_1$ is the only simple handlebody-knot with $\chi=-1$.  
For the case $\chi\leq -2$, Theorem \ref{teo:simpleness} implies $\Lambda_\mathcal{E}$  are $3$-connected and simplicial. Therefore, hereinafter, we assume $\Lambda$ is $3$-connected and simplicial. 
Recall that $\Lambda$ is contained in the $2$-sphere $S_\ast\subset S^3$. 
The closure $\sigma$ of a component of $S_\ast-\Lambda$ is called a \emph{closed $2$-cell}, or simply a \emph{cell}, of $\Lambda\subset S_\ast$. An edge of $\Lambda$ in $\partial \sigma$ is called a \emph{side} of $\sigma$, and $\sigma$ is \emph{$n$-sided} if $\partial \sigma$ consists of $n$ edges. Since $\Lambda$ is simplicial, no $2$-sided cell, or bigon, exists. A $3$-sided, $4$-sided or $5$-sided cell is called a \emph{triangle, square or pentagon}, respectively.
 
\subsection{Transversal and contact triplets}\label{subsec:transversal_contact}
Here we classify simple Hopf handlebody-link into two types. A triplet $\epsilon=\{v,e_1,e_2\}\in\mathcal{E}$ is \emph{transversal} if $e_1,e_2$ are not incident to the same cell of $\Lambda$. The pre-looping system $\mathcal{E}$ is \emph{transversal} if it contains one transversal triplet, and is \emph{contact} otherwise. 

\begin{lemma}
Given two equivalent simple fat loopings $V^{\mathcal{E}'}, V^\mathcal{E}$, then $\mathcal{E}$ is transversal if and only if $\mathcal{E}'$ is transversal.
\end{lemma}
\begin{proof}
Observe first the quotient map $\pi:(S_\ast,\Lambda)\rightarrow (S_\ast,\Lambda_\mathcal{E})$ 
sends a cell $\sigma$ of $\Lambda\subset S_\ast$ onto a cell $\sigma'$ of $\Lambda_\mathcal{E}\subset S_\ast$, and no two cells of $\Lambda$ sends to the same cell of $\Lambda_\mathcal{E}$. In particular, given a triplet $\epsilon=\{v,e_1,e_2\}$, then $e_1,e_2$ are the sides of the same cell $\sigma$ if and only if $\pi(e_1),\pi(e_2)$ are the sides of the same cell, namely $\pi(\sigma)$. This, together with Theorem \ref{teo:equivalence_Lambda}, implies the assertion. 
\end{proof}

\begin{definition}
A simple handlebody-link is \emph{transversal} if it is obtained from a fat looping on $\Lambda$ with respect to a transversal per-looping system $\mathcal{E}$, and it is \emph{contact} otherwise. 
\end{definition}

Recall that the \emph{maximal valence} of $\Lambda$ is the maximum among all vertex valences in $\Lambda$. If $\Lambda$ has the maximal valence $3$, namely, \emph{trivalent}, then $\mathcal{E}$ is always contact. 

\begin{lemma}\label{lm:quatri_graphs}\hfill
\begin{enumerate}[label=(\roman*)]
\item\label{itm:T} Every contact simple Hopf handlebody-link can be obtained from a fat looping on a trivalent trivial spatial graph. 
\item\label{itm:X} Every transversal simple Hopf handlebody-link can be obtained from a fat looping on a trivial spatial graph of maximal valence $4$.  
\end{enumerate}
\end{lemma}
\begin{proof}
Suppose the handlebody-link is equivalent to $V^\mathcal{E}$, and there is a vertex $v$ in $\Lambda$ with valence greater than $3$. Let $B$ be a cone neighborhood of $v$ with $N:=B\cap S_\ast$. 
We now construct a new graph $\tilde \Lambda\subset S$ from $\Lambda$ having less vertices of valence greater than $3$ and a pre-looping system $\tilde{\mathcal{E}}$ so that $(\tilde{\Lambda}_{\tilde{\mathcal{E}}},\tilde{\mathcal{E}})=(\Lambda_\mathcal{E},\mathcal{E})$. 

\textbf{Case $1$: $v$ is not in any triplet in $\mathcal{E}$.} 
We choose any trivalent tree $T\subset N$ with $T\cap \partial N=\Lambda\cap \partial N$, and define $\tilde{\Lambda}$ to be the graph given by $(\Lambda-N)\cup T$.

Suppose $v$ is in a triplet $\epsilon=\{v,e_1,e_2\}\in\mathcal{E}$. Then the union 
$e_1\cup v\cup e_2$ splits $N$ into two disks $D,D'$ with $D\cap D'=N\cap (e_1\cup v\cup e_2)$. 

\textbf{Case $2$: $\epsilon$ is not transversal.}  
Since $e_1,e_2$ are sides of the same cell $\sigma$ of $\Lambda$, one of $D,D'$, say $D$ is contained in $\sigma$. 
Take a trivalent tree $T\subset D'$ with 
$T\cap \partial D'=(\Lambda-(e_1\cup e_2))\cap\partial D'$, and define $\tilde{\Lambda}$ to be the graph $(\Lambda-D')\cup T\cup e_1\cup e_2$. 
 
\textbf{Case $3$: $\epsilon$ is transversal.} In this case, the interiors of $D,D'$ both meet $\Lambda$. Let $T,T'$, respectively, be trivalent trees in $D,D'$ with $T\cap \partial D=(\Lambda-e_1\cup e_2)\cap \partial D$ and $T'\cap \partial D'=(\Lambda-e_1\cup e_2)\cap \partial D'$. Then $\tilde{\Lambda}:=(\Lambda-N)\cup T\cup T'\cup e_1\cup e_2$. Note that the vertex adjacent to $e_1,e_2$ in $\tilde{\Lambda}$ is $4$-valent.
 
Label each vertex and edge in $\tilde\Lambda$ so that if it meets $S_\ast-N$ then the labeling is identical to the one in $\Lambda$. Thus $\mathcal{E}$ gives a pre-looping system $\tilde{\mathcal{E}}$ on $\tilde{\Lambda}$. In particular,  $(\Lambda_\mathcal{E},\mathcal{E})=(\tilde\Lambda_{\tilde{\mathcal{E}}},\tilde{\mathcal{E}})$, and hence $V^{\mathcal{E}}, V^{\tilde{\mathcal{E}}}$ are equivalent by Lemma \ref{lm:Lambda_E_to_V_E}. 
By induction, \ref{itm:T} follows from the first two cases, and \ref{itm:X} from all three cases.  
\end{proof}

\subsection{Simplicial $3$-connected graphs}\label{subsec:simplcial_three_connected}
In view of Lemma \ref{lm:quatri_graphs}, 
we classify here simplicial $3$-connected plane graphs with maximal valence $3$ or $4$ and $\chi\geq -4$. The next two lemmas are elementary but included here for the sake of completeness. 

\begin{lemma}\label{lm:tri_classification}
Suppose $\Lambda$ is a simplicial $3$-connected plane graphs with maximal valence $3$. Then $\chi<-1$, and  
\begin{enumerate}
\item $\Lambda$ is the tetrahedron in Fig.\ \ref{fig:tetrahedron} if $\chi =-2$;
\item $\Lambda$ is the prism in Fig.\ \ref{fig:prism} if $\chi =-3$;
\item $\Lambda$ is the cube in Fig.\ \ref{fig:cube} or the diamond in Fig.\ \ref{fig:pinched_pentagon} if $\chi =-4$. 
\end{enumerate} 
\end{lemma} 
\begin{figure}[b]
\begin{subfigure}{.24\linewidth}
\centering
\includegraphics[scale=.1]{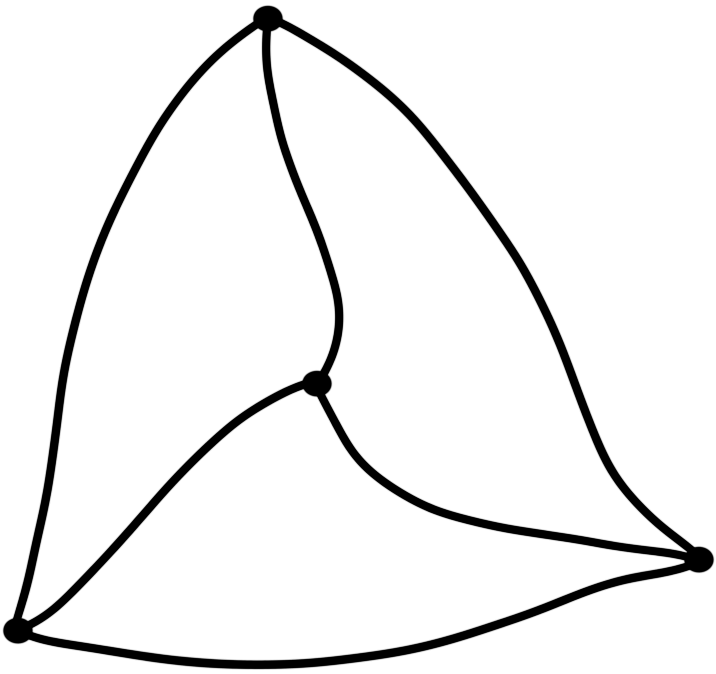}
\caption{Tetrahedron.}
\label{fig:tetrahedron}
\end{subfigure} 
\begin{subfigure}{.24\linewidth}
	\centering
\includegraphics[scale=.1]{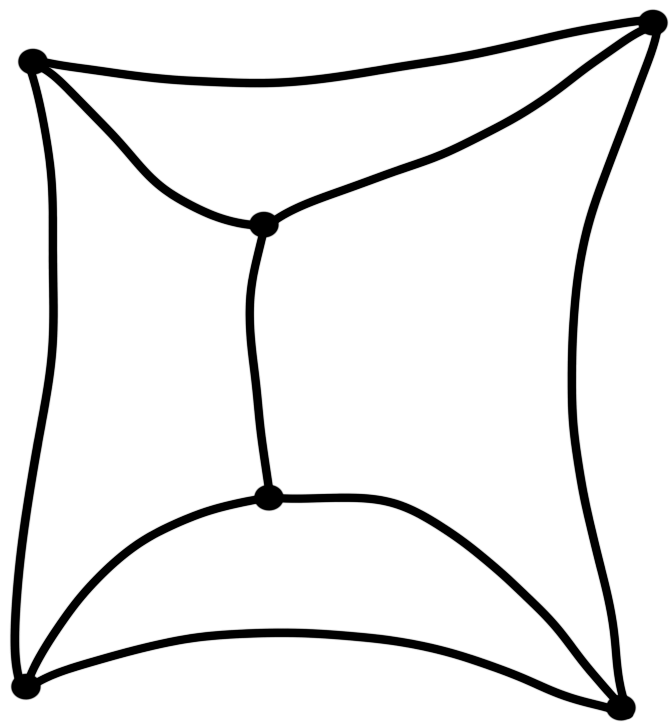}
\caption{Prism.}
\label{fig:prism}
\end{subfigure}
\begin{subfigure}{.24\linewidth}
	\centering
\includegraphics[scale=.1]{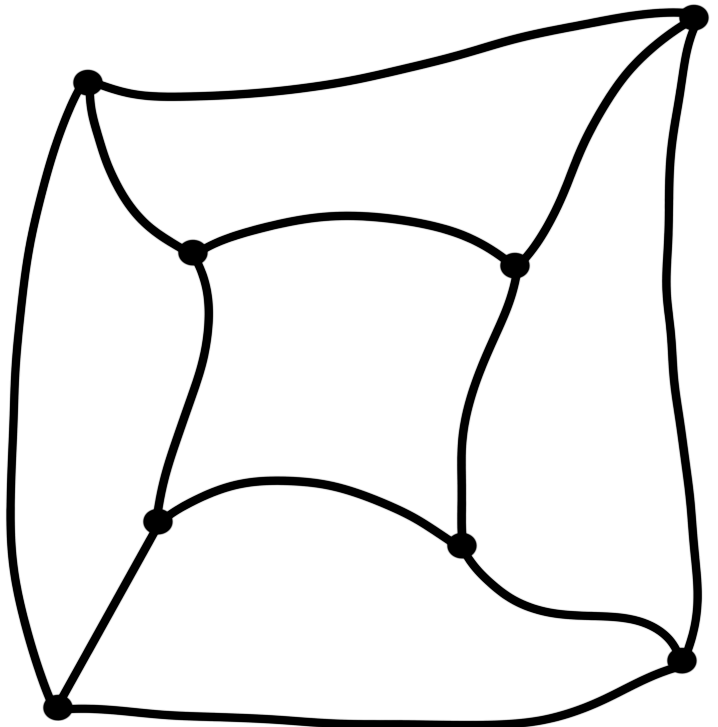}
\caption{Cube.}
\label{fig:cube}
\end{subfigure} 
\begin{subfigure}{.24\linewidth}
	\centering
\includegraphics[scale=.1]{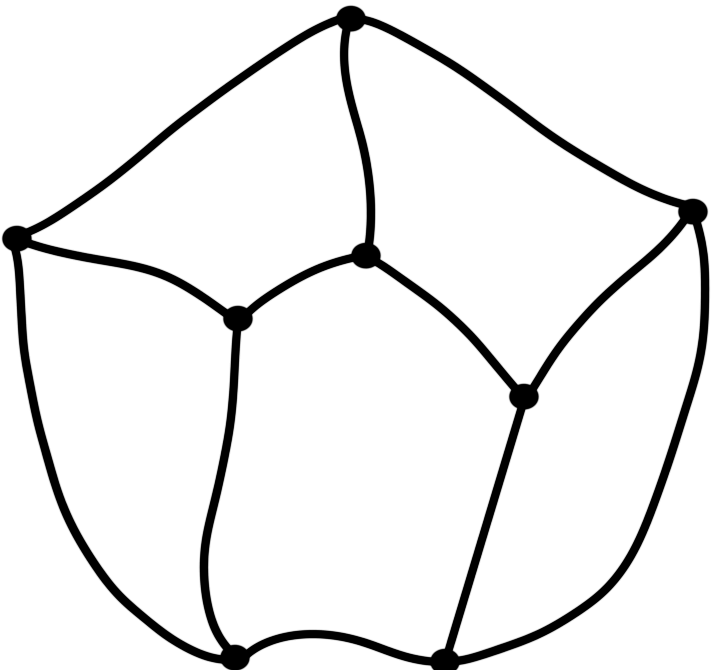}
\caption{Diamond.}
\label{fig:pinched_pentagon}
\end{subfigure} 
\caption{Simplicial $3$-connected trivalent plane graphs.} 
\end{figure}  
 
\begin{proof}
Let $\ver$ be the number of vertices, and $\edg$ the number of edges in $\Lambda$. 
Since $\Lambda$ is trivalent, we have $2\edg=3\ver$. Thus $\chi=v-e=\ver-\frac{3}{2}\ver$, so
\begin{equation}\label{eq:ver}
\ver=-2\chi 
\end{equation}
On the other hand, since $\Lambda$ is embedded in a $2$-sphere $S$, the number $\fac$ of cells is given by $\ver-\frac{3}{2}\ver+\fac=2$, so 
$\fac=-\chi+2$. More specifically, let $\fac_k$ be the number of $k$-sided cells $\Lambda\subset S_\ast$ admits. Denote by $\sigma$ a cell of $\Lambda$ that has the greatest number $\mcell$ of sides. Then, $\fac_2=0$, given $\Lambda$ has no multiple edges, and 
\begin{equation}\label{eq:face_k}
\fac=\fac_3+\cdots+\fac_m=\fac=-\chi+2. 
\end{equation}
Since $\Lambda$ is trivalent, counting the number of vertices in terms of $\fac_k$ yields
\begin{equation}\label{eq:w_face_k}
3\fac_3+\cdots+\mcell\fac_\mcell=3\ver= -6\chi.  
\end{equation}
Since $\Lambda$ is $3$-connected, 
the closure of $\Lambda-\partial \sigma$ is a connected graph with $1$-valent vertices precisely the vertices in $\partial \sigma$. 
The maximal number of $1$-valent vertices in $\Lambda-\partial \sigma$ is attained when it is a tree, so we have 
$\ver-\mcell+2\geq \mcell$, or equivalently,     
\begin{equation}\label{eq:max_cell}
\mcell\leq \frac{1}{2}\ver+2. 
\end{equation} 
 
Now, if $\chi\geq -1$, then $\ver\leq 2$ by \eqref{eq:ver}, contradicting $\Lambda$ having no bigons. 
If $\chi=-2$, then $\ver=4$ by \eqref{eq:ver}, and hence $\mcell\leq 3$ by \eqref{eq:max_cell}. Therefore $\fac_3=4$, and we obtain the tetrahedron in Fig.\ \ref{fig:tetrahedron}. 
If $\chi=-2$, then $\ver=6$ by \eqref{eq:ver}, and 
hence $m\leq 4$ by \eqref{eq:max_cell}. Thus, by \eqref{eq:face_k} and \eqref{eq:w_face_k}, we have $\fac_3+\fac_4=5$ and $3\fac_3+4\fac_4=18$. Solving the equations, we obtain $\fac_3=2$, $\fac_4=3$. This gives the prism in Fig.\ \ref{fig:prism}. Lastly, if $\chi=-4$, we have $\ver=8$ by \eqref{eq:ver} and $m\leq 5$ by \eqref{eq:max_cell}. Thus by \eqref{eq:face_k} and \eqref{eq:w_face_k}, we have $\fac_3+\fac_4+\fac_5=6$ and 
$3\fac_3+4\fac_4+5\fac_5=24$. If $\fac_5\neq 0$, namely, it admits a pentagon $\sigma$, then there are three vertices not in $\sigma$, and hence the equality in \eqref{eq:max_cell} holds; namely,  $\overline{\Lambda-\sigma}$ is a tree. This implies $\Lambda$ is the diamond in Fig.\ \ref{fig:pinched_pentagon}. If $\fac_5=0$, then we have 
$\fac_3+\fac_4=6$ and 
$3\fac_3+4\fac_4=24$, so $\fac_3=0$, $\fac_4=6$. 
This implies there is a cycle in $\overline{\Lambda-\sigma}$, and $\Lambda$ is the cube in Fig.\ \ref{fig:cube}.
\end{proof}

A \emph{contracting subgraph} $\Lambda_0$ is a subgraph of $\Lambda$ whose components are all contractible. 
Given a contracting subgraph $\Lambda_0$ of $\Lambda$, the associated equivalence relation $\sim_c$ is defined by $x\sim_c y$ if $x,y$ are in the same component of $\Lambda_0$. 
The quotient graph $\Lambda/\sim_c$ is called the \emph{contraction} of $\Lambda$ by $\Lambda_0$. 
Note that, if $\Lambda/\sim_0$ is simplicial and $3$-connected, then so is $\Lambda$. 
In the following, by \emph{up to homeomorphism}, we mean up to homeomorphism of $\Lambda$.

\begin{figure}[h]
	\begin{subfigure}{.32\linewidth}
		\centering
		\includegraphics[scale=.09]{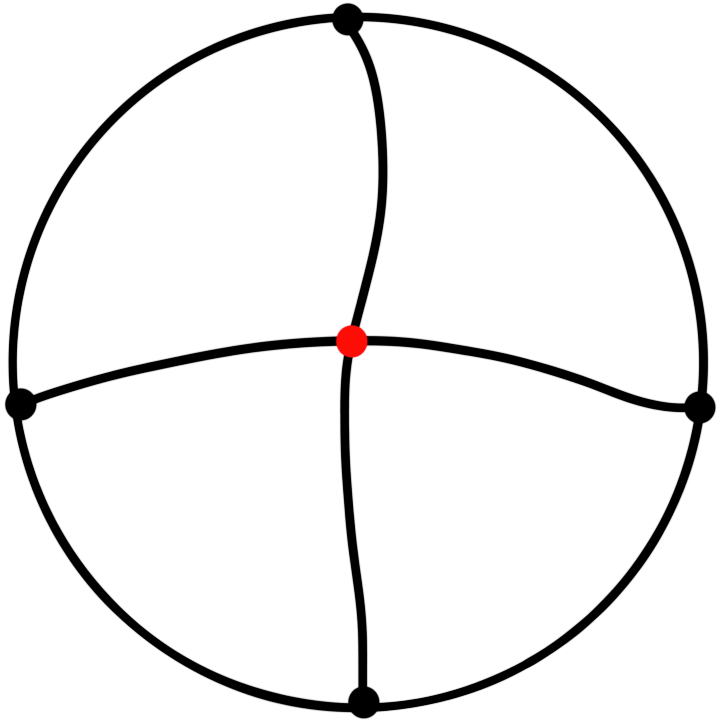}
		\caption{Wheel.}
		\label{fig:wheel}
	\end{subfigure} 
	\begin{subfigure}{.32\linewidth}
		\centering
		\includegraphics[scale=.09]{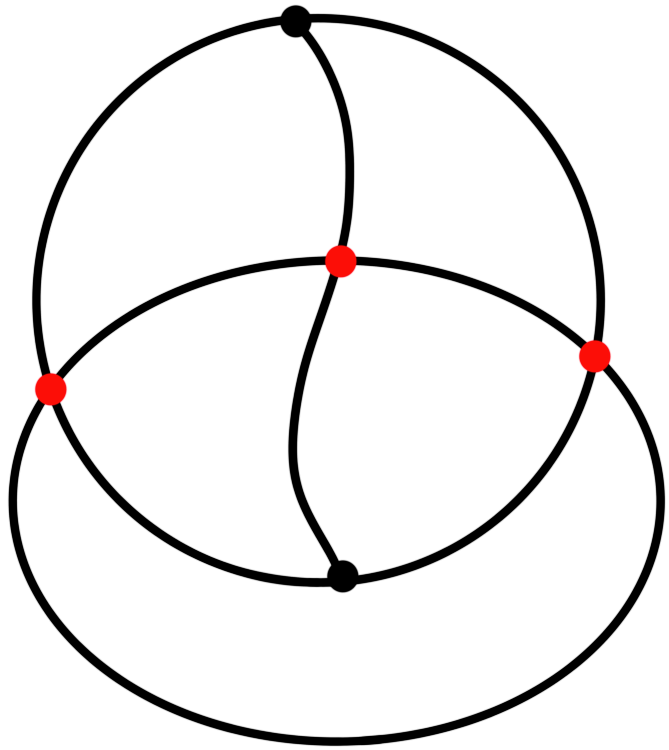}
		\caption{$K_4$.}
		\label{fig:Kfour}
	\end{subfigure}
	\begin{subfigure}{.32\linewidth}
		\centering
		\includegraphics[scale=.09]{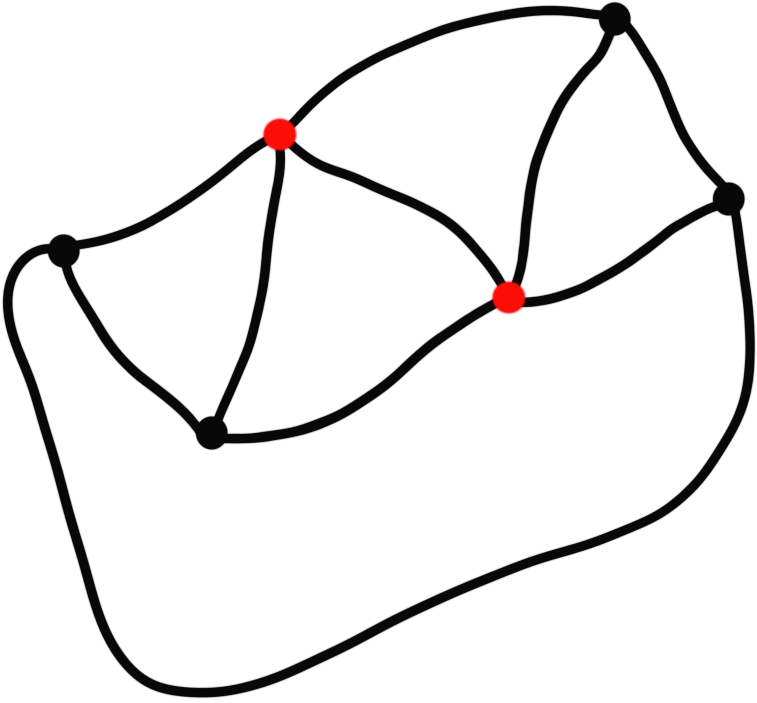}
		\caption{Kite.}
		\label{fig:kite}
	\end{subfigure} 
	\begin{subfigure}{.32\linewidth}
		\centering
		\includegraphics[scale=.09]{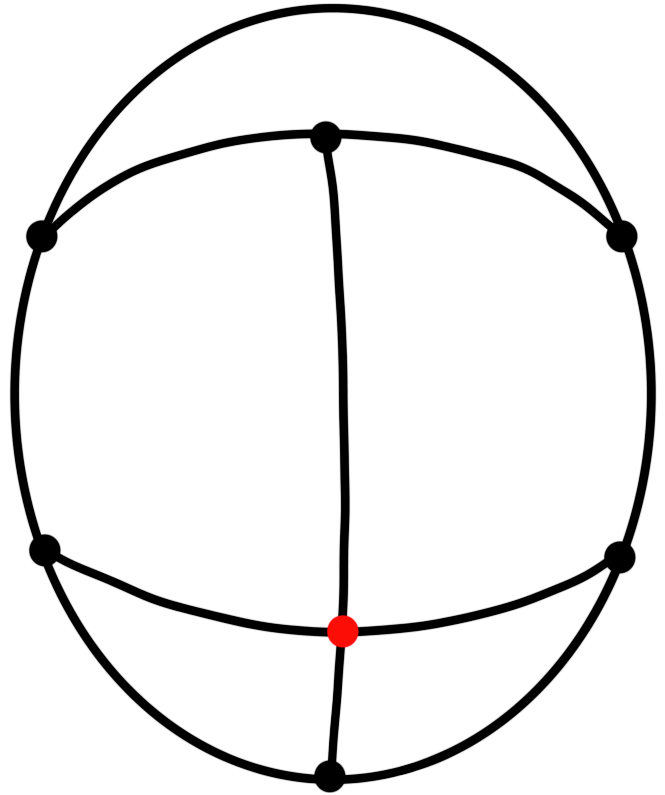}
		\caption{Bug.}
		\label{fig:bug}
	\end{subfigure} 
	\begin{subfigure}{.32\linewidth}
		\centering
		\includegraphics[scale=.09]{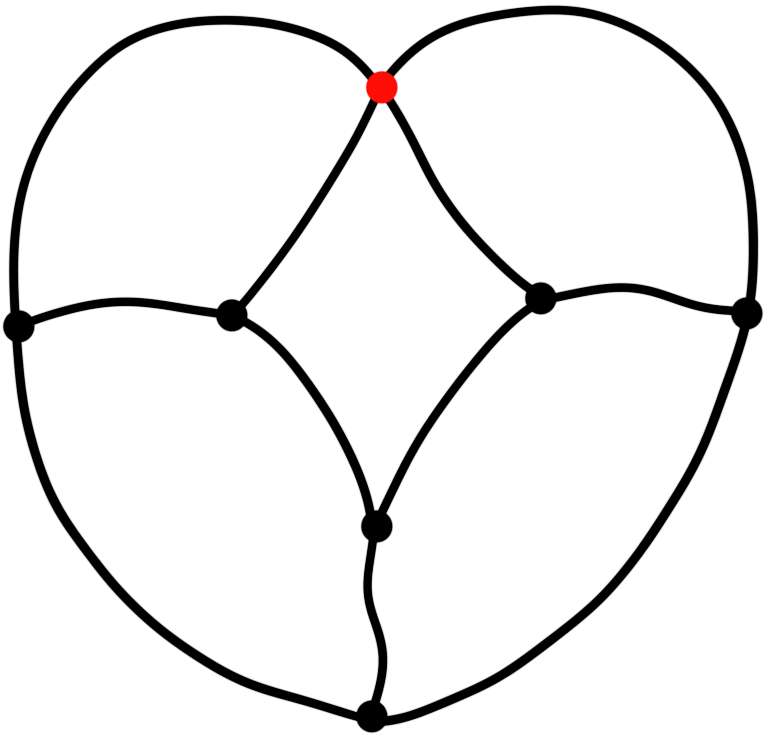}
		\caption{Heart.}
		\label{fig:heart}
	\end{subfigure} 
	\begin{subfigure}{.32\linewidth}
		\centering
		\includegraphics[scale=.1]{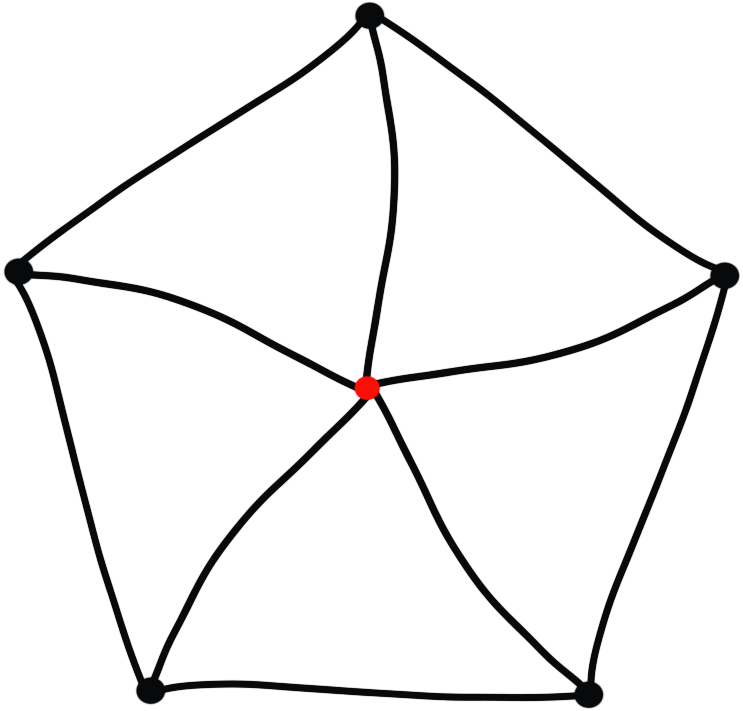}
		\caption{Web.}
		\label{fig:web}
	\end{subfigure}
	\caption{Simplicial $3$-connected plane graphs with $\chi=-3,-4$.}
	\label{fig:four_five_valence} 
\end{figure}
\begin{lemma}\label{lm:quatri_classification}
Suppose $\Lambda$ is a simplicial, $3$-connected plane graphs with maximal valence $4$. Let $\ver_4$ be the number of $4$-valent vertex in $\Lambda$. 
Then $\chi\leq -3$ and
\begin{enumerate}[label={(\roman*)}]
		\item\label{itm:che_n3} if $\chi=-3$, then $\ver_4=1$ and $\Lambda$ is the wheel in Fig.\ \ref{fig:wheel}.
		\item\label{itm:chi_m4} if $\chi=-4$, then 
		$\ver_4\leq 3$, and 
		\begin{enumerate}[label={(\alph*)}]
		\item\label{itm:bug_heart} $\Lambda$ is the bug in Fig.\ \ref{fig:bug} or the heart in Fig.\ \ref{fig:heart} if $\ver_4=1$.
		\item\label{itm:cup} $\Lambda$ is the kite in Fig.\ \ref{fig:kite} if $\ver_4=2$.  
		\item\label{itm:Kfour} $\Lambda$ is $K_4$ in Fig.\ \ref{fig:Kfour} if $\ver_4=3$.
		\end{enumerate}
\end{enumerate}
\end{lemma} 
\begin{figure}[b]
	\begin{subfigure}{.24\linewidth}
		\centering
		\includegraphics[scale=.1]{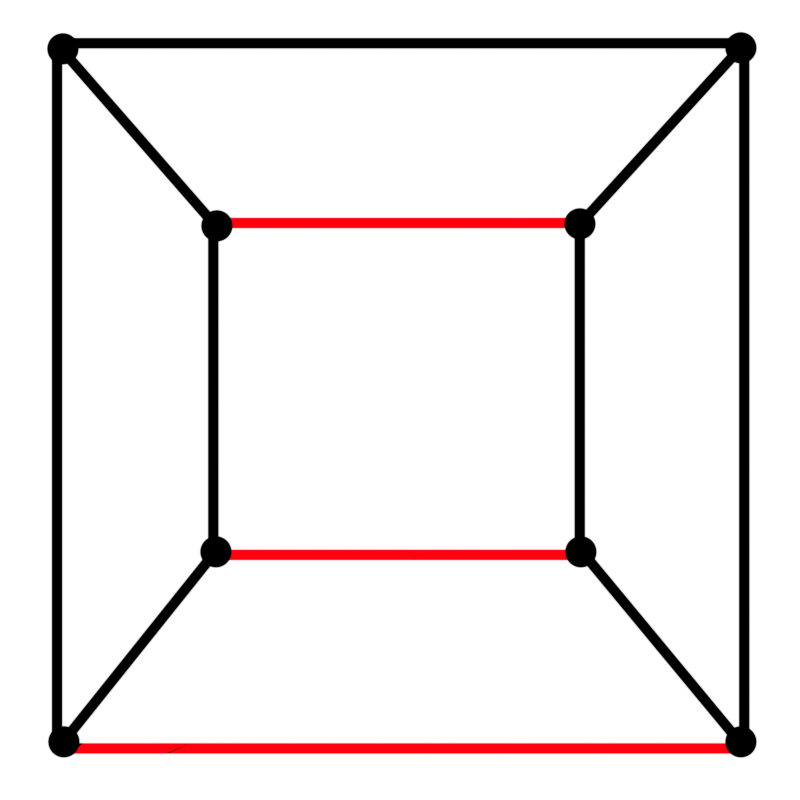}
		\caption{}
		\label{fig:cube_forbidden}
	\end{subfigure}
	\begin{subfigure}{.24\linewidth}
		\centering		\includegraphics[scale=.1]{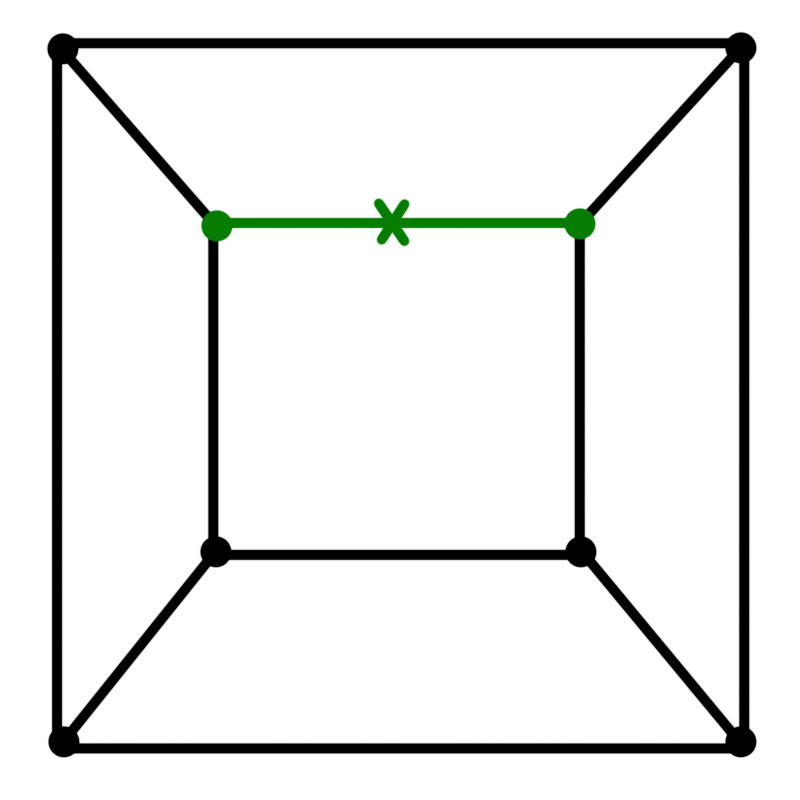}
		\caption{}
		\label{fig:cube_one}
	\end{subfigure} 
	\begin{subfigure}{.24\linewidth}
		\centering
		\includegraphics[scale=.1]{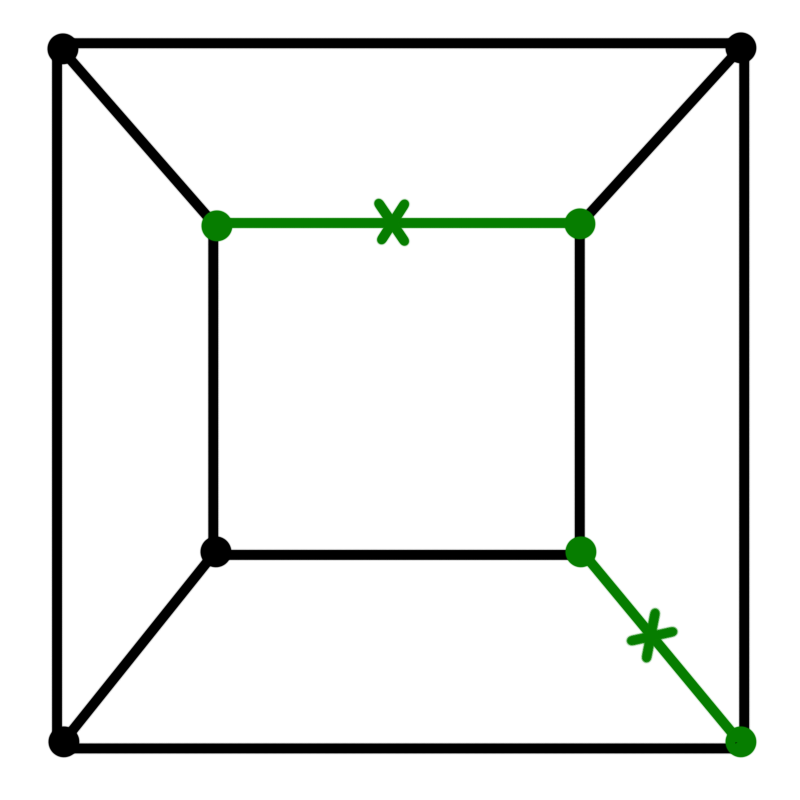}
		\caption{}
		\label{fig:cube_two}
	\end{subfigure}
	\begin{subfigure}{.24\linewidth}
		\centering
		\includegraphics[scale=.1]{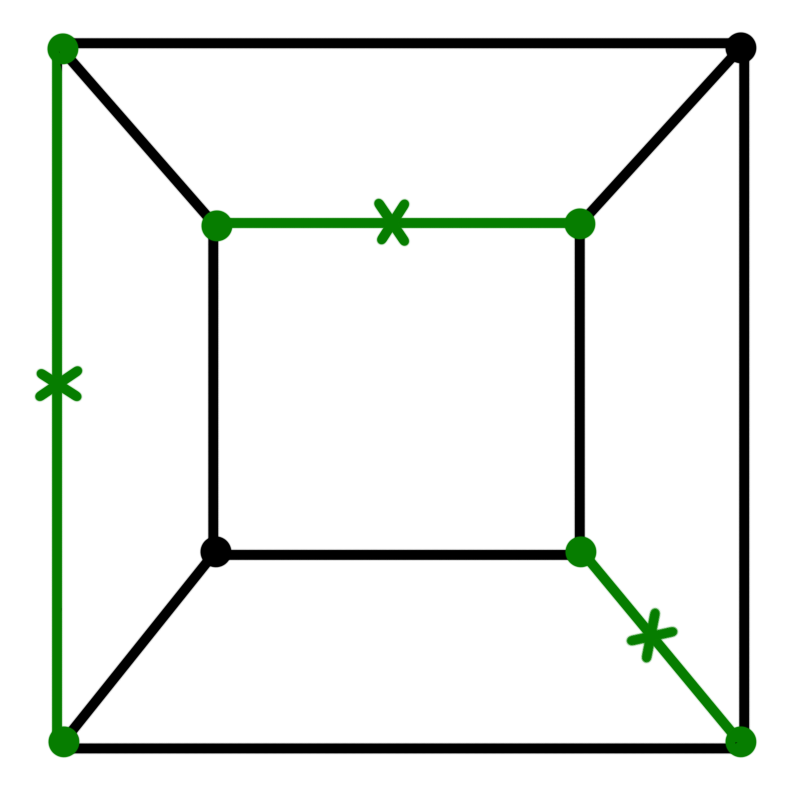}
		\caption{}
		\label{fig:cube_three}
	\end{subfigure} 
	\begin{subfigure}{.24\linewidth}
		\centering
		\includegraphics[scale=.1]{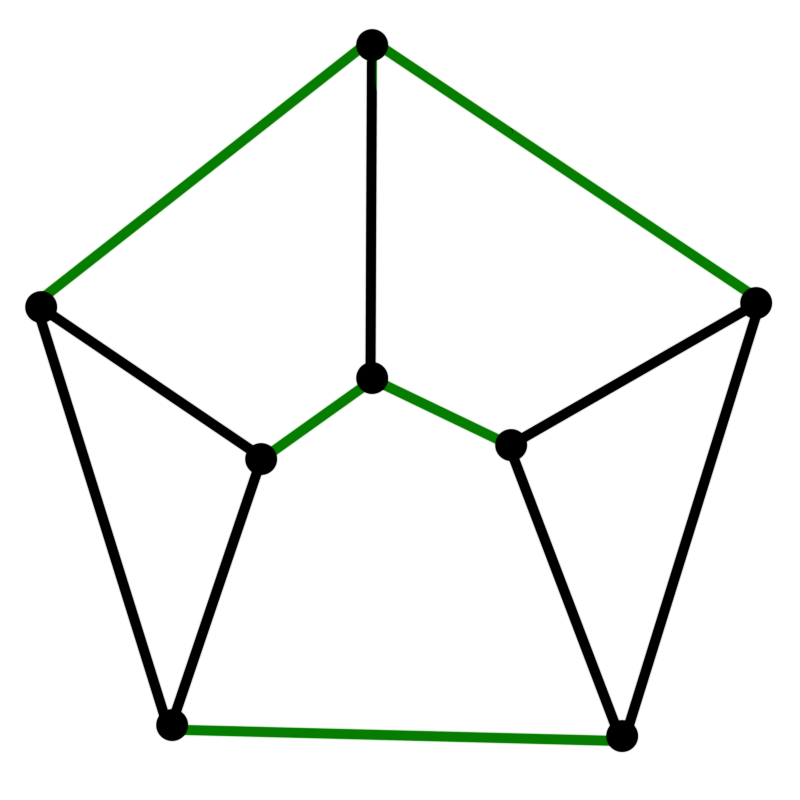}
		\caption{}
		\label{fig:diamond_allowed}
	\end{subfigure} 
	\begin{subfigure}{.24\linewidth}
		\centering
		\includegraphics[scale=.1]{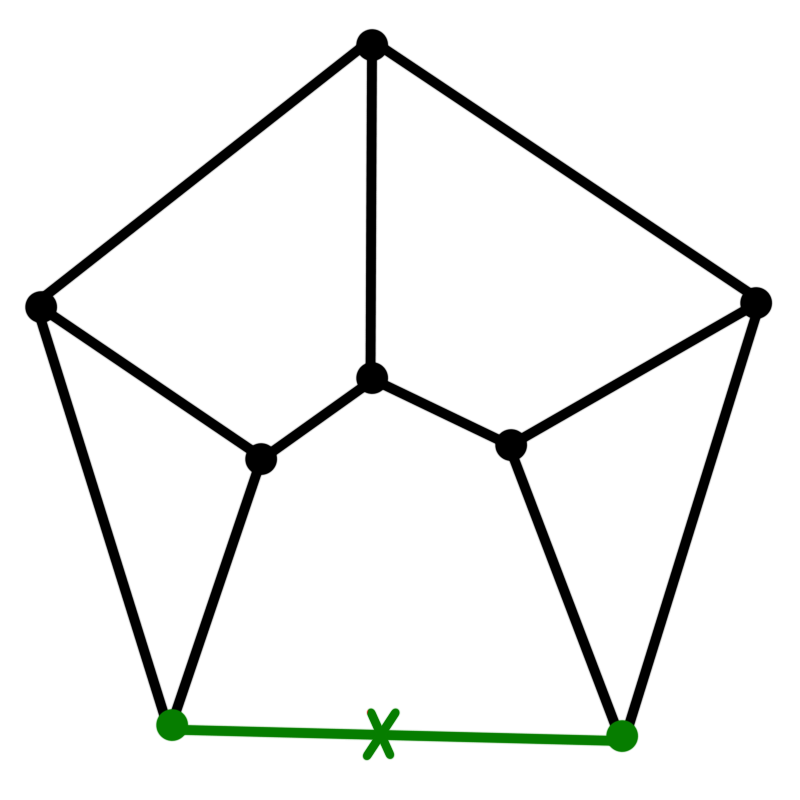}
		\caption{}
		\label{fig:diamond_one_bottom}
	\end{subfigure} 
	\begin{subfigure}{.24\linewidth}
		\centering
		\includegraphics[scale=.1]{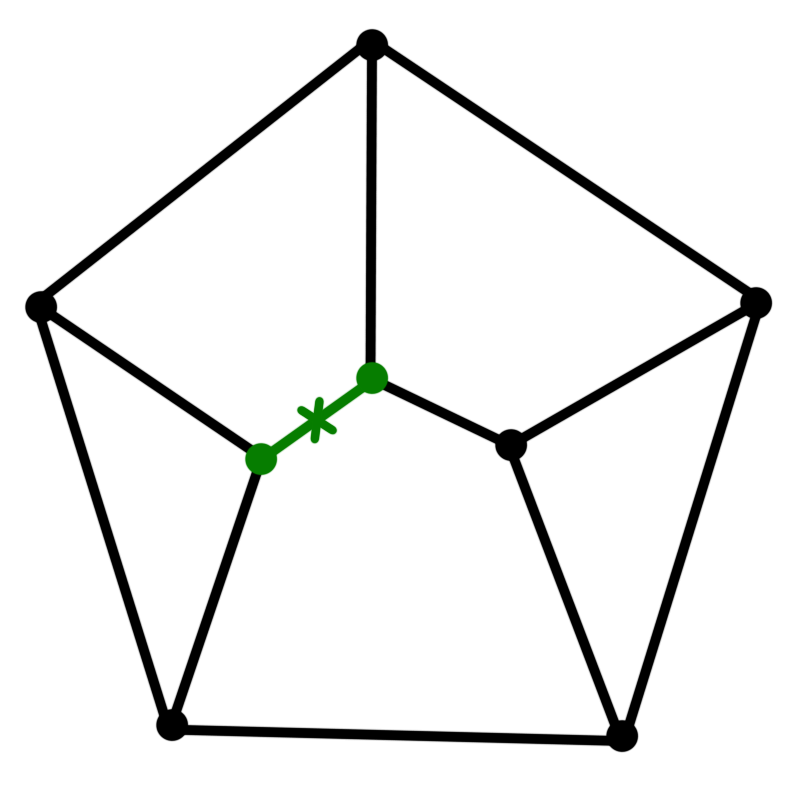}
		\caption{}
		\label{fig:diamond_one_top}
	\end{subfigure}

	\begin{subfigure}{.24\linewidth}
		\centering
		\includegraphics[scale=.1]{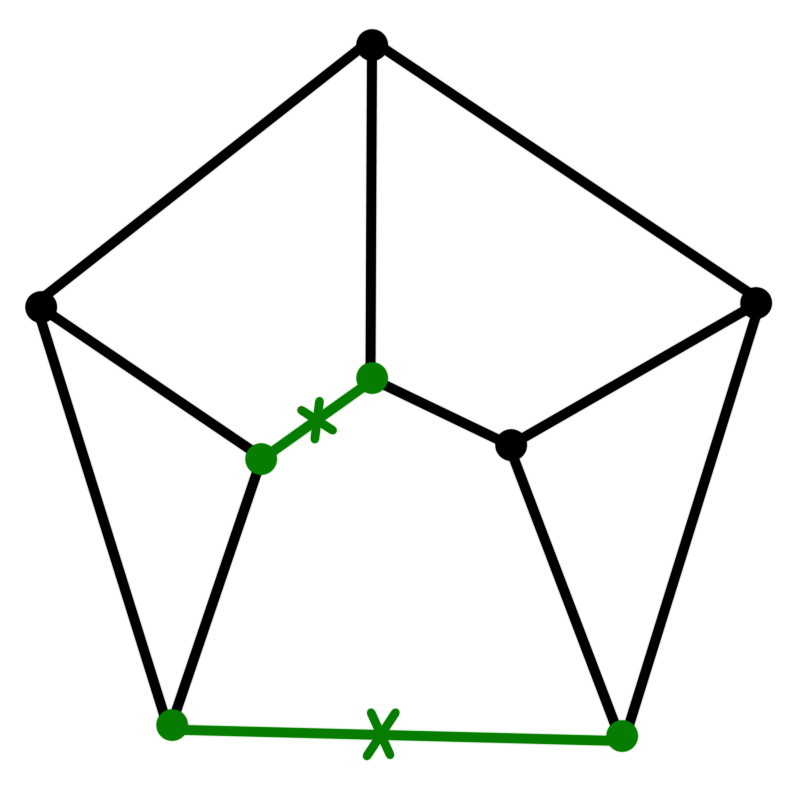}
		\caption{}
		\label{fig:diamond_two_bottom}
	\end{subfigure} 
	\begin{subfigure}{.24\linewidth}
		\centering
		\includegraphics[scale=.1]{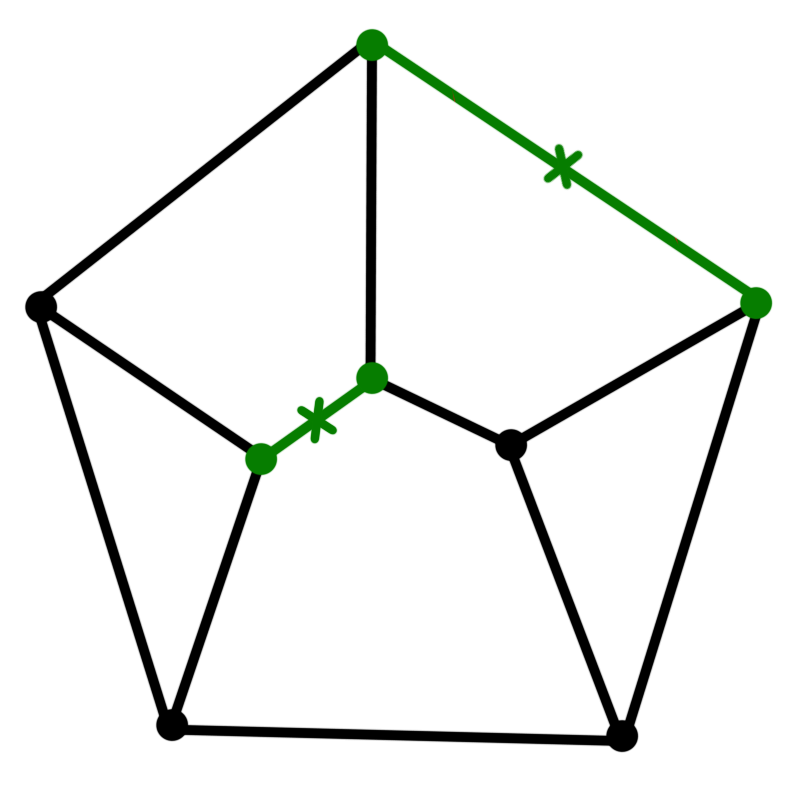}
		\caption{}
		\label{fig:diamond_two_top}
	\end{subfigure} 
	\begin{subfigure}{.24\linewidth}
		\centering
		\includegraphics[scale=.1]{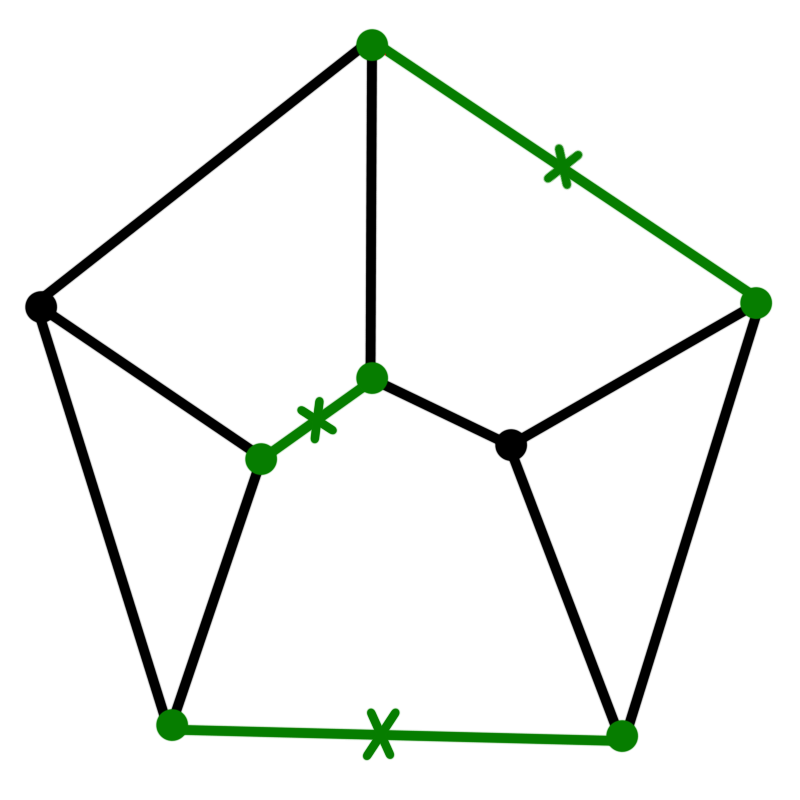}
		\caption{}
		\label{fig:diamond_three}
	\end{subfigure} 
	\caption{Contracting subgraphs.} 
\end{figure} 
\begin{proof}
By the observation preceding the theorem, $\Lambda$ is a contraction of some trivalent simplicial, $3$-connected graph $\tilde{\Lambda}$ by some contracting subgraph $\tilde\Lambda_0$. In addition, since $\Lambda$ has the maximal valence equal to $4$, edges in $\tilde{\Lambda}_0$ are mutually not incident. Since $\Lambda$ is simplicial, 
a $k$-sided cell can have at most $k-3$ sides in $\tilde{\Lambda}_0$; for instance, no side of a triangle can be in $\tilde{\Lambda}_0$. 
By Lemma \ref{lm:tri_classification}, this implies $\chi\leq -3$. 

Suppose $\chi=-3$. Then $\tilde{\Lambda}$ is the prism in Fig.\ \ref{fig:prism}. 
Thus $\tilde{\Lambda}_0$ is the closure of an edge incident to the two rectangles. The contraction yields the wheel in Fig.\ \ref{fig:wheel}.

Suppose $\chi=-4$. Then because $\Lambda$ is $3$-connected and simplicial, if $\tilde{\Lambda}$ is the cube, then $\tilde{\Lambda}_0$, up to homeomorphism, cannot contain any two of the red edges in Fig.\ \ref{fig:cube_forbidden}, and if $\tilde{\Lambda}$ is the diamond, only the five edges in green in Fig.\ \ref{fig:diamond_allowed} may appear in $\tilde{\Lambda}_0$.

If $\ver_4=1$, then $\tilde\Lambda_0$ contains only one edge. If $\tilde{\Lambda}$ is the cube in Fig.\ \ref{fig:cube}, then $\tilde{\Lambda}_0$ is the closure of any edge; see Fig.\ \ref{fig:cube_one}. 
The corresponding contraction is the heart in Fig.\ \ref{fig:heart}. If $\tilde{\Lambda}$ is the diamond in Fig.\ \ref{fig:pinched_pentagon}, then up to homeomorphism, there are two possible $\tilde{\Lambda}_0$: one shown in Fig.\ \ref{fig:diamond_one_top} and the other in Fig.\ \ref{fig:diamond_one_bottom}. Contracting $\tilde{\Lambda}$ with the former yields the bug in Fig.\ \ref{fig:bug}, and with the latter the heart in Fig.\ \ref{fig:heart}.

If $\ver_4=2$, then $\tilde\Lambda_0$ contains two edges. If $\tilde{\Lambda}$ is the cube, then,
given the forbidden edges in Fig.\ \ref{fig:cube_forbidden}, the subgraph $\tilde{\Lambda}_0$ is the one in Fig.\ \ref{fig:cube_two}, up to homeomorphism. Suppose $\tilde{\Lambda}$ is the diamond. By the allowed edges in Fig.\ \ref{fig:diamond_allowed}, if the edge incident to the two pentagons is in $\tilde{\Lambda}_0$, then $\tilde{\Lambda}_0$ is the one in Fig.\ \ref{fig:diamond_two_bottom}, and if the edge incident to the two pentagons is not in $\tilde{\Lambda}_0$, then $\tilde{\Lambda}_0$ is the one in Fig.\ \ref{fig:diamond_two_top}, up to homeomorphism. The contraction in any of the three cases above yields the kite in Fig.\ \ref{fig:kite}. 

If $\ver_4=3$, then $\tilde{\Lambda}_0$ contains three edges. Up to homeomorphism, it may be assumed two of them are the ones in Fig.\ \ref{fig:cube_two} if $\tilde{\Lambda}$ is the cube, and are the ones in Fig.\ \ref{fig:diamond_two_bottom} or Fig.\ \ref{fig:diamond_two_top} if $\tilde{\Lambda}$ is the diamond. By the forbidden edges in Fig.\ \ref{fig:cube_forbidden}, and the allowed edges in \ref{fig:diamond_allowed}, up to homeomorphism, $\tilde{\Lambda}_0$ is the one in Fig.\ \ref{fig:cube_three} if $\tilde{\Lambda}$ is the cube, and is the one in Fig.\ \ref{fig:diamond_three} if $\tilde{\Lambda}$ is the diamond. In either case, the contraction yields the $K_4$ in Fig.\ \ref{fig:Kfour}. 

If $\ver_4>3$, then, up to homeomorphism, it may be assumed that $\tilde{\Lambda}_0$ contains the three edges in Fig.\ \ref{fig:cube_three} or in Fig.\ \ref{fig:diamond_allowed}. In either case, any other edge is incident to one of the three edges, contradicting $\tilde{\Lambda}_0$ contains more than $3$ edges.
\end{proof}

Since the number of edges in $\Lambda$ is $\frac{1}{2}(3\ver_3+4\ver_4)$, we have the following relation:  
\begin{equation}\label{eq:vertex_euler}
\ver_3+2\ver_4=-2\chi. 
\end{equation}

\subsection{Connectedness}
This section concerns the implication of $V^\mathcal{E}$ being connected.
Hereinafter, we assume $\Lambda$ has the maximal valence $4$ if 
$\mathcal{E}$ is transversal and is trivalent if $\mathcal{E}$ is contact, and every $4$-valent vertex, if exists, is in a transversal triplet.

The \emph{split} $\Sigma_\mathcal{E}$ of $\Lambda$ at $\mathcal{E}$ is a graph resulting from the operation: for every triplet $\epsilon=\{v,e_1,e_2\}\in \mathcal{E}$, 
replace $v$ in $\Lambda$ with two vertices $v_+,v_-$ such that $e_1,e_2$ are incident to $v_-$ and the other edges incident to $v$ in $\Lambda$ are incident to $v_+$ in $\Sigma_\mathcal{E}$; all other edges and vertices in $\Lambda$ remain intact; see Fig.\ \ref{fig:4hopf_2} and \ref{fig:split}. 

\begin{figure}[b]
	\begin{subfigure}{.24\linewidth}
		\centering
		\begin{overpic}[scale=.09,percent]{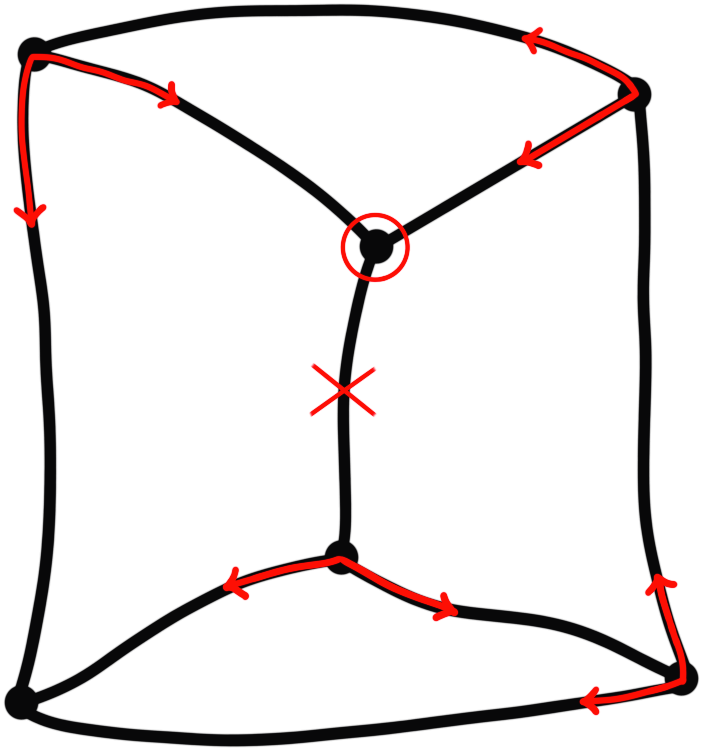}
		\end{overpic}
		\caption{}
		\label{fig:4hopf_2}
	\end{subfigure}  
	\begin{subfigure}{.24\linewidth}
		\centering
		\begin{overpic}[scale=.09,percent]{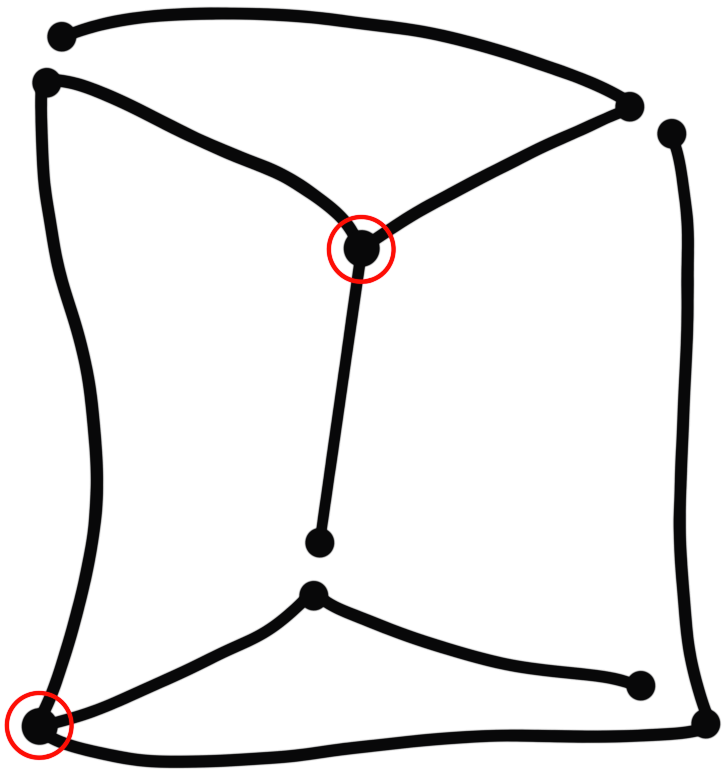}
			\put(30,28){\footnotesize $+$}
			\put(30,21){\footnotesize $-$}
			\put(77,15){\footnotesize $+$}
			\put(90,0){\footnotesize $-$}
			\put(90,80){\footnotesize $+$}
			\put(80,88){\footnotesize $-$}
			\put(-1,95){\footnotesize $+$}
			\put(-1,89){\footnotesize $-$}
		\end{overpic}
		\caption{The split $\Sigma_\mathcal{E}$.}
		\label{fig:split}
	\end{subfigure}  
	\begin{subfigure}{.24\linewidth}
		\centering
		\begin{overpic}[scale=.11,percent]{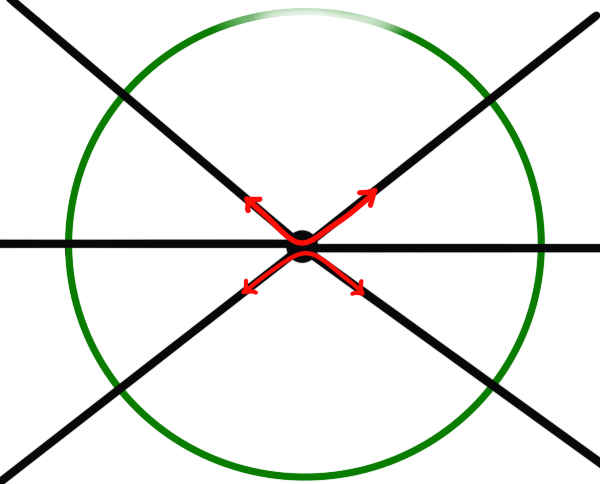}
			\put(48,29){\footnotesize $v$}
			\put(45,74){\footnotesize $N_v$}
		\end{overpic}
		\caption{Two triplets at $v$.}
		\label{fig:plane_cone_nbhd}
	\end{subfigure}  
	\begin{subfigure}{.24\linewidth}
		\centering
		\begin{overpic}[scale=.1,percent]{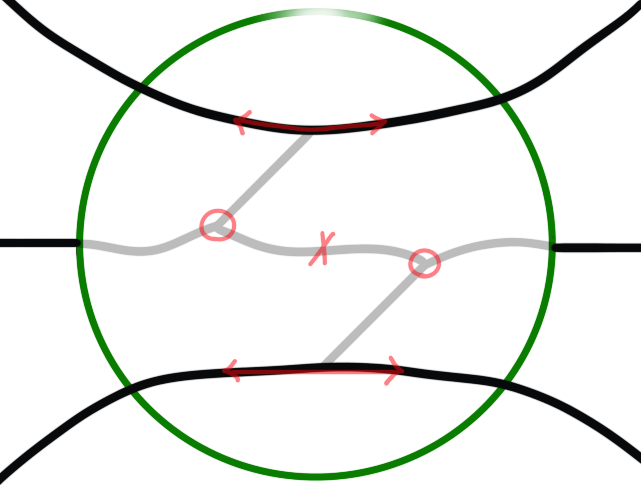}
			\put(67,20){\tiny $\gamma_v$}
			\put(64,50){\tiny $\gamma_v$}
			\put(43,69){\footnotesize $N_v$}
			\put(20,45){\footnotesize $D$}
		\end{overpic}
		\caption{$\mathfrak{ind}_c(v)=1$.}
		\label{fig:smoothing}
	\end{subfigure} 
	\caption{}  
	\label{}
\end{figure}
\begin{lemma}
The number of components $V^\mathcal{E}$ is the same as the number of components of $\Sigma_\mathcal{E}$.
\end{lemma}
\begin{proof}
This follows from the fact that the looping $\phi_n\subset B^3$ of the standard vertical cone $\psi_n\subset B^3$ has two components $\psi_z$, corresponding to the edges and vertex of a triplet, and $o_{xy}$, corresponding to other edges incident to the vertex in the triplet. 
\end{proof}

Let $\unloopedv$ and $\unloopede$ be the numbers of unlooped vertices and unlooped edges, respectively. By the convention, any unlooped vertex is 
$3$-valent. An edge in $\Lambda$ is called \emph{doubled} if it occurs in two triplets of $\mathcal{E}$; denote by $\doubled$ the number of doubled edges. Since every looped vertex corresponds to two edges, we have the following:
\[ 
2(\ver_3+\ver_4-\unloopedv)-\doubled+\unloopede=\edg=\frac{1}{2}(3\ver_3+4\ver_4),
\]
which can be simplified into 
\begin{equation}\label{eq:edge_equation}
\unloopede-\doubled=2\unloopedv-\frac{1}{2}\ver_3. 
\end{equation}   
Since $\Lambda_\mathcal{E}$ is simplicial and $3$-connected, by Lemmas \ref{lm:tri_classification} and \ref{lm:quatri_classification}, we have the upper bounds of $\uncovered$:
\begin{corollary}\label{cor:uncovered_upper_bound}\hfill
\begin{enumerate}[label={(\roman*)}]
\item If $\chi= -2$, then $\uncovered=\ver_4=0$. 
\item If $\chi=-3$, then $\uncovered+\ver_4\leq 1$.
\item If $\chi=-4$, then $\uncovered+\ver_4\leq 3$.
\end{enumerate}
\end{corollary} 

Note that the trivalent vertices in $\Sigma_\mathcal{E}$ are precisely those unlooped vertices and one-valent vertices in $\Sigma_\mathcal{E}$ correspond to those looped trivalent vertices. Let $\mathsf{n}_1$ (resp.\ $\mathsf{n}_3$) be the number of one-valent (resp.\ trivalent) vertices in $\Sigma_\mathcal{E}$ and $\chi'$ the Euler characteristic of $\Sigma_\mathcal{E}$. 
Then 
\begin{equation}\label{eq:formula_in_split}
\chi'=\mathsf{n}_3-\mathsf{n}_1-(\frac{3\mathsf{n}_3+\mathsf{n}_1}{2})=-\frac{1}{2}(\mathsf{n}_3-\mathsf{n}_1)=-\frac{1}{2}(\unloopedv-(\ver_3-\unloopedv))=\unloopedv-\frac{1}{2}\ver_3.
\end{equation}

\begin{lemma}
If $V^\mathcal{E}$ is connected with $\chi\geq -4$, then $\Sigma_\mathcal{E}$ is a tree.
\end{lemma}
\begin{proof}
Suppose $\Sigma_\mathcal{E}$ contains a cycle, namely $\chi'\leq 0$. Then by \eqref{eq:formula_in_split}, we have $2\notlooped \geq \ver_3$, and by \eqref{eq:edge_equation} and \eqref{eq:vertex_euler}, this implies   
\[\unloopede-\doubled=2\unloopedv-\frac{1}{2}\ver_3\leq \frac{1}{2}\ver_3=-\chi-\ver_4,\] 
contradicting Corollary \ref{cor:uncovered_upper_bound}.
\end{proof}

The following is the main formula our enumeration is based on.
\begin{lemma}\label{lm:main_formula} 
Suppose $V^\mathcal{E}$ is connected with $\chi\geq -4$. Then
\begin{enumerate}[label={(\roman*)}]
\item\label{itm:notlooped} $\unloopedv=-\chi-\ver_4-1$.
\item\label{itm:uncovered_doubled} $\unloopede-\doubled=-\chi-\ver_4-2$.
\end{enumerate}
\end{lemma}
\begin{proof}
Since $\Sigma_\mathcal{E}$ is a tree, by \eqref{eq:formula_in_split}, we have 
\begin{equation}\label{eq:notlooped_ver3}
2\unloopedv+2=\ver_3. 
\end{equation}
Replacing $\ver_3$ with $\ver_4,\chi$ by \eqref{eq:vertex_euler} yields \ref{itm:notlooped}. 
Plugging 
\eqref{eq:notlooped_ver3} into \eqref{eq:edge_equation} yields
\begin{equation} 
\unloopede-\doubled=\frac{1}{2}\ver_3-2.
\end{equation}
Then applying \eqref{eq:vertex_euler} to replace $\ver_3$ with $\ver_4,\chi$ gives us \eqref{itm:uncovered_doubled}.
\end{proof}


\subsection{Invariants}\label{subsec:invariants}
Here we define some invariants for simple Hopf handlebody-links. They are used to show entries in Table \ref{tab:hopf_handlebody_knots} are mutually inequivalent.

Suppose $\mathcal{E}, \mathcal{E}'$, respectively, are pre-looping systems of two trivial spatial graphs $\Lambda,\Lambda'$ lying on the $2$-spheres $S_\ast,S_\ast'$, respectively, with $\Lambda_\mathcal{E}, \Lambda_{\mathcal{E}'}'$ simplicial and $3$-connected. If their fat loopings $V^\mathcal{E}, V^{\mathcal{E}'}$ are equivalent, then by Theorem \ref{teo:equivalence_Lambda} there is a homeomorphism between $(\Lambda_\mathcal{E},\mathcal{E})$, 
$(\Lambda_{\mathcal{E}'},\mathcal{E}')$, and hence, 
by the Whitney theorem \cite{Whi:33}, 
there is a homeomorphism between  
$(S_\ast,\Lambda_\mathcal{E},\mathcal{E})$,
$(S_\ast',\Lambda_{\mathcal{E}'},\mathcal{E}')$.
Any invariant of $(S,\Lambda_\mathcal{E},\mathcal{E})$ is thus an invariant of $V^\mathcal{E}$, for instance, the homeomorphism type of $\Lambda_\mathcal{E}$.

\subsubsection*{Contraction index}
Recall from Section \ref{subsec:Lambda_E_to_V_E} the cone neighborhood $B_v$ of a vertex $v\in \Lambda_\mathcal{E}$ and the smoothing $\gamma_v\subset N_v:=B_v\cap S_\ast$; see Figs.\ \ref{fig:plane_cone_nbhd}, \ref{fig:smoothing}. Define the \emph{contraction index} as 
\[\mathfrak{ind}_c(v):=\mathrm{max}\{\vert \partial_f D\vert+\vert \Lambda_\mathcal{E}\cap\partial D\vert-3 \mid \text{ $D$ is a component of $N-\gamma_v$}\},\]
and the \emph{maximal contraction index} $\mathrm{ind}_c$ of $(\Lambda_\mathcal{E},\mathcal{E})$ is defined as 
\[\mathfrak{ind}_c:=\mathrm{max}\{\mathrm{ind}_c(v)\mid \text{ $v$ a vertex of $\Lambda_\mathcal{E}$}\}.\]

\begin{figure}[t]
	\begin{subfigure}{.22\linewidth}
		\centering
		\includegraphics[scale=.08]{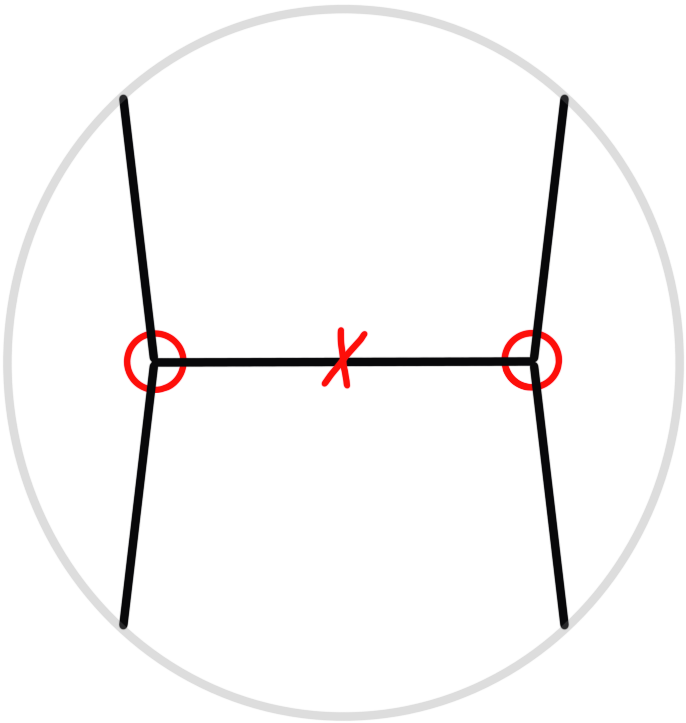}
		\caption{}
		\label{fig:IH_move_H}
	\end{subfigure} 
	\begin{subfigure}{.22\linewidth}
		\centering
		\includegraphics[scale=.08]{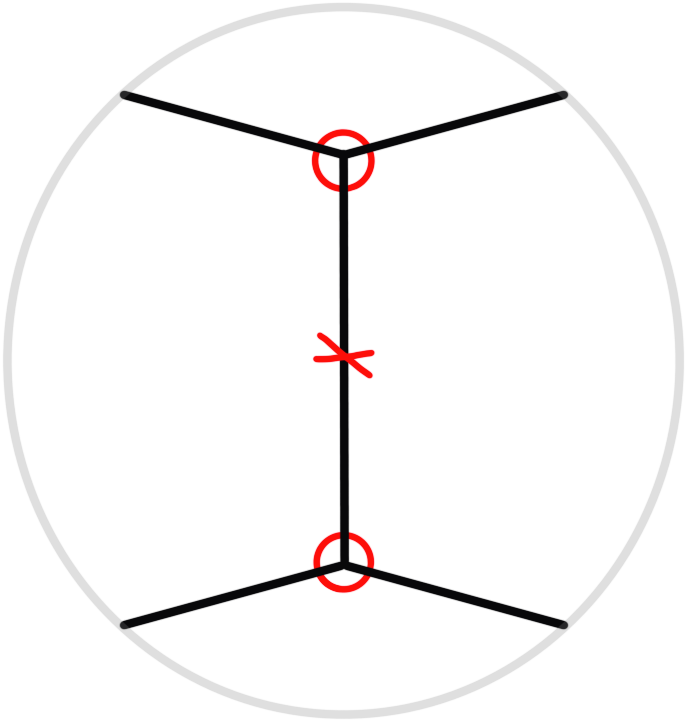}
		\caption{}
		\label{fig:IH_move_I}
	\end{subfigure}
	\begin{subfigure}{.26\linewidth}
		\centering
		\begin{overpic}[scale=.08,percent]{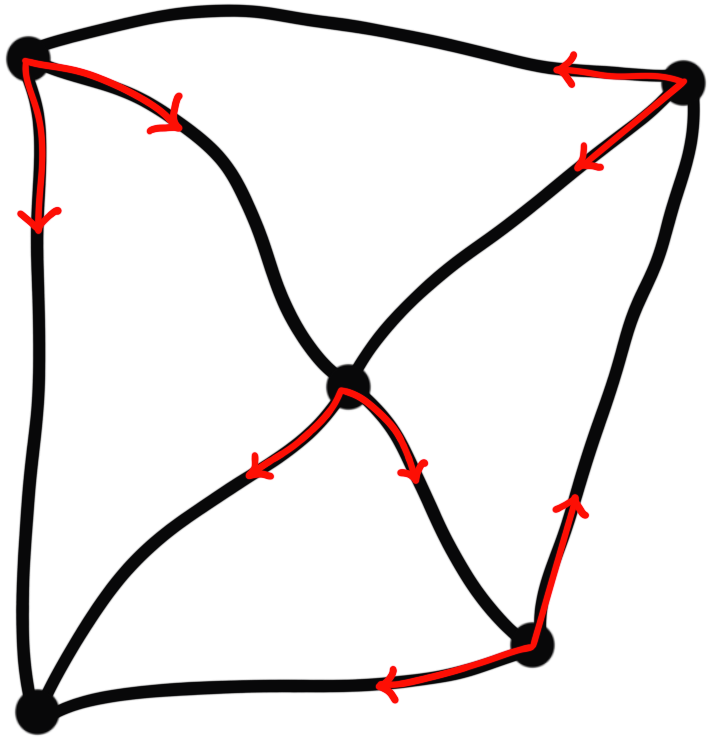}
			\put(40,20){\footnotesize $\sigma_1$}
			\put(20,50){\footnotesize $\sigma_2$}
			\put(44,70){\footnotesize $\sigma_3$}
			\put(90,10){\footnotesize $\tau$}
		\end{overpic}
		\caption{$\mathfrak{car}=(3,1)$.}
		\label{fig:carrier}
	\end{subfigure} 
	\begin{subfigure}{.26\linewidth}
		\centering
		\begin{overpic}[scale=.08,percent]{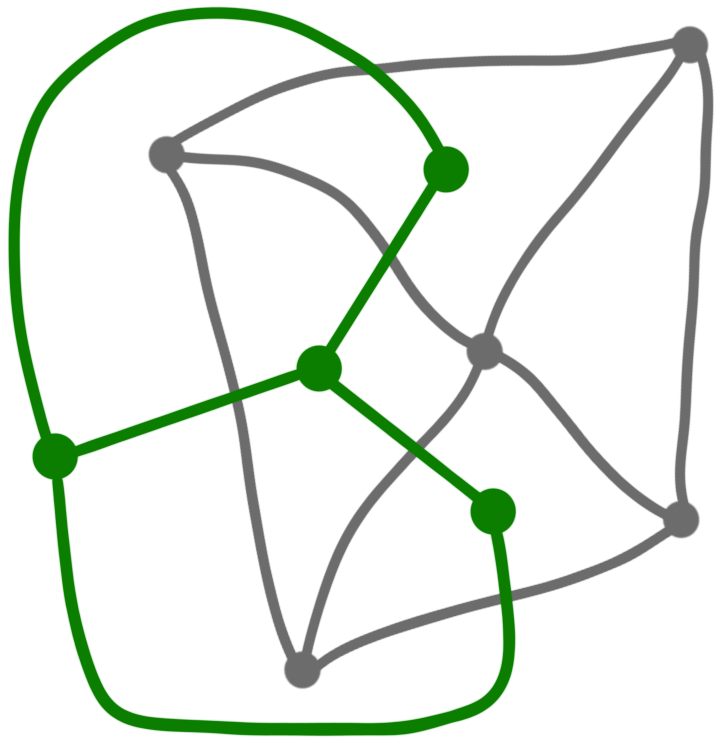}
			\put(70,0){\footnotesize $\mathfrak{map}$}
		\end{overpic}
		\caption{$\mathfrak{map}\simeq$ theta graph.}
		\label{fig:map}
	\end{subfigure}
\end{figure} 
\subsubsection*{IH-allowable}
The contraction index is related to the unlooped edges in $\Lambda$. 
An edge is \emph{IH-allowable} if both vertices incident to it are unlooped.
An IH-allowable edge is necessarily unlooped though the converse is not true in general.

Two IH-allowable edges $e,e'$ are in the same \emph{IH-cluster} if there exists a sequence of IH-allowable edges $\{e_1=e,e_2,\dots,e_n=e'\}$
such that $e_{i},e_{i+1}$ are incident.
Then $\mathfrak{ind}_c$ is the maximal number of edges an IH-cluster can have; see Fig.\ \ref{fig:smoothing}. 

If $(\Lambda, \mathcal{E})$ contains an IH-allowable edge $e$, then the local transformation around $e$ depicted in Fig.\ \ref{fig:IH_move_H}, \ref{fig:IH_move_I} changes $\Lambda$ into another plane graph $\Lambda'$ and changes the pre-looping system $\mathcal{E}$ into a pre-looping system $\mathcal{E}'$ on $\Lambda'$. The fat looping $V^\mathcal{E}, V^{\mathcal{E}'}$ are, however, equivalent because they have spines differ by an IH-move; see \cite{Ish:08}. 

\subsubsection*{Carrier}
Given a contact triplet $\epsilon\in\mathcal{E}$, 
the \emph{carrier} of $\epsilon$ is a cell of $\Lambda_\mathcal{E}\subset S_\ast$ that contains $\varepsilon$; such a cell is unique since $\Lambda_\mathcal{E}$ has no $2$-valent vertex. Denote by $c_i$ the number of triplets whose carriers are $i$-sided, and define \emph{the carrier sequence} to be 
\[\mathfrak{car}:=(c_3,c_4,\dots, c_\mcell);\  \text{see Fig.\ \ref{fig:carrier} for an example}.\]

Consider the dual graph $\Lambda^\ast_\mathcal{E}$ of $\Lambda_\mathcal{E}\subset S_\ast$, and let $v_C$ be the set of vertices in $\Lambda_\mathcal{E}^\ast$ that correspond to carriers. Then the full subgraph given by $v_C$ is called the \emph{carrier map}, denoted by $\mathfrak{map}$; see Fig.\ \ref{fig:map}. Both $\mathfrak{car}$ and $\mathfrak{map}$ are invariants of $V^\mathcal{E}$.

\subsection{Proof of Theorem \ref{teo:enumeration}} 
\begin{figure}[b]
	\begin{subfigure}{.24\linewidth}
		\centering
		\includegraphics[scale=.1]{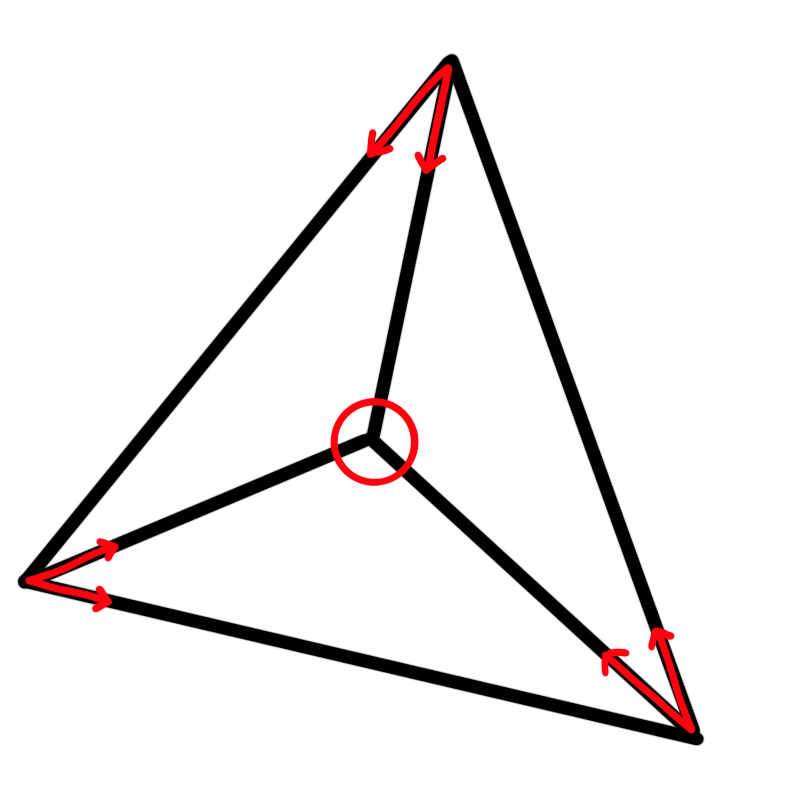}
		\caption{$\chi=-2$.}
		\label{fig:pre2h1}
	\end{subfigure} 
	\begin{subfigure}{.24\linewidth}
		\centering
		\includegraphics[scale=.1]{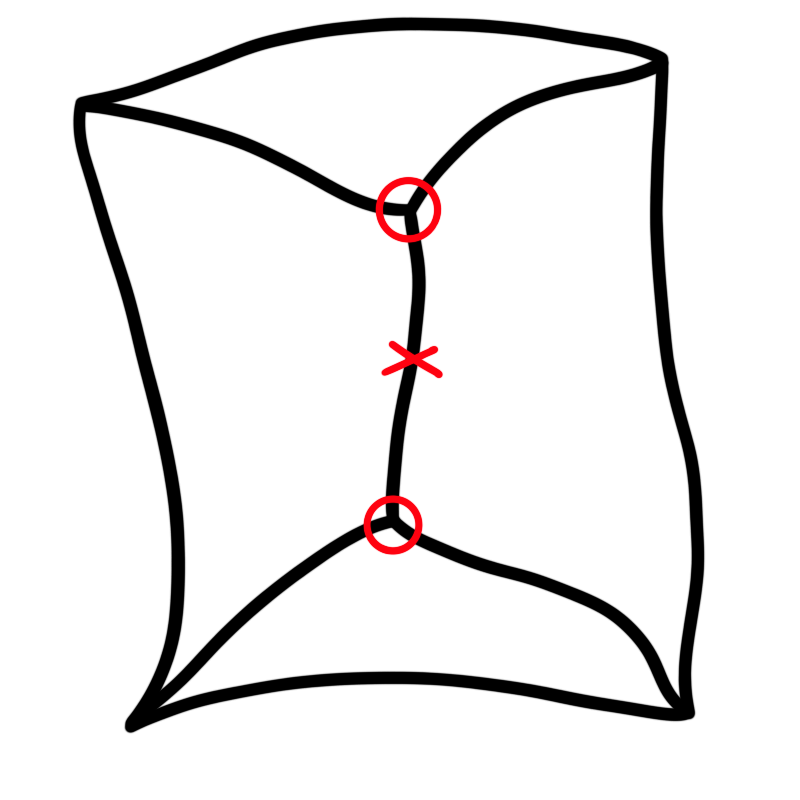}
		\caption{Case $1$.}
		\label{fig:pre3h1}
	\end{subfigure} 
	\begin{subfigure}{.24\linewidth}
		\centering
		\includegraphics[scale=.1]{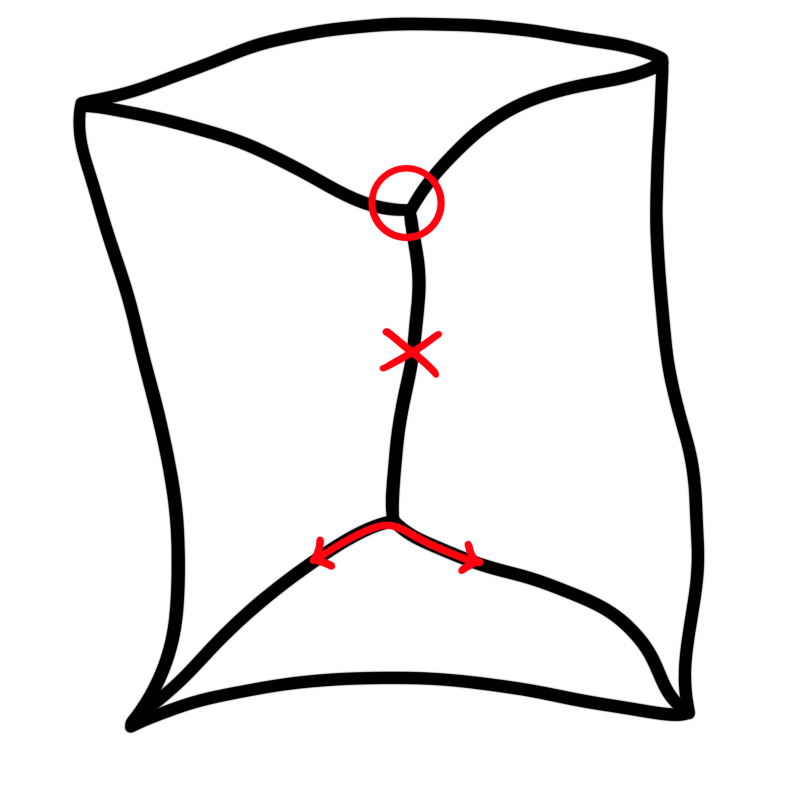}
		\caption{Case $1$.}
		\label{fig:pre3h2}
	\end{subfigure} 
	\begin{subfigure}{.24\linewidth}
		\centering
		\includegraphics[scale=.1]{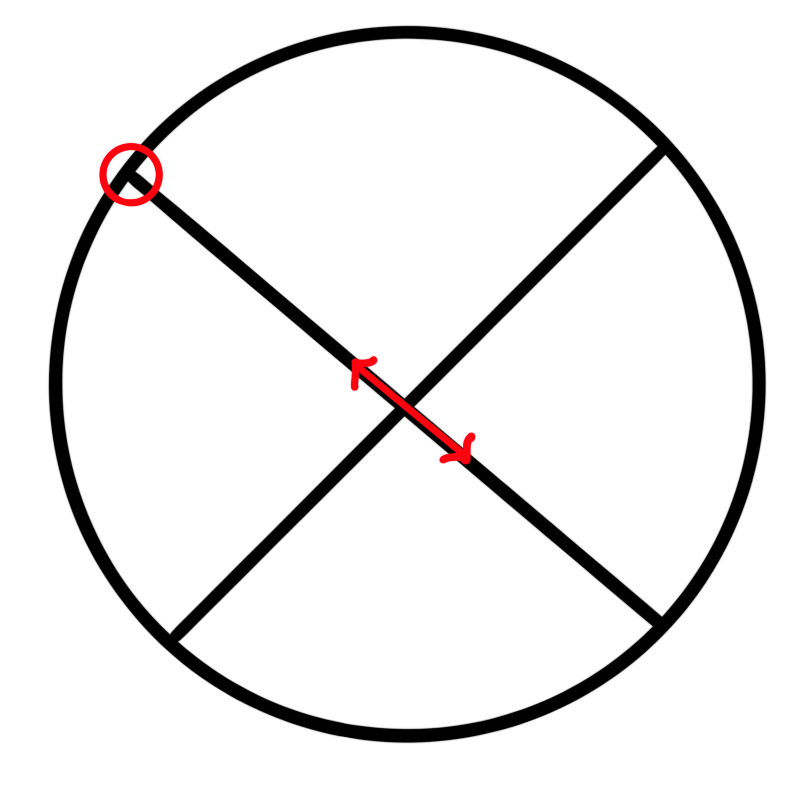}
		\caption{Case $2$.}
		\label{fig:pre3h3}
	\end{subfigure} 
	\caption{$\chi=-2$ and $\chi=-3$.}  
	\label{fig:chi_n2_n3}
\end{figure} 

We enumerate simple Hopf handlebody-knots 
based on $\chi,\ver_4$ and $(\uncovered,\doubled)$. 
%
Recall that $1\hopf_1$ is the only simple Hopf handlebody-knot with $\chi\geq -1$.
For the case $\chi <-1$, the three rules are used repeatedly:   
\begin{enumerate}[label=(\roman*)]
\item For each cell of $\Lambda\subset S_\ast$, at least $3$ sides must be looped. For instance, this implies only edges in Fig.\ \ref{fig:diamond_allowed} can be looped.
\item Any edge incident to two unlooped vertices is unlooped. 
\item At most one edge in Fig.\ \ref{fig:cube} can be looped, up to homeomorphism.
\end{enumerate}
In all figures, the convention is adopted: an unlooped edge (resp.\ vertex) is marked with a crossing (resp.\ circle); see Fig.\ \ref{fig:chi_n2_n3}. 
 
\noindent 
$\bullet$ If $\chi=-2$, then $\Lambda$ is the tetrahedron. Corollary \ref{cor:uncovered_upper_bound} implies $\uncovered=\ver_4=0$, and Lemma \ref{lm:main_formula} implies $\doubled=0$, $\notlooped=1$. Only one such pre-looping system exists, namely Fig.\ \ref{fig:pre2h1}, which gives rise to $2\hopf_1$ 
in Table \ref{tab:hopf_handlebody_knots}.

\noindent
$\bullet$ If $\chi=-3$, then by Corollary \ref{cor:uncovered_upper_bound}, there are two cases: 

\noindent
\textbf{Case $1$: $\ver_4=0, \unloopede\leq 1$.}
In this case, $\Lambda$ is the prism, and $\notlooped=2, \uncovered-\doubled=1$ by Lemma \ref{lm:main_formula}. Therefore, $(\uncovered,\doubled)=(1,0)$. The unlooped edge $e$ is necessarily an edge incident to two rectangles, and there are two possibilities, depending on whether $e$ is IH-allowable; see Figs.\ \ref{fig:pre3h1} and \ref{fig:pre3h2}. In each case, the unlooped vertex/vertices and $e$ determine a unique pre-looping system $\mathcal{E}$. In the former, $V^\mathcal{E}$ is $3\hopf_1$, and in the latter, $V^\mathcal{E}$ is $3\hopf_2$ in Table \ref{tab:hopf_handlebody_knots}.  

\noindent
\textbf{Case $2$: $\ver_4=1,\unloopede=0$.}
In this case, $\Lambda$ is the wheel, and $\notlooped=1, \uncovered-\doubled=0$ by Lemma \ref{lm:main_formula}. In particular, $(\uncovered,\doubled)=(0,0)$. The unlooped vertex being a $3$-valent vertex determines the transversal triplet; see Fig.\ \ref{fig:pre3h3}, They together determine a unique pre-looping system $\mathcal{E}$ with $V^\mathcal{E}$, however, non-connected.  

\noindent
$\bullet$ If $\chi=-4$, then, by Corollary \ref{cor:uncovered_upper_bound}, there are four possibilities:

\noindent 
\textbf{Case $1$: $\ver_4=0, \unloopede\leq 3$.}
In this case, $\Lambda$ is either the cube or the diamond, and by Lemma \ref{lm:main_formula}, $\notlooped=3, \uncovered-\doubled=2$. We further divide it into two cases by $(\uncovered,\doubled)$.

\noindent
\textbf{Case $1.1$: $(\uncovered,\doubled)=(3,1)$.}
Suppose $\Lambda$ is the diamond. Then it may be assumed that the unlooped edges and vertices are those shown in Fig.\ \ref{fig:threeone}. Since a doubled edge is not incident to any unlooped vertex, there are only two possibilities, namely Fig.\ \ref{fig:threeone_double_1} and Fig.\ \ref{fig:threeone_double_2}. Each of them uniquely determines a pre-looping system $\mathcal{E}$. In the former, $V^\mathcal{E}$ is 
$4\hopf_1$ in Table \ref{tab:hopf_handlebody_knots} and in the latter, $V^\mathcal{E}$ is non-connected.  

Suppose $\Lambda$ is the cube. Then it may be assumed the unlooped edges and vertices are those in Fig.\ \ref{fig:threeone_cube}. Let $v_1,v_2,v_3$ be the three unlooped vertices. Then there is a component of $\Lambda-v_1\cup v_2\cup v_3$ consisting of three looped edges with only one looped vertex, an impossibility.
  
\begin{figure}[h]
	\begin{subfigure}{.24\linewidth}
		\centering
		\includegraphics[scale=.1]{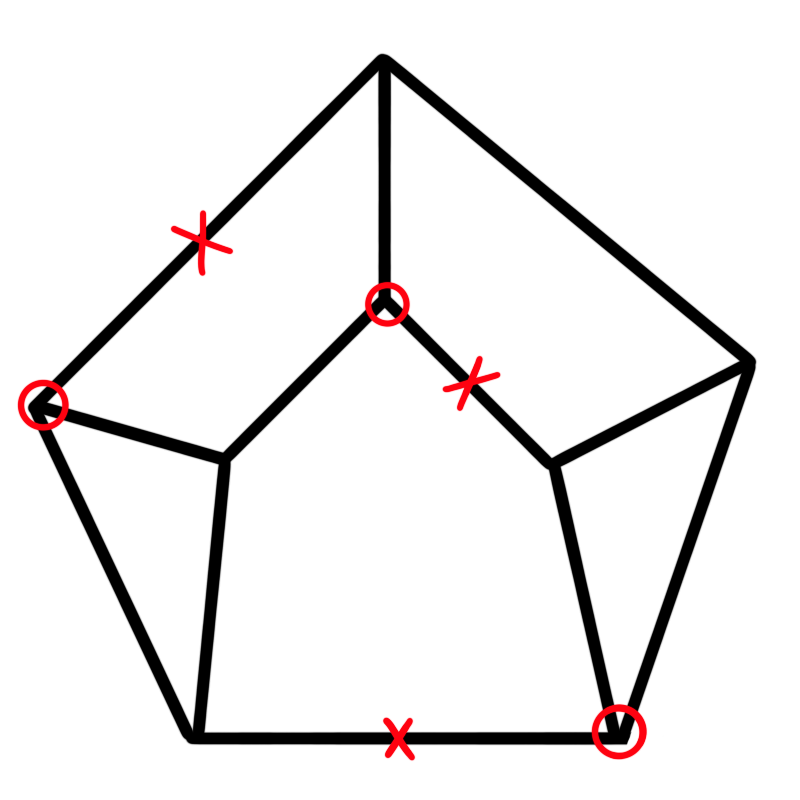}
		\caption{Case $1$: Diamond.}
		\label{fig:threeone}
	\end{subfigure} 
	\begin{subfigure}{.24\linewidth}
		\centering
		\includegraphics[scale=.1]{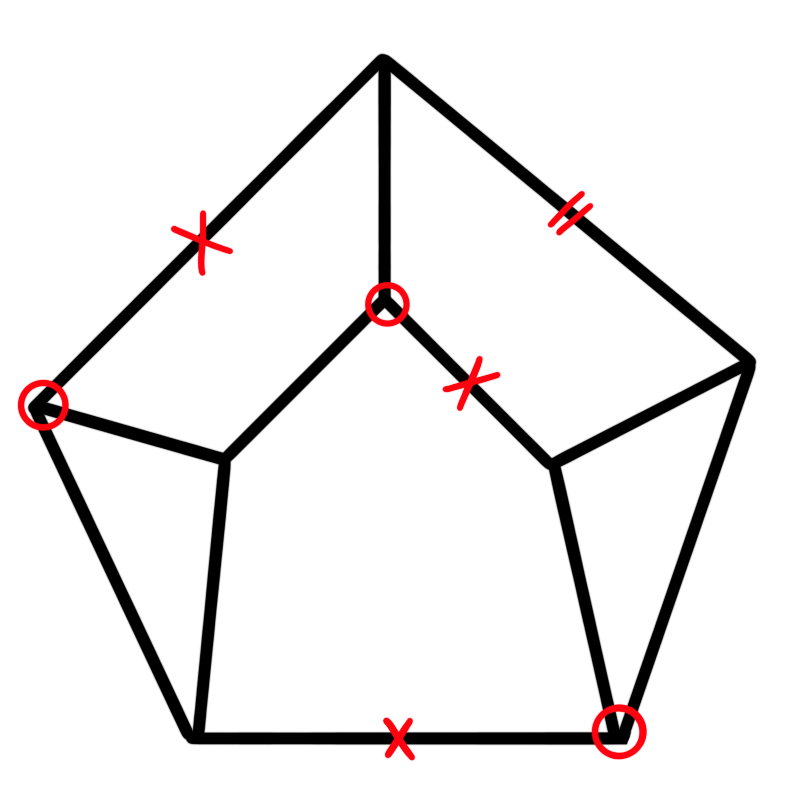}
		\caption{Case $1.1$.}
		\label{fig:threeone_double_1}
	\end{subfigure} 
	\begin{subfigure}{.24\linewidth}
		\centering
		\includegraphics[scale=.1]{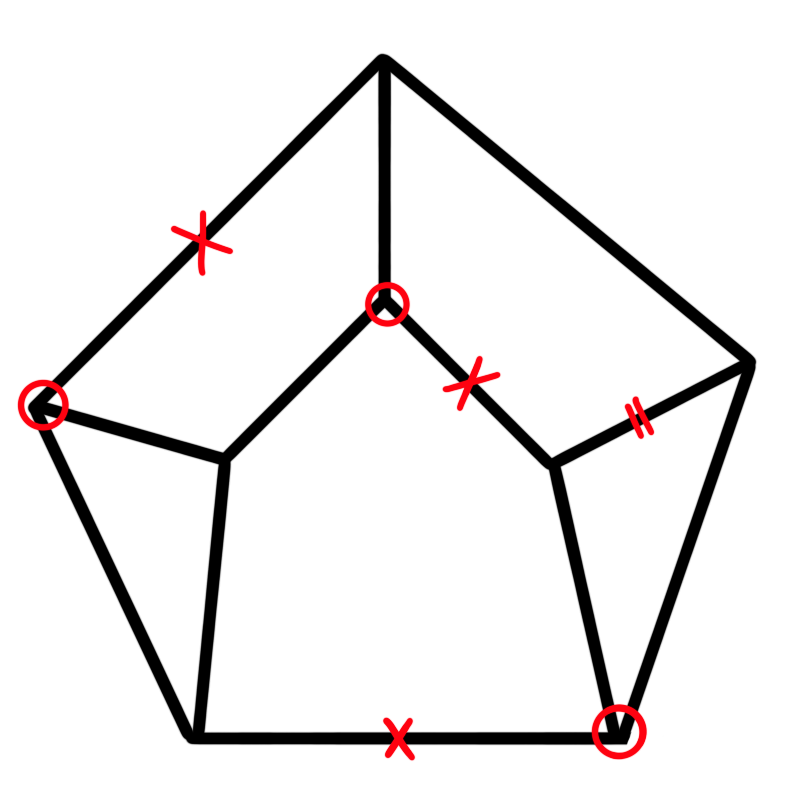}
		\caption{Case $1.1$.}
		\label{fig:threeone_double_2}
	\end{subfigure} 
	\begin{subfigure}{.24\linewidth}
		\centering
		\includegraphics[scale=.09]{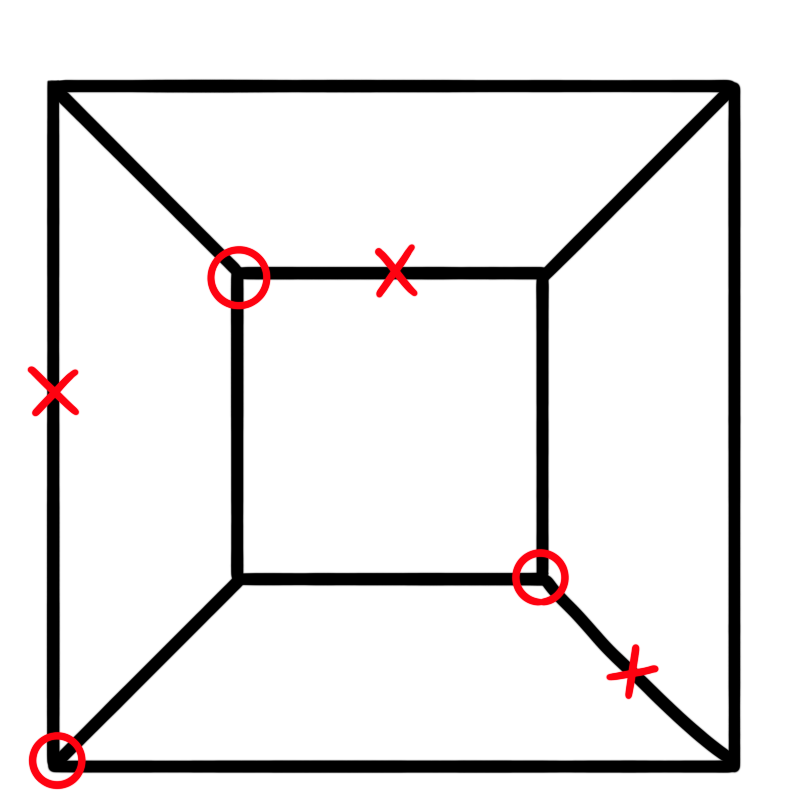}
		\caption{Case $1$: cube.}
		\label{fig:threeone_cube}
	\end{subfigure} 
	\caption{Case $1.1$.}  
	\label{fig: }
\end{figure}

\begin{figure}[b]
	\begin{subfigure}{.24\linewidth}
		\centering
		\includegraphics[scale=.1]{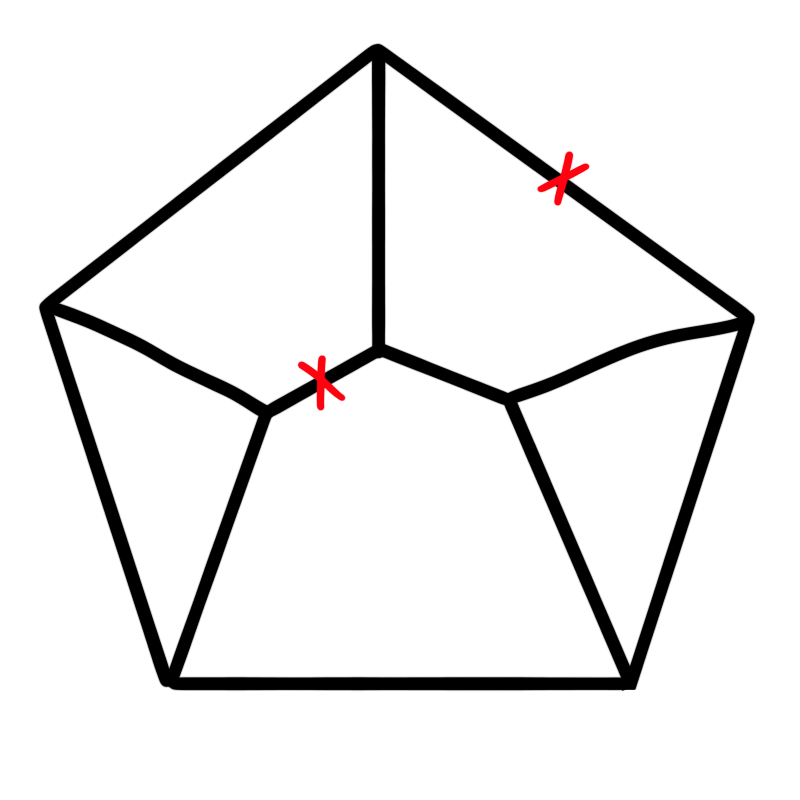}
		\caption{}
		\label{fig:twozero_uncovered_1}
	\end{subfigure} 
	\begin{subfigure}{.24\linewidth}
		\centering
		\includegraphics[scale=.1]{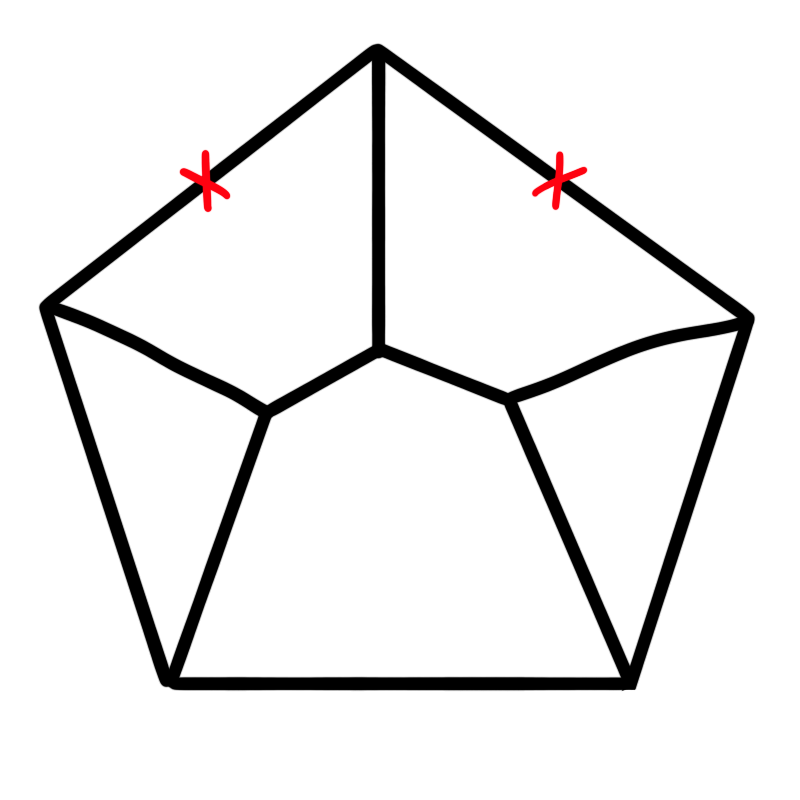}
		\caption{}
		\label{fig:twozero_uncovered_2}
	\end{subfigure} 
	\begin{subfigure}{.24\linewidth}
		\centering
		\includegraphics[scale=.1]{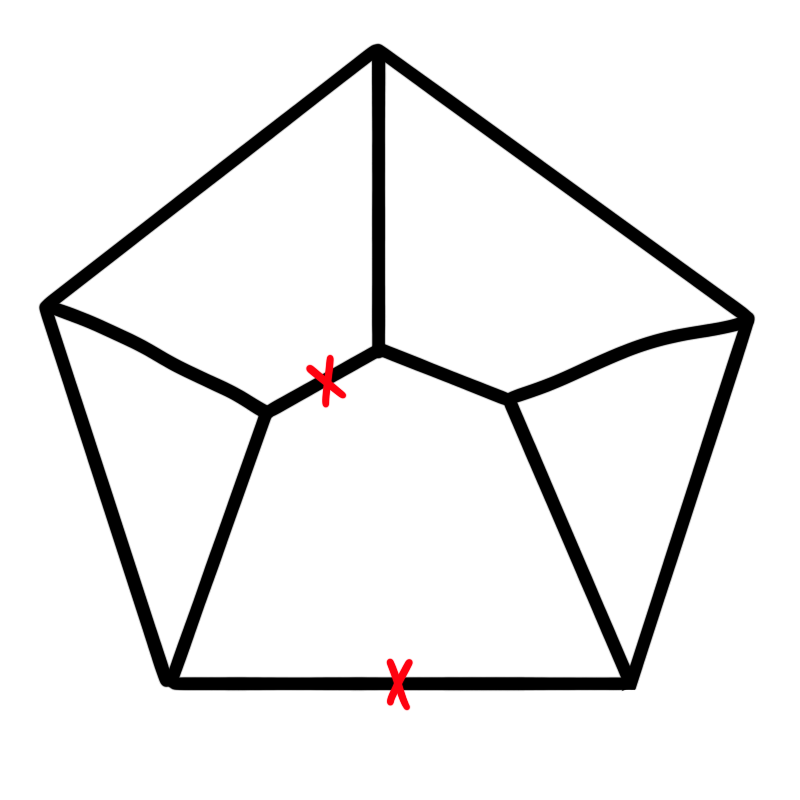}
		\caption{}
		\label{fig:twozero_uncovered_3}
	\end{subfigure} 
	\caption{Case $1.2$.}  
    \label{fig:twozero_uncovered}
\end{figure} 
 
\noindent
\textbf{Case $1.2$: $(\uncovered,\doubled)=(2,0)$.}
If $\Lambda$ is the diamond, then by the allowed edges in Fig.\ \ref{fig:diamond_allowed}, there are three possibilities for the two looped edges as shown in Figs.\ \ref{fig:twozero_uncovered}.

\noindent
$\star$ \textit{Fig.\ \ref{fig:twozero_uncovered_1}:}
this configuration gives rise to two cases: the one with an IH-allowable edge and the one without.
Each case uniquely determines a pre-looping system $\mathcal{E}$ with $V^\mathcal{E}$ being $4\hopf_2$ in Table \ref{tab:hopf_handlebody_knots} in the former, and $V^\mathcal{E}$ being $4\hopf_3$ in the latter.

\noindent
$\star$ \textit{Fig.\ \ref{fig:twozero_uncovered_2}:} this configuration gives rise to three cases:
one having two IH-allowable edges, the other having one, and yet another having none.
In the first and the third case, the unlooped edges and vertices uniquely determine a pre-looping system $\mathcal{E}$ with $V^\mathcal{E}$ corresponding to $4\hopf_4$ and $ 4\hopf_6$, respectively. In the second case, there are two possibilities for the third unlooped vertex, shown in Figs.\ \ref{fig:twozero_2_oneallowable_1} and \ref{fig:twozero_2_oneallowable_2}, each of which uniquely determines a pre-looping system $\mathcal{E}$, yet they differ by an IH-move. Their fat loopings are equivalent to $4\hopf_5$ in Table \ref{tab:hopf_handlebody_knots}.

\noindent
$\star$ \textit{Fig.\ \ref{fig:twozero_uncovered_3}:} 
In this configuration, there is at most one IH-allowable edge, and if it is the one incident to a rectangle and pentagon, then an IH-move will transform it into the first case of Fig.\ \ref{fig:twozero_uncovered_1}. Excluding this case, there are, up to homeomorphism, only three possible configurations for the unlooped vertices and edges: one has the edge incident to the two pentagons being IH-allowable, and the other two have no IH-allowable edge, as shown in Figs.\ \ref{fig:twozero_3_noallowable_1} and \ref{fig:twozero_3_noallowable_2}. Each of the three uniquely determines a pre-looping system, and the associated fat loopings, respectively, correspond to $4\hopf_7, 4\hopf_8, 4\hopf_9$ in Table \ref{tab:hopf_handlebody_knots}.

If $\Lambda$ is the cube, then up to homeomorphism, there is only one possible configuration of unlooped edges, namely the ones in Fig.\ \ref{fig:cube_two}. If any of them is an IH-allowable edge, then an IH-move transforms it into the first case of Fig.\ \ref{fig:twozero_uncovered_3}. If no allowable edges exist, then, up to homeomorphism, there is only one possibility for the unlooped vertices, which uniquely determines a pre-looping system that yields a non-connected handlebody-link.

\begin{figure}[t]
	\begin{subfigure}{.24\linewidth}
		\centering
		\includegraphics[scale=.1]{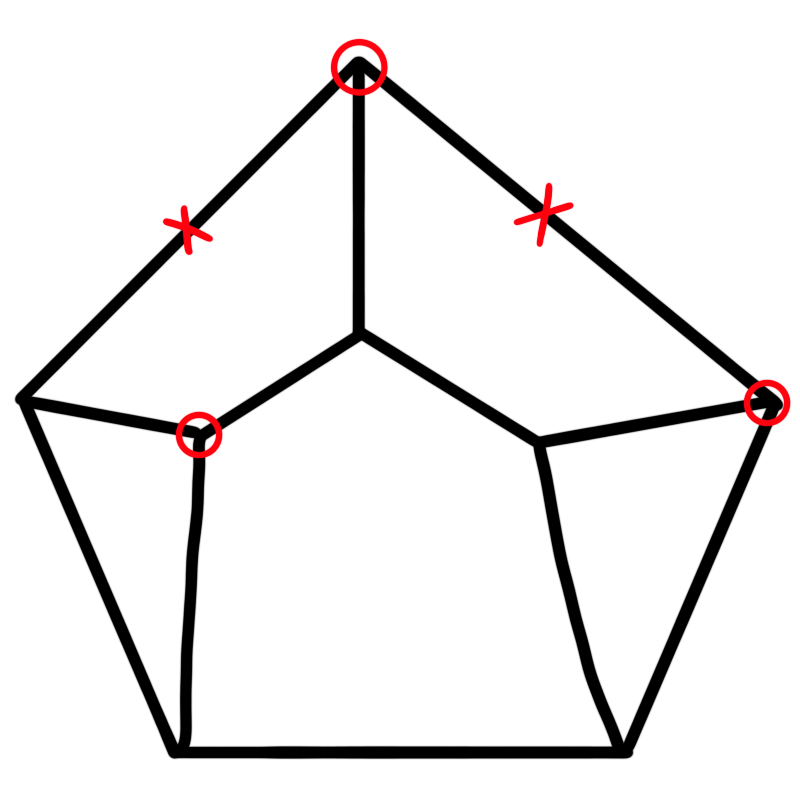}
		\caption{}
		\label{fig:twozero_2_oneallowable_1}
	\end{subfigure} 
	\begin{subfigure}{.24\linewidth}
		\centering
		\includegraphics[scale=.1]{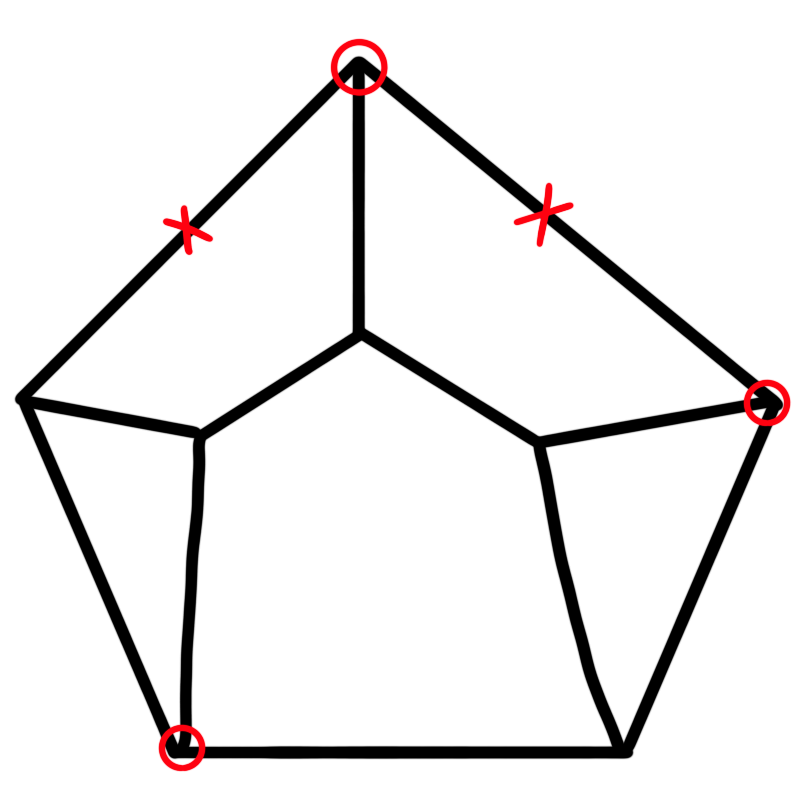}
		\caption{}
		\label{fig:twozero_2_oneallowable_2}
	\end{subfigure} 
	\begin{subfigure}{.24\linewidth}
		\centering
		\includegraphics[scale=.1]{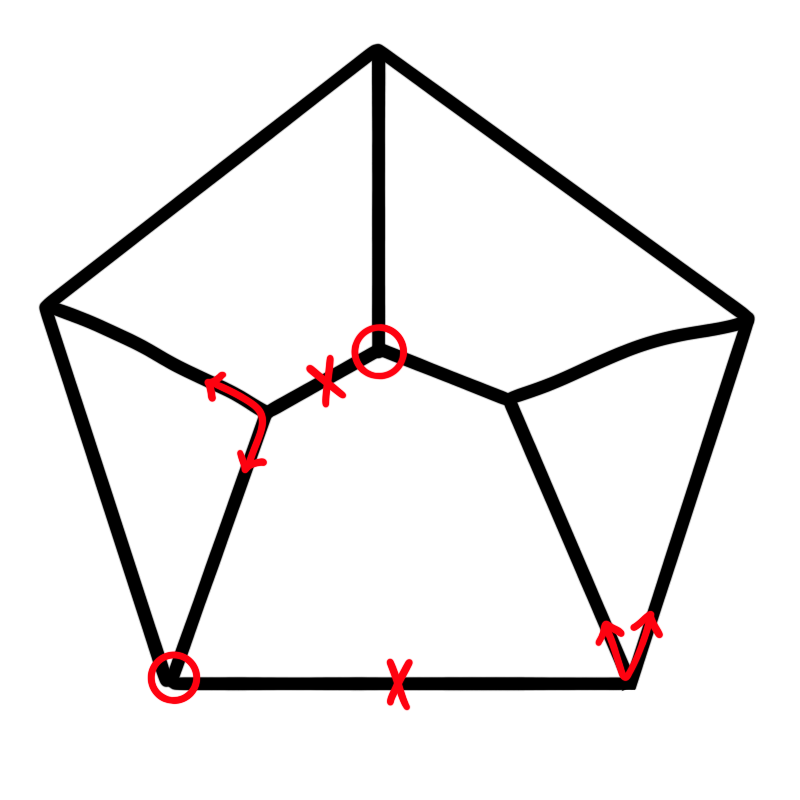}
		\caption{}
		\label{fig:twozero_3_noallowable_1}
	\end{subfigure} 
	\begin{subfigure}{.24\linewidth}
	\centering
	\includegraphics[scale=.1]{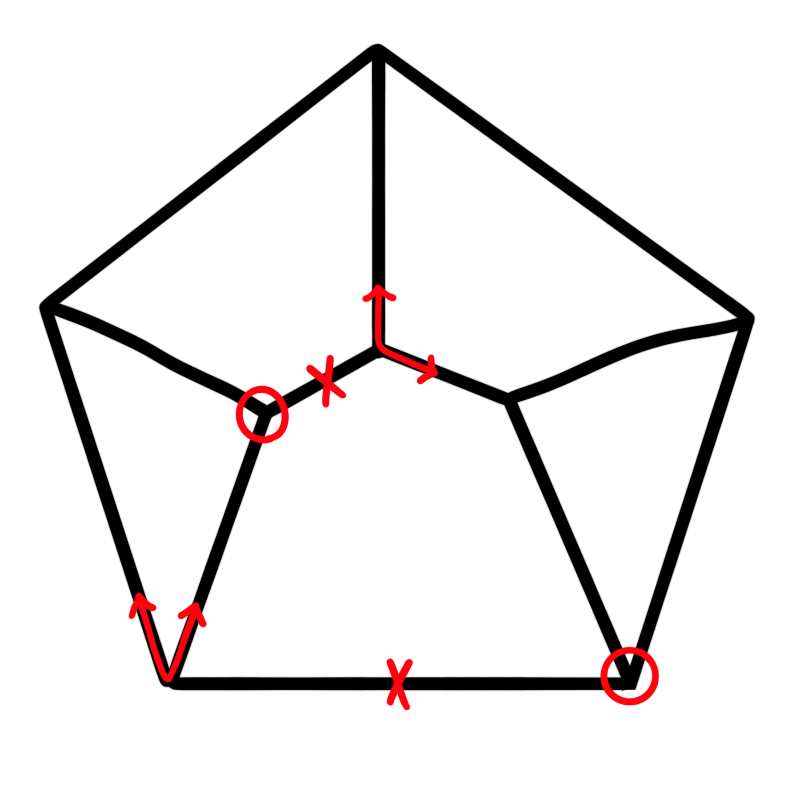}
	\caption{}
	\label{fig:twozero_3_noallowable_2}
\end{subfigure} 
	\caption{Case $1.2$: subcases of Figs.\ \ref{fig:twozero_uncovered_2}, \ref{fig:twozero_uncovered_3}.}  
	\label{}
\end{figure}

\noindent 
\textbf{Case $2$: $\ver_4=1$.}
In this case, $\Lambda$ is either the bug or the heart.
By Lemma \ref{lm:main_formula}, $\notlooped=2$, and $\uncovered-\doubled=1$. 
Since $\uncovered\leq 2$ by Corollary \ref{cor:uncovered_upper_bound}, there are two cases. 

\begin{figure}[b]
	\begin{subfigure}{.24\linewidth}
		\centering
		\includegraphics[scale=.1]{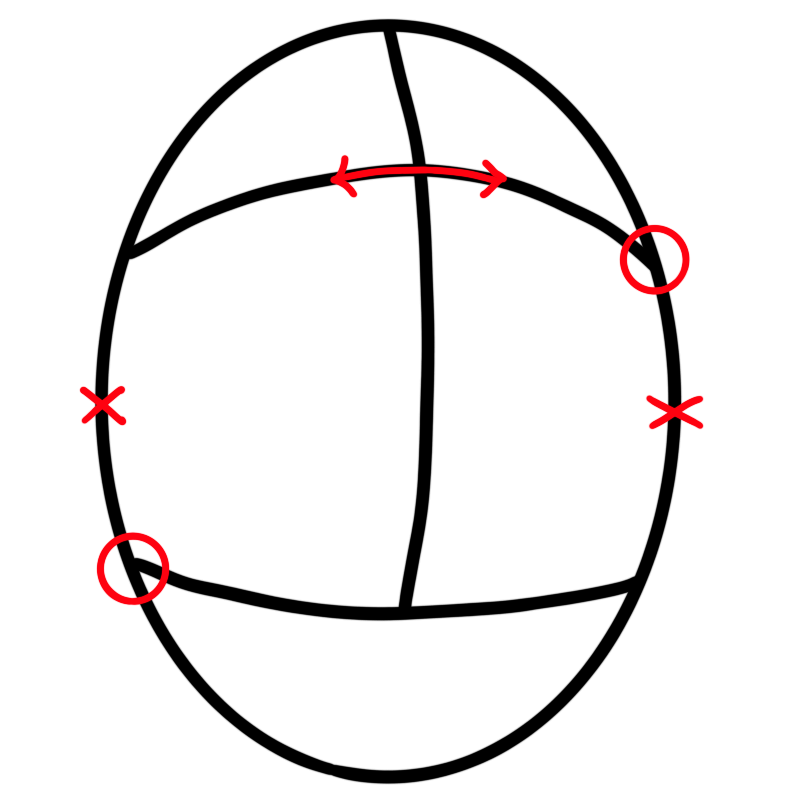}
		\caption{}
		\label{fig:twoone_bug}
	\end{subfigure} 
	\begin{subfigure}{.24\linewidth}
		\centering
		\includegraphics[scale=.1]{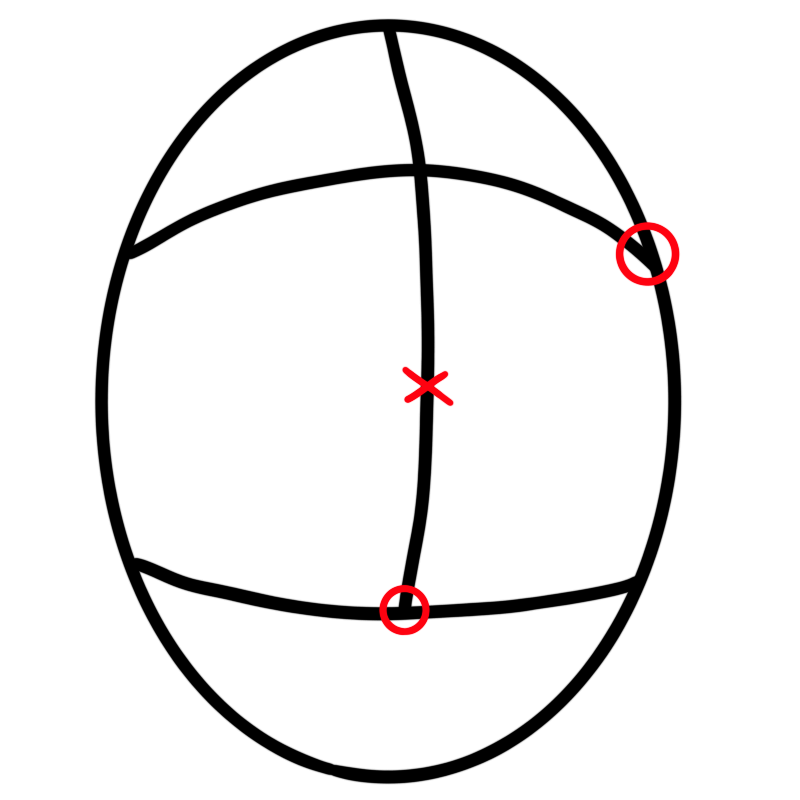}
		\caption{}
		\label{fig:onezero_bug_1}
	\end{subfigure} 
	\begin{subfigure}{.24\linewidth}
		\centering
		\includegraphics[scale=.1]{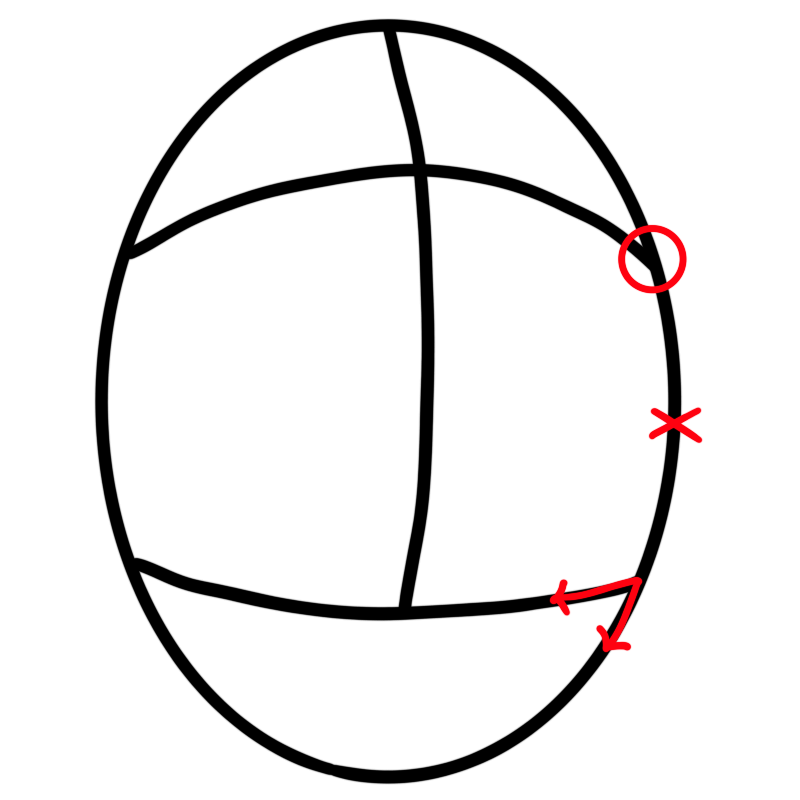}
		\caption{}
		\label{fig:onezero_bug_2}
	\end{subfigure} 
	\begin{subfigure}{.24\linewidth}
		\centering
		\includegraphics[scale=.1]{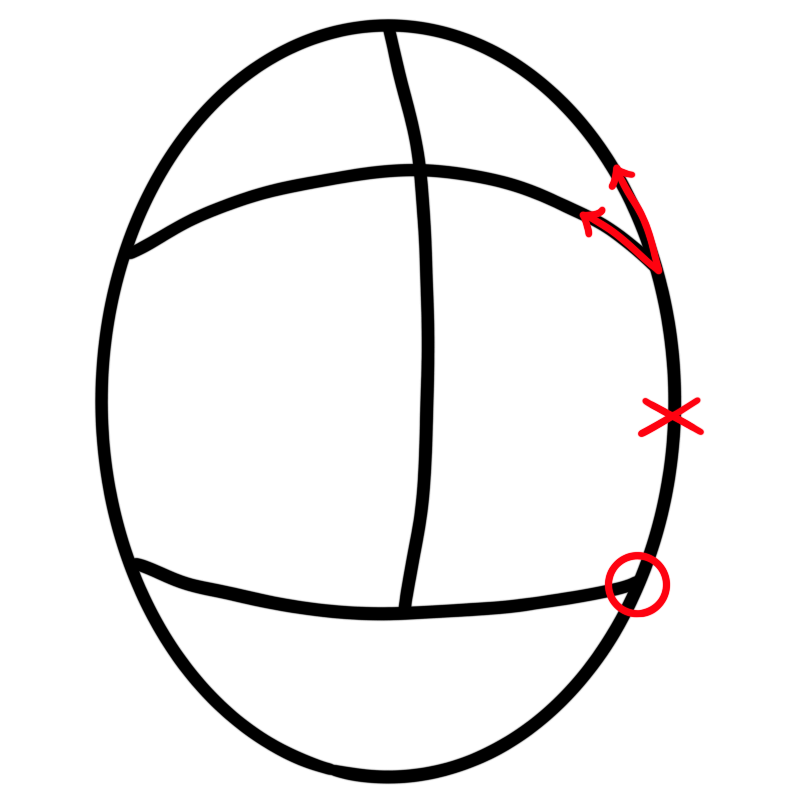}
		\caption{}
		\label{fig:onezero_bug_3}
	\end{subfigure} 
	\caption{Case $2$: bug.}  
\end{figure}

\noindent
\textbf{Case $2.1$: $(\uncovered,\doubled)=(2,1)$.}   
If $\Lambda$ is the bug, then the unlooped edges must be the two incident to a pentagon and a square. Since every unlooped edge has to be incident to one unlooped vertex, at least one unlooped vertex is adjacent to the $4$-valent vertex and hence determines the transversal triplet and the other unlooped vertex; see Fig.\ \ref{fig:twoone_bug}. These in turn determine a unique a transversal pre-looping system $\mathcal{E}$ with $V^\mathcal{E}$ being $4\hopf_{10}$ in Table \ref{tab:hopf_handlebody_knots}. 

If $\Lambda$ is the heart, then by Fig.\ \ref{fig:cube_three}, it may be assumed that the unlooped edges are the ones in Fig.\ \ref{fig:twoone_heart}. This implies one of the unlooped vertex is adjacent to the $4$-valent vertex, and hence it determines the transversal triplet as well as the other unlooped vertex. They together determine a unique transversal pre-looping system $\mathcal{E}$, whose $V^\mathcal{E}$ is, however, non-connected.

\noindent
\textbf{Case $2.2$: $(\uncovered,\doubled)=(1,0)$.} 
If $\Lambda$ is the bug, then the unlooped edge is either the one incident to two squares or the one incident to a pentagon and a square. 
In the former, they uniquely determine the unlooped vertices; see Fig.\ \ref{fig:onezero_bug_1}, and they together determine a unique transversal pre-looping system $\mathcal{E}$ with $V^\mathcal{E}$, however, non-connected.  

In the latter, there are two cases, depending on whether the unlooped edge is IH-allowable. If it is, then it uniquely determines a transversal pre-looping system that yields a non-connected handlebody-link. If not, then there are two possibilities, depending on whether the unlooped vertex incident to the unlooped edge is adjacent to the $4$-valent vertex; see Figs.\ \ref{fig:onezero_bug_2} and \ref{fig:onezero_bug_3}. Both cases uniquely determine a pre-looping system with the former giving rise to $4\hopf_{11}$ and the latter
a non-connected handlebody-link.

If $\Lambda$ is the heart, then, up to homeomorphism, it may be assumed the unlooped edge is the one shown in Fig.\ \ref{fig:onezero_heart}. If the unlooped edge is IH-allowable, then it determines a transversal pre-looping system, which yields a non-connected handlebody-link. If it is not IH-allowable, then there are two possibilities, depending on whether the $4$-valent vertex is incident to the unlooped vertex incident to the unlooped edge; see Figs.\ \ref{fig:onezero_heart_1} and \ref{fig:onezero_heart_2}. Either case uniquely determines a pre-looping system $\mathcal{E}$ with $V^\mathcal{E}$ being non-connected in the former and $4\hopf_{12}$ of Table \ref{tab:hopf_handlebody_knots} in the latter.

\begin{figure}[t]
	\begin{subfigure}{.24\linewidth}
		\centering
		\includegraphics[scale=.1]{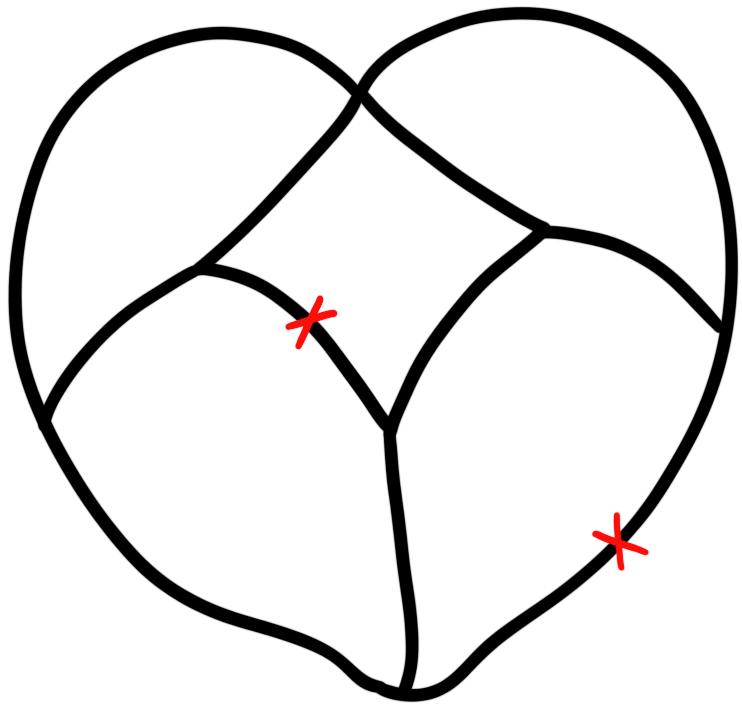}
		\caption{}
		\label{fig:twoone_heart}
	\end{subfigure} 
	\begin{subfigure}{.24\linewidth}
	\centering
	\includegraphics[scale=.1]{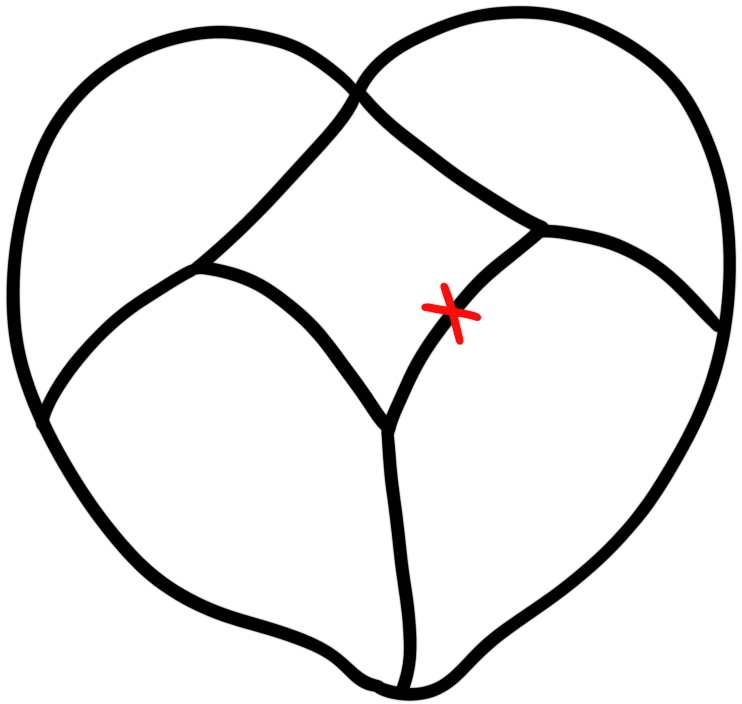}
	\caption{}
	\label{fig:onezero_heart}
	\end{subfigure} 
	\begin{subfigure}{.24\linewidth}
		\centering
		\includegraphics[scale=.1]{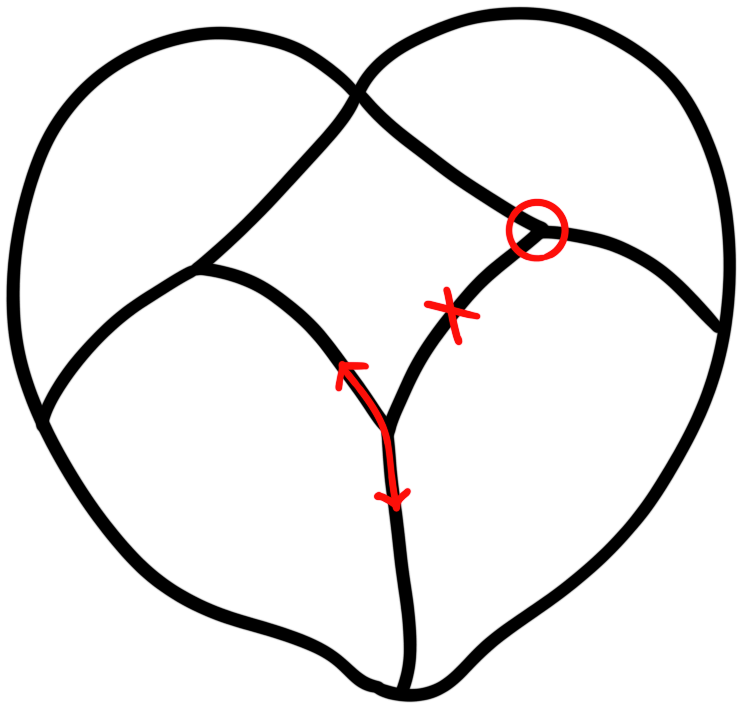}
		\caption{}
		\label{fig:onezero_heart_1}
	\end{subfigure} 
	\begin{subfigure}{.24\linewidth}
		\centering
		\includegraphics[scale=.1]{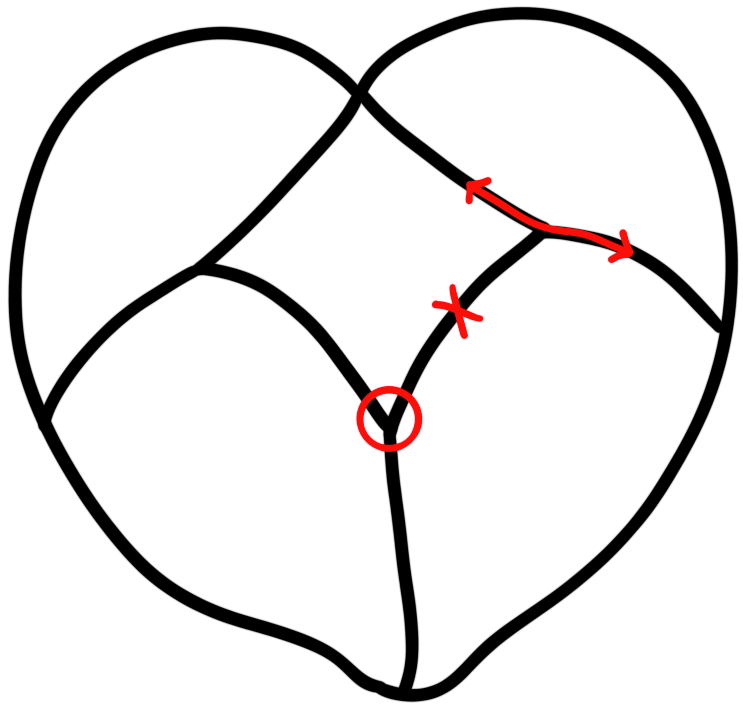}
		\caption{}
		\label{fig:onezero_heart_2}
	\end{subfigure} 
	\caption{Case $2$: heart.}  
\end{figure}

\begin{figure}[b]
	\begin{subfigure}{.24\linewidth}
		\centering
		\includegraphics[scale=.1]{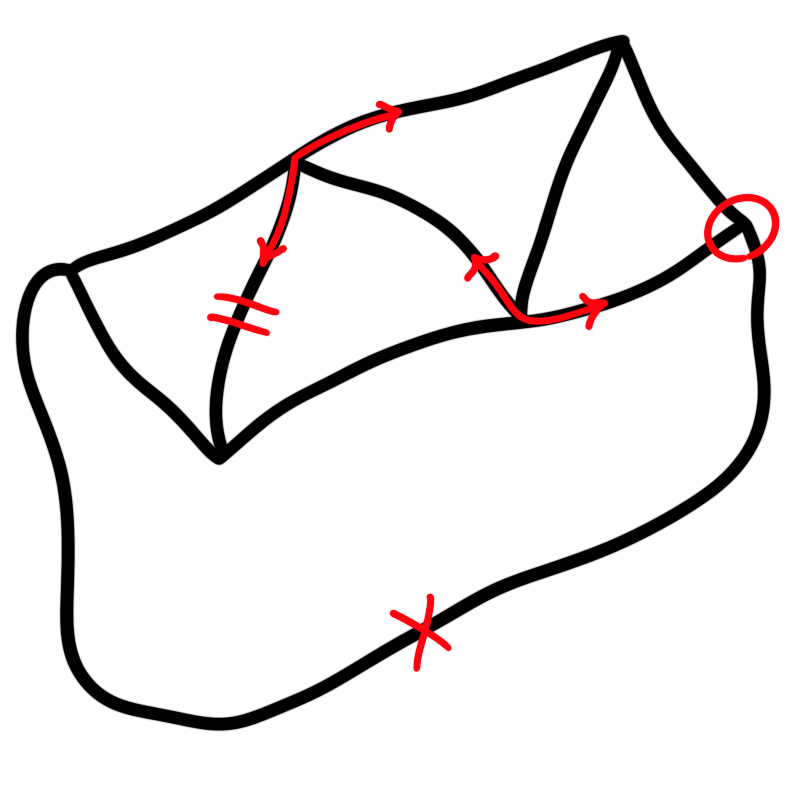}
		\caption{}
		\label{fig:oneone_kite_1}
	\end{subfigure} 
	\begin{subfigure}{.24\linewidth}
		\centering
		\includegraphics[scale=.1]{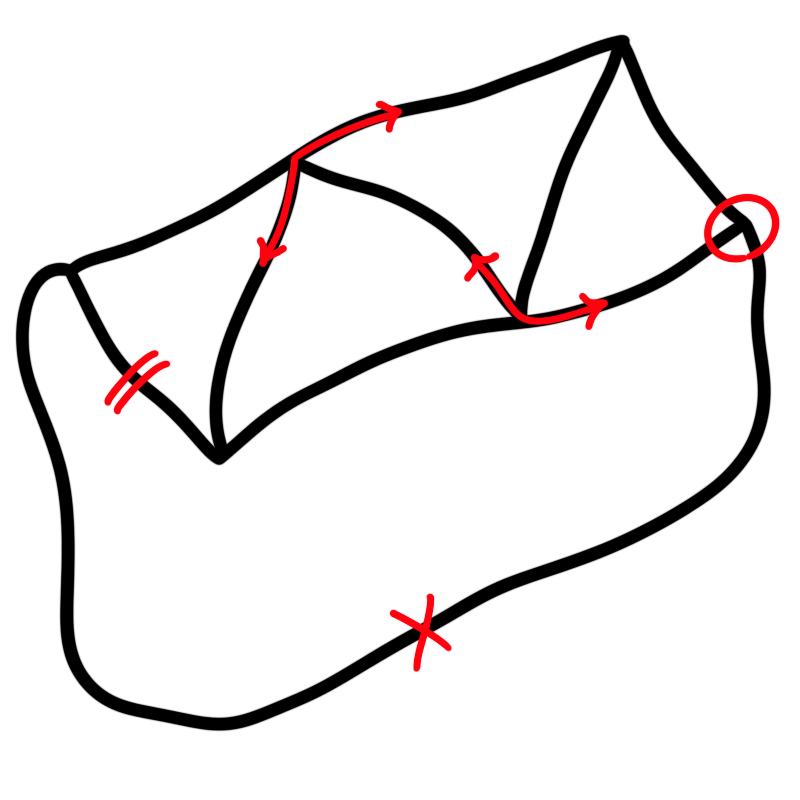}
		\caption{}
		\label{fig:oneone_kite_2}
	\end{subfigure} 
	\begin{subfigure}{.24\linewidth}
		\centering
		\includegraphics[scale=.1]{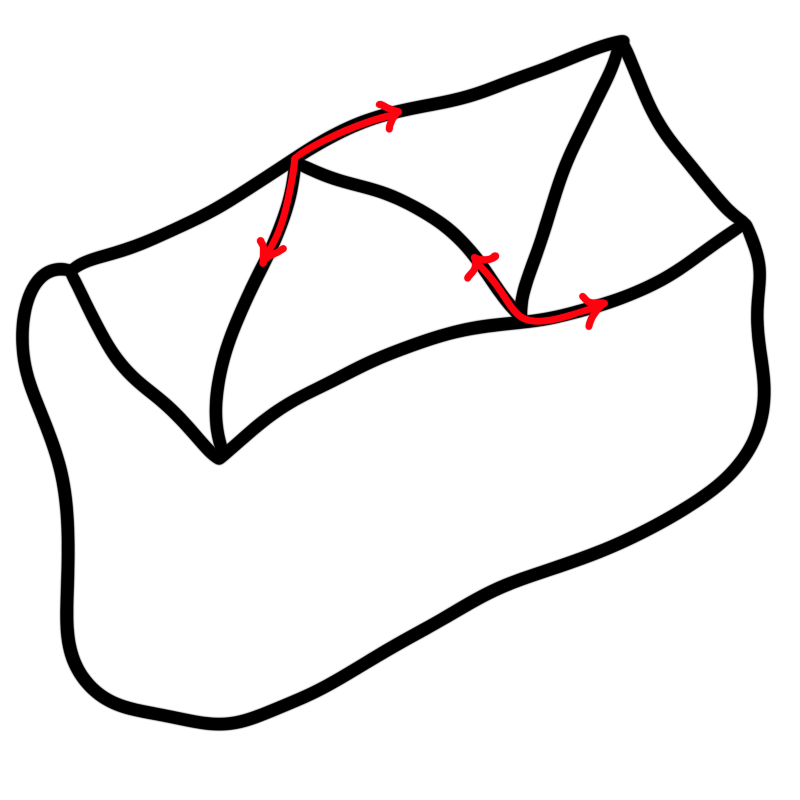}
		\caption{}
		\label{fig:zerozero_kite}
	\end{subfigure} 
	\begin{subfigure}{.24\linewidth}
		\centering
		\includegraphics[scale=.1]{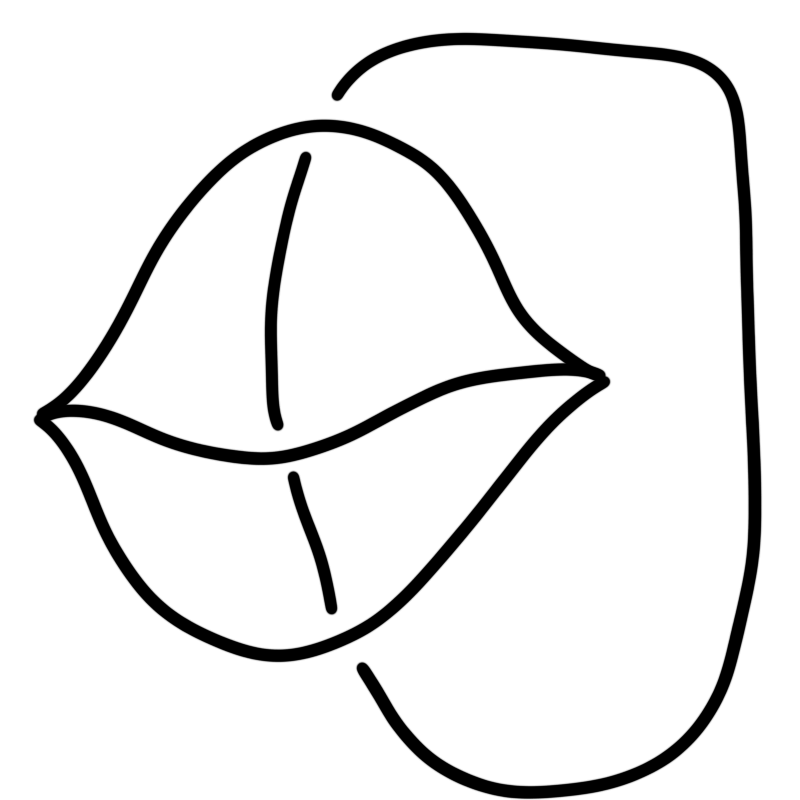}
		\caption{}
		\label{fig:two_components}
	\end{subfigure} 
	\caption{Cases $3$ and $4$.}  
\end{figure}

\noindent
\textbf{Case $3$: $\ver_4=2,\unloopede\leq 1$.}
In this case, $\Lambda$ is the kite, and $\ver_4=1$ and $\unloopede-\doubled=0$ by Lemma \ref{lm:main_formula}, so either 
$(\unloopede,\doubled)=(1,1)$ or $(\unloopede,\doubled)=(0,0)$. 

\noindent
\textbf{Case $3.1$: $(\uncovered,\doubled)=(1,1)$.}  
The unlooped edge must be the one incident to the two squares, and this determines the unlooped vertex, up to homeomorphism. Now, there are two possibilities for the doubled edge, namely Figs.\ \ref{fig:oneone_kite_1} and \ref{fig:oneone_kite_2}. Each determines a unique pre-looping system; the former yields $4\hopf_{13}$ and the latter yields a non-connected handlebody-link.

\noindent
\textbf{Case $3.1$: $(\uncovered,\doubled)=(0,0)$.} 
There is only one possible configuration for the transversal triplets, namely Fig.\ \ref{fig:zerozero_kite}. The transversal triplets determine a unique pre-looping system $\mathcal{E}$ whose $V^\mathcal{E}$ is not connected.

\noindent
\textbf{Case $4$: $\ver_4=3$.}
$\Lambda$ is $K_4$. Since, for any pre-looping system $\mathcal{E}$, $\Sigma_\mathcal{E}$ contains more than one component; see Fig.\ \ref{fig:two_components}, no connected fat looping exists in this case.

The above enumeration shows the table \ref{tab:hopf_handlebody_knots} contains all simple Hopf handlebody-knots with $\chi\geq -4$. The invariants $\chi$, $\Lambda_\mathcal{E}$, $\mathfrak{ind}_c$, $\mathfrak{car}$, and the Euler characteristic $\chi_{\mathfrak{map}}$ of $\mathfrak{map}$ in Table \ref{tab:invariants} shows Table \ref{tab:hopf_handlebody_knots} contains no duplication. Table \ref{tab:invariants} also lists their positive symmetry groups computed via Theorem \ref{teo:classification}\ref{itm:finiteness}.


\begin{table}[h] 
	\centering
	\begin{tabular}{ c|c|c|c|l|r|c} 
		& $\chi$ &  $\Lambda_\mathcal{E}$ & $\mathfrak{ind}_c$   & $\mathfrak{car}$ & $\chi_\mathfrak{map}$ & $\mathrm{sym}_+$\\
		\hline 
		\hline
		$3\hopf_1$ & $-3$ & Wheel & $2$ & (4) & $0$ & $\mathbb Z_2$\\
		$3\hopf_2$ & $-3$ & Wheel & $1$ & (3,1) & $-1$&$\mathbb Z_3$\\
		$4\hopf_1$ & $-4$ & $K_4$ & $1$ & (5) & $-2$&$\mathbb Z_4$\\
		$4\hopf_2$ & $-4$ & Kite & $2$ & (4,1) & $-2$ & $\mathbbm 1$\\
		$4\hopf_3$ & $-4$ & Kite & $1$ & (4,1) & $-2$& $\mathbbm 1$\\
		$4\hopf_4$ & $-4$ & Web & $3$ & (5) & $0$&$\mathbb Z_5$\\
		$4\hopf_5$ & $-4$ & Web & $2$ & (4,0,1) & $-2$& $\mathbbm 1$\\
		$4\hopf_6$ & $-4$ & Web & $1$ & (3,0,2)&$-1$& $\mathbbm 1$\\		
		$4\hopf_7$ & $-4$ & Kite & $1$ & (2,3) & $0$& $\mathbbm 1$\\
		$4\hopf_8$ & $-4$ & Kite & $1$ & (3,2) & $-2$& $\mathbbm 1$\\		
		$4\hopf_9$ & $-4$ & Kite & $2$ & (3,2) & $-2$& $\mathbbm 1$\\
		$4\hopf_{10}$ & $-4$ & $K_4$ & $1$ & (4) & $1$& $\mathbbm 1$\\
		$4\hopf_{11}$ & $-4$ & Kite & $1$ & (2,2) &$0$& $\mathbbm 1$\\
		$4\hopf_{12}$ & $-4$ & Kite & $1$ & (2,2) & $-2$& $\mathbbm 1$\\
		\hline
	\end{tabular}
	\caption{Invariants; see Fig.\ \ref{fig:four_five_valence} for the graph type of $\Lambda_\mathcal{E}$.}
	\label{tab:invariants}
\end{table}


\end{document}